\documentclass[12pt,reqno]{amsart}
\usepackage{dsfont}
\usepackage{bbm}
\usepackage{mathrsfs}
\usepackage{color}
\usepackage{comment}
\usepackage{caption}
\usepackage{enumitem}
\usepackage{diagbox}
\allowdisplaybreaks[4]
\usepackage{amssymb, amsmath}
\usepackage{bm}
\DeclareMathOperator{\diag}{diag}
\newcommand{\one}{\mathbbm{1}}
\DeclareMathOperator{\adj}{adj}

\usepackage{mathtools}
\usepackage{verbatim}
\usepackage{graphicx,amsmath,amsfonts,amssymb,amscd,amsthm}
\usepackage{hyperref}
\usepackage[numbers,sort&compress]{natbib}
\usepackage{enumitem}
\usepackage{xstring}   
\usepackage{etoolbox}
\usepackage{multirow} 
\usepackage[numbers,sort&compress]{natbib}
\usepackage{array}
\usepackage{booktabs}
\usepackage{multirow}
\usepackage{makecell}
\usepackage{mathtools}
\usepackage{abstract}
\usepackage{float}
\usepackage{pifont}
\usepackage{xcolor,listings,fancyhdr}
\setenumerate[1]{itemsep=1pt,partopsep=0pt,parsep=\parskip,topsep=1pt}
\setitemize[1]{itemsep=1pt,partopsep=0pt,parsep=\parskip,topsep=1pt}

\newtheorem{theorem}{Theorem}[section]
\newtheorem{lemma}{Lemma}[section]
\newtheorem{proposition}{Proposition}[section]
\newtheorem{remark}{Remark}[section]

\newtheorem{corollary}{Corollary}

 \newcommand{\Bp}{\begin{proof}}
 \newcommand{\Ep}{\end{proof}}
\hypersetup{colorlinks,linkcolor=blue,citecolor=blue}
\newtheoremstyle{kai}
{3pt} {3pt} {} {} {\bfseries} {.} {.5em} {}
\def\EquationsBySection{\def\theequation
{\thesection.\arabic{equation}}
\@addtoreset{equation}{section}}

\newcommand\old[1]{}
\renewcommand{\theequation}{\thesection.\arabic{equation}}

 \newcommand{\R}{\mathbb{R}}

\newcommand{\Om}{\Omega}
\newcommand{\dd}{\,\mathrm{d}}
\newcommand{\Rea}{\operatorname{Re}}
\newcommand{\Ima}{\operatorname{Im}}

\newcommand{\norm}[1]{\left\lVert#1\right\rVert}

\newcommand{\C}{\mathbb C}
\newcommand{\Xzero}{\mathcal X_0}
\newcommand{\HN}{H^2_N}
\newcommand{\beq}{\begin{equation}}
\newcommand{\eeq}{\end{equation}}
\newfloat{figtab}{htb}{fgtb}
\makeatletter
  \newcommand\figcaption{\def\@captype{figure}\caption}
  \newcommand\tabcaption{\def\@captype{table}\caption}
\makeatother

\numberwithin{equation}{section}
\addcontentsline{toc}{section}{Acknowledgement}

\begin{document}
\raggedbottom

\author{Fang Li}
\address{School of Mathematics, Sun Yat-sen University,
Guangzhou 510275, China}
\email{lifang55@mail.sysu.edu.cn}

\author{Pengfei Song}
\address{School of Mathematics and Statistics, Xi'an Jiaotong University, Xi'an, Shaanxi 710049, China }
\email{song921012@xjtu.edu.cn}

\author{Kexin Zhang}
\address{School of Mathematics, Sun Yat-sen University,
Guangzhou 510275, China}
\email{kxzmath@163.com}

\author{Maolin Zhou}
\address{Chern Institute of Mathematics, Nankai University,
Tianjin 300071, China}
\email{zhouml123@nankai.edu.cn}

\title
{Time-periodic solutions of an SIS epidemic model}
\date{}
\maketitle

 \begin{abstract}
We prove the existence of strictly positive, nonstationary time-periodic solutions for the spatially heterogeneous SIS epidemic model with standard incidence. On a bounded interval with Neumann boundary conditions, we construct smooth, time-independent transmission and recovery rates satisfying \(\beta(x)>\gamma(x)>0\), together with diffusion coefficients \(d_I>d_S>0\), such that periodic solutions coexist with the unique endemic equilibrium for every prescribed positive total population. Thus, pointwise dominance of transmission over recovery does not ensure convergence to equilibrium. Our approach combines an inverse construction of an unstable endemic equilibrium with a coefficient path that preserves both the equilibrium and the pointwise inequalities. A change in the number of unstable eigenvalues along this path yields small-amplitude periodic solutions through a center-manifold reduction and a generalized Hopf bifurcation argument.
The proof uses an AI-assisted construction of an unstable
endemic equilibrium under the pointwise condition
$\beta>\gamma>0$.
\end{abstract}

\section{Introduction}
Reaction–diffusion models are widely used to study the effects of spatial heterogeneity and individual movement on the spread of infectious diseases; see \cite{Bailey1980,BellomoBrezziChaplain2021,BellomoBrezziChaplain2022,BellomoBuriniOutada2022,LouSalako2023,LouSalakoSong2023,Mollison1972} and the references therein. The susceptible–infected–susceptible (SIS) model is one of the basic models in mathematical epidemiology.

Allen, Bolker, Lou and Nevai \cite{Allen} introduced an SIS reaction–diffusion model with standard incidence in a spatially heterogeneous environment. Its one-dimensional form on the interval $(0,\ell)$ is
\begin{equation}\label{eq:modeli}
\begin{cases}
    &S_t=d_SS_{xx}-\beta(x)\frac{SI}{S+I}+\gamma(x)I,
0<x<\ell,\ t>0,\\[5pt]
&I_t=d_II_{xx}+\beta(x)\frac{SI}{S+I}-\gamma(x)I,
0<x<\ell,\ t>0,
\end{cases}
\end{equation}
with homogeneous Neumann boundary conditions
\[
S_x(t,0)=S_x(t,\ell)=I_x(t,0)=I_x(t,\ell)=0.
\]
Here $S(t,x)$ and $I(t,x)$ denote the densities of susceptible and infected individuals, respectively. The positive constants $d_S$ and $d_I$ are their respective diffusion coefficients. The positive functions $\beta(x)$ and $\gamma(x)$ represent the disease transmission and recovery rates, respectively. The constant $\ell$ is the length of the interval. The total population
\[
M=\int_0^\ell \bigl(S(t,x)+I(t,x)\bigr)\,dx
\]
is conserved. 

For a prescribed total population $M>0$, an equilibrium
is a nonnegative, time-independent solution $(S_0,I_0)$.
It satisfies
\begin{equation}\label{eq:steady}
\begin{cases}
\displaystyle
d_S S_0''-\beta(x)\frac{S_0I_0}{S_0+I_0}
+\gamma(x)I_0=0,
&0<x<\ell,\\[5pt]
\displaystyle
d_I I_0''+\beta(x)\frac{S_0I_0}{S_0+I_0}
-\gamma(x)I_0=0,
&0<x<\ell,\\[5pt]
S_0'(0)=S_0'(\ell)=I_0'(0)=I_0'(\ell)=0,
\end{cases}
\end{equation}
together with the mass constraint
\[
\int_0^\ell \bigl(S_0(x)+I_0(x)\bigr)\,dx=M.
\]
The disease-free equilibrium is $(M/\ell,0)$.
An equilibrium with both components strictly positive
on $[0,\ell]$ is called an endemic equilibrium.

Allen et al.~\cite{Allen} thoroughly studied the existence and uniqueness of  the endemic equilibrium. In particular,
they proved that under the   condition
$\int_0^\ell  (\beta(x)-\gamma(x) )dx>0,$
the system admits a unique endemic equilibrium for any
$d_S,d_I>0$ and any prescribed total population $M>0$.
However, existence and uniqueness alone do not determine
the long-term behavior of solutions. The stability of the
endemic equilibrium has therefore attracted further attention.

Peng and Liu \cite{pl} established global stability of the endemic equilibrium in two special cases: equal diffusion coefficients, or $\beta(x)=r\gamma(x)$ for a constant $r>1$. Peng and Zhao \cite{pz} also studied the global dynamics of the model in a time-periodic environment. For unequal diffusion coefficients and spatially nonconstant endemic equilibria, Bai, Du, Li and Zhou \cite{Bai} obtained sufficient conditions for local stability. Under suitable conditions on the coefficients, they also proved local stability in certain small- or large-diffusion regimes.
In particular, if $\beta(x)>\gamma(x)$ on $[0,\ell]$,
there exists $\delta>0$ such that, for every $d_S>0$,
the endemic equilibrium is locally stable whenever
$0<d_I\le\delta$.
The global stability result under the stronger
condition $\beta(x)>2\gamma(x)$ \cite{PENG}
further motivates asking whether
$\beta(x)>\gamma(x)>0$ on $[0,\ell]$ alone
ensures convergence to the endemic equilibrium.
The question of whether the system can admit
nonstationary positive periodic solutions was also
explicitly raised by Bai et al.~\cite{Bai}.

Our main result shows that such periodic solutions
can occur   when $\beta(x)>\gamma(x)>0$
on $[0,\ell]$.

\begin{theorem}\label{thm:main}
For every $d_S>0$, there exist $d_I>d_S$ and smooth
functions $\widehat\beta,\widehat\gamma$ satisfying
\[
\widehat\beta(x)>\widehat\gamma(x)>0,
\qquad x\in[0,\ell],
\]
such that, for every $M>0$, system \eqref{eq:modeli}
has a unique endemic equilibrium and a strictly
positive classical periodic solution that is
nonconstant in time.
\end{theorem}

A difficulty in the proof is to construct
smooth functions satisfying $\beta>\gamma>0$
on $[0,\ell]$ and admitting an endemic equilibrium
whose linearization at fixed total mass has an
eigenvalue with positive real part.
AI-assisted exploration helped us identify such
a construction.
The instability of the resulting equilibrium is
proved analytically in
Section~\ref{seed:construction}.

Starting from this unstable equilibrium, we vary
the coefficients along a path that keeps the
equilibrium fixed and leads to a stable regime.
A center-manifold reduction and a generalized
Hopf theorem then yield nonstationary positive
periodic solutions; see Section~\ref{basic:path}.

As a complement to the main result, we present 
  an explicit analytical
construction based on the shadow system obtained
as $d_I\to+\infty$ in Appendix~\ref{A}.
Combined with the bifurcation argument in
Section~\ref{basic:path}, this construction obtains the
nonstationary positive periodic solutions for
suitable smooth positive functions satisfying
\begin{equation}\label{int-positive}
    \int_0^\ell \bigl(\beta(x)-\gamma(x)\bigr)\,dx>0.
\end{equation}
This provides an alternative route to periodic
solutions through an explicit analytical construction
of the unstable equilibrium, although it does not
ensure the pointwise inequality required in
Theorem~\ref{thm:main}.

Compared to the continuous SIS model, finite-patch SIS models provide a discrete-space counterpart
of the reaction--diffusion model. Allen, Bolker, Lou and
Nevai~\cite{AllenPatch} established the reproduction-number
threshold and uniqueness of endemic equilibria.
Li and Peng~\cite{LiPengPatch} proved global stability
in a special proportional-risk case and studied limiting
profiles under dispersal. Chen, Shi, Shuai and
Wu~\cite{ChenPatch} studied steady states with asymmetric
connections. Gao and Li~\cite{GaoLiPatch} examined the
effects of heterogeneous mobility on disease prevalence.

To compare stability in the continuous and discrete models,
we consider equal-volume patches arranged along a path,
with equal bidirectional weights on each edge.
We assume that the transmission and recovery rates are
positive in every patch, and that the transmission rate
is larger than the recovery rate in each patch.
We conclude: for two or three patches, the endemic equilibrium is locally exponentially stable at fixed total mass for all
positive dispersal rates. Whereas, for four patches, we find
an unstable endemic equilibrium under the same condition.
Thus four is the smallest number of patches allowing
local endemic instability within this path class.

The rest of this paper is organized as follows. In Section \ref{basic:path}, we assume the existence of an unstable positive equilibrium. We then use a bifurcation argument to obtain small-amplitude periodic solutions that are nonconstant in time. Section \ref{seed:construction} presents an AI-assisted construction of an unstable positive equilibrium and an analytical proof of its instability. Together with the results of Section \ref{basic:path}, this completes the proof of Theorem \ref{thm:main}. Section \ref{sec:patch} studies the discrete-space SIS counterpart of the continuous model and compares the stability of endemic equilibria in two-, three-, and four-patch systems. Appendix \ref{A} constructs an unstable positive equilibrium satisfying the  condition \eqref{int-positive}. Appendix \ref{B} provides theoretical proofs for the two- and four-patch cases, a computer-assisted proof for the three-patch case, and the corresponding verification code.

\section{From an unstable equilibrium to periodic solutions}
\label{basic:path}
For simplicity, we use $p=d_S$, $q=d_I$, the endemic equilibrium $S_0(x)=s(x)$, $I_0(x)=h(x)$ in the following.
We consider the SIS reaction--diffusion system
\begin{equation}\label{eq:model}
\begin{cases}
S_t=pS_{xx}-f(x,S,I),\\
I_t=qI_{xx}+f(x,S,I),
\end{cases}
\qquad
f(x,S,I)=\beta(x)\frac{SI}{S+I}-\gamma(x)I,
\end{equation}
on the interval $\Omega=(0,\ell)$, with Neumann boundary conditions
\[
S_x(t,0)=S_x(t,\ell)=I_x(t,0)=I_x(t,\ell)=0.
\]
Here $p,q>0$, and $\beta,\gamma$ are smooth functions satisfying
\[
\int_0^\ell\beta(x)-\gamma(x)dx>0,
\quad x\in[0,\ell].
\]
The total mass
\[
M=\int_0^\ell \bigl(S(t,x)+I(t,x)\bigr)\,dx
\]
is conserved. A positive equilibrium is a time-independent solution
$(s,h)$ with both components strictly positive.

In this Section, we aim to prove the following theorem.
\begin{theorem}
\label{thm:pathp}
Fix $p,q>0$, an interval $(0,\ell)$, smooth coefficients
$\beta>\gamma>0$, and a positive equilibrium $(s,h)$.
Suppose that the linearization at $(s,h)$ has an eigenvalue
$\lambda$ with $\operatorname{Re}\lambda>0$ on the space of
perturbations satisfying
\[
\int_0^\ell(u+v)\,dx=0.
\]
For $\tau\ge1$, define
\begin{equation}\label{eq:path}
\beta_\tau(x)=\tau\beta(x),
\qquad
\gamma_\tau(x)
=\gamma(x)+(\tau-1)\frac{\beta(x)s(x)}{s(x)+h(x)}.
\end{equation}
Then there exist a finite $\tau_1>1$, a parameter
$\tau_*\in(1,\tau_1)$, and a sequence $\tau_j\to\tau_*$ such that
the system with coefficients $(\beta_{\tau_j},\gamma_{\tau_j})$
has strictly positive classical periodic solutions $(S_j,I_j)$
that are nonconstant in time.

Each solution has the same total mass as $(s,h)$, and
\[
\sup_{t\in\mathbb R}
\left(
\|S_j(t,\cdot)-s\|_{L^\infty(0,\ell)}
+
\|I_j(t,\cdot)-h\|_{L^\infty(0,\ell)}
\right)
\longrightarrow0.
\]
After passing to a subsequence, their minimal periods satisfy
\[
T_j\longrightarrow T_*
\qquad\text{for some }T_*\in(0,\infty).
\]
\end{theorem}
\begin{remark}\label{rmk}
    If we change the condition $\beta>\gamma$ into $\int_0^l \beta-\gamma dx>0$, results in Theorem \ref{thm:pathp}  still holds.
\end{remark}

\subsection{The linearized problem  and the parameter path }

Let $p,q>0$. Assume that $\beta$ and $\gamma$ are smooth and satisfy
\[
\beta(x)>\gamma(x)>0
\qquad\text{for all }x\in[0,l].
\]
We study the following reaction--diffusion system on $(0,l)$:
\begin{equation}\label{aug:original}
\begin{cases}
\partial_t S(t,x)=p\partial_{xx}S(t,x)
-f_\tau(x,S(t,x),I(t,x)),&t>0,\quad 0<x<l,\\
\partial_t I(t,x)=q\partial_{xx}I(t,x)
+f_\tau(x,S(t,x),I(t,x)),&t>0,\quad 0<x<l,\\
\partial_x S(t,0)=\partial_x S(t,l)
=\partial_x I(t,0)=\partial_x I(t,l)=0,&t>0,
\end{cases}
\end{equation}
where
\[
f_\tau(x,S,I)=\beta_\tau(x)\frac{SI}{S+I}-\gamma_\tau(x)I.
\]
At $\tau=1$, we take $\beta_1=\beta$ and $\gamma_1=\gamma$;
the coefficient path for $\tau\ge1$ is defined below.
We restrict attention to solutions with a fixed total mass,
\[
\int_0^l\bigl(S(t,x)+I(t,x)\bigr)\,dx=M.
\]
Let $(s(x),h(x))$ be a smooth positive equilibrium at $\tau=1$.
Thus
\begin{equation}\label{basic:steady}
\begin{cases}
ps''(x)=\displaystyle\beta(x)\frac{s(x)h(x)}{s(x)+h(x)}-\gamma(x)h(x),
&0<x<l,\\[4pt]
qh''(x)=\displaystyle-\beta(x)\frac{s(x)h(x)}{s(x)+h(x)}+\gamma(x)h(x),
&0<x<l,\\[4pt]
s'(0)=s'(l)=h'(0)=h'(l)=0.
\end{cases}
\end{equation}
Set
\[
M=\int_0^l(s+h)\,dx,\quad
a(x)=\frac{\beta h^2}{(s+h)^2}>0.
\]

Adding the two equations in \eqref{basic:steady} gives
\[
(ps(x)+qh(x))''=0.
\]
Hence $(ps(x)+qh(x))'$ is constant. The Neumann conditions show that
this constant is zero. Integrating once more gives
\begin{equation}\label{basic:k}
ps(x)+qh(x)=k
\quad\text{for all }x\in[0,l],
\end{equation}
where $k>0$ because $p,q,s(x),h(x)>0$.

Throughout the paper, primes on functions of $x$ denote spatial
derivatives. We use $(u(t,x),v(t,x))$ for perturbations and
$(\phi(x),\psi(x))$ for eigenfunctions.

We consider spaces
\begin{equation}\label{basic:spaces}
\begin{aligned}
X&:=\left\{U=(u,v)\in H^1(0,l;\mathbb R^2):
\int_0^l(u(x)+v(x))\,dx=0\right\},\\
Z&:=\left\{U=(u,v)\in H^3(0,l;\mathbb R^2)\cap X:
u'(0)=u'(l)=v'(0)=v'(l)=0\right\}.
\end{aligned}
\end{equation}
The spaces $X$ and $Z$ carry the usual $H^1$ and $H^3$ norms,
respectively. In particular, $Z$ is densely and compactly embedded in $X$.
For spectral analysis, write $X_{\mathbb C}$ and $Z_{\mathbb C}$
for their complexifications, obtained by replacing $\mathbb R^2$
with $\mathbb C^2$ in \eqref{basic:spaces}. The inner product is
\[
\langle U,V\rangle_{X_{\mathbb C}}
=\int_0^l\bigl(U(x)\cdot\overline{V(x)}
+U'(x)\cdot\overline{V'(x)}\bigr)\,dx.
\]
For normed spaces $Z$ and $X$, $\mathcal L(Z,X)$ denotes the
space of bounded linear maps from $Z$ into $X$, and
$\mathcal L(X):=\mathcal L(X,X)$.

The diffusion operator is
\[
D:Z\subset X\longrightarrow X,\quad
(DU)(x)=\binom{pu''(x)}{qv''(x)}.
\]
Indeed, $U''\in H^1$ for $U\in Z$, and the Neumann conditions give
$\int_0^l((DU)_1+(DU)_2)\,dx=0$. Hence $DU\in X$.

We use the same notation for its complex-linear extension
$D:Z_{\mathbb C}\subset X_{\mathbb C}\to X_{\mathbb C}$.
Finite Neumann cosine sums satisfying the mass constraint are
dense in $X_{\mathbb C}$. They also give an orthogonal eigenbasis for $D$
in the above $H^1$ inner product.
Moreover,
\begin{equation}\label{basic:graph-estimate}
\|U\|_{H^3}\le C\bigl(\|DU\|_{H^1}+\|U\|_{H^1}\bigr),
\qquad U\in Z_{\mathbb C}.
\end{equation}

\begin{lemma}
\label{lem:fixed-equilibrium-path}
For each $\tau\ge1$, define
\begin{equation}\label{basic:coefficients}
\beta_\tau=\tau\beta,
\qquad
\gamma_\tau=\gamma+(\tau-1)\frac{\beta s}{s+h}.
\end{equation}
Write
\[
f_\tau(x,S,I)=\beta_\tau\frac{SI}{S+I}-\gamma_\tau I.
\]
Then $\beta_\tau$ and $\gamma_\tau$ are smooth, and
\[
\beta_\tau(x)>\gamma_\tau(x)>0
\qquad\text{for all }x\in[0,l].
\]
Moreover, $(s,h)$ remains a positive equilibrium for these
coefficients, with the same total mass $M$.
\end{lemma}

\begin{proof}
Since $s+h>0$, the coefficients in \eqref{basic:coefficients}
are smooth.
At $(S,I)=(s,h)$, we have
\begin{align}
f_\tau(x,s,h)
=\tau\beta\frac{sh}{s+h}
-\left[\gamma+(\tau-1)\frac{\beta s}{s+h}\right]h
=\beta\frac{sh}{s+h}-\gamma h.
\label{basic:sameequilibrium}
\end{align}
Thus the steady equations and the boundary conditions for $(s,h)$
remain unchanged. Its total mass is still $M$.

Also, for every $x\in[0,l]$,
\begin{equation}\label{basic:positivecoefficients}
\begin{aligned}
\gamma_\tau&\ge\gamma>0,\\
\beta_\tau-\gamma_\tau
&=(\beta-\gamma)+(\tau-1)\frac{\beta h}{s+h}>0.
\end{aligned}
\end{equation}
The last inequality uses $\beta>\gamma$ and $\tau\ge1$.
\end{proof}

\subsection{The linearized equations and the eigenvalue problem}

Write
\begin{equation}\label{si}
S(t,x)=s(x)+u(t,x),\qquad I(t,x)=h(x)+v(t,x).
\end{equation}
The second equation in \eqref{basic:steady} gives
\[
\gamma=\frac{\beta s}{s+h}+q\frac{h''}{h}.
\]
It follows from \eqref{basic:coefficients} that
\[
\gamma_\tau=\frac{\tau\beta s}{s+h}+q\frac{h''}{h}.
\]
Taking the partial derivatives of $f_\tau$ at $(s,h)$, we obtain
\begin{align}
a_\tau&:=(f_\tau)_S(x,s,h)
=\frac{\tau\beta h^2}{(s+h)^2}=\tau a,
\notag\\
c_\tau&:=(f_\tau)_I(x,s,h)
=\frac{\tau\beta s^2}{(s+h)^2}-\gamma_\tau
=-\tau a\frac{s}{h}-q\frac{h''}{h}.
\label{basic:derivatives}
\end{align}
Keeping only the terms that are linear in $(u,v)$ gives
\begin{equation}\label{basic:linear}
\begin{cases}
\partial_tu(t,x)=p\partial_{xx}u(t,x)
-\tau a(x)u(t,x)-c_\tau(x)v(t,x),&t>0,\quad 0<x<l,\\
\partial_tv(t,x)=q\partial_{xx}v(t,x)
+\tau a(x)u(t,x)+c_\tau(x)v(t,x),&t>0,\quad 0<x<l,\\
\partial_xu(t,0)=\partial_xu(t,l)
=\partial_xv(t,0)=\partial_xv(t,l)=0,&t>0.
\end{cases}
\end{equation}
Adding the equations and integrating in $x$, we obtain
\[
\frac{d}{dt}\int_0^l(u(t,x)+v(t,x))\,dx
=0.
\]
For solutions with the same mass as $(s,h)$, this gives
\begin{equation}\label{basic:zeromass}
\int_0^l(u(t,x)+v(t,x))\,dx=0.
\end{equation}

To study the eigenvalue problem, set
\[
(u(t,x),v(t,x))=e^{\lambda t}(\phi(x),\psi(x)),
\]
where $\lambda,\phi,\psi$ may be complex. Substitution gives
\begin{equation}\label{basic:mode}
\begin{cases}
\lambda\phi(x)=p\phi''(x)-\tau a(x)\phi(x)-c_\tau(x)\psi(x),
&0<x<l,\\
\lambda\psi(x)=q\psi''(x)+\tau a(x)\phi(x)+c_\tau(x)\psi(x),
&0<x<l,\\
\phi'(0)=\phi'(l)=\psi'(0)=\psi'(l)=0,\\[2pt]
\displaystyle\int_0^l(\phi(x)+\psi(x))\,dx=0,
\end{cases}
\qquad (\phi,\psi)\not\equiv(0,0).
\end{equation}

The matrix of the linear terms in \eqref{basic:mode} is
\[
J_\tau(x)
:=
\begin{pmatrix}
-a_\tau(x)&-c_\tau(x)\\
a_\tau(x)&c_\tau(x)
\end{pmatrix}.
\]
Using the expressions for $a_\tau$ and $c_\tau$, we obtain
\[
J_\tau(x)
=
\tau
\begin{pmatrix}
-a(x)&\displaystyle a(x)\frac{s(x)}{h(x)}\\[2mm]
a(x)&\displaystyle-a(x)\frac{s(x)}{h(x)}
\end{pmatrix}
+
\begin{pmatrix}
0&\displaystyle q\frac{h''(x)}{h(x)}\\[2mm]
0&\displaystyle-q\frac{h''(x)}{h(x)}
\end{pmatrix}.
\]
Both matrices on the right-hand side are independent of $\tau$.
Thus $J_\tau$ depends affinely on $\tau$.
Define
\[
(J_\tau U)(x):=J_\tau(x)U(x).
\]

The smooth entries of $J_\tau$ define a bounded multiplication
operator on $X_{\mathbb C}$. Its two components sum to zero, so it
preserves the mass constraint. For each $\tau\ge1$, define
\[
L_\tau:Z_{\mathbb C}\subset X_{\mathbb C}\longrightarrow X_{\mathbb C},
\qquad L_\tau=D+J_\tau.
\]
For $\Phi=(\phi,\psi)\in Z_{\mathbb C}$,
\[
(L_\tau\Phi)(x)=
\binom{p\phi''(x)-\tau a(x)\phi(x)-c_\tau(x)\psi(x)}
      {q\psi''(x)+\tau a(x)\phi(x)+c_\tau(x)\psi(x)}.
\]
Its common domain is $Z_{\mathbb C}$, independent of $\tau$.
Thus \eqref{basic:mode} is equivalent to
\[
L_\tau\Phi=\lambda\Phi,\qquad
\Phi\in Z_{\mathbb C}\setminus\{0\}.
\]
The restriction to real-valued functions is denoted by the same symbol,
$L_\tau:Z\subset X\to X$.

\subsection{An energy identity and stability for large $\tau$}

We now use the time-dependent linear system \eqref{basic:linear}.
The equilibrium functions $s(x),h(x)$ are real-valued,
strictly positive, and independent of time.
We will use the identities
\[
ps''(x)=-qh''(x),\qquad
\tau a(x)s(x)+c_\tau(x)h(x)=-qh''(x).
\]

We first prove an energy identity for \eqref{basic:linear}.
We then apply it to a solution of the form
\[
u(t,x)=e^{\lambda t}\phi(x),\qquad
v(t,x)=e^{\lambda t}\psi(x),
\]
to estimate the real part of an eigenvalue $\lambda$.

Define
\[
Z(t,x)=\frac{u(t,x)}{s(x)},\qquad
W(t,x)=\frac{v(t,x)}{h(x)},
\]
and
\[
E(t)=\frac12\int_0^l
\left(s(x)|Z(t,x)|^2+h(x)|W(t,x)|^2\right)\,\mathrm dx.
\]
Equivalently,
\[
E(t)=\frac12\int_0^l
\left(\frac{|u(t,x)|^2}{s(x)}
+\frac{|v(t,x)|^2}{h(x)}\right)\,\mathrm dx.
\]

\begin{lemma}\label{lem:energy-identity}
Let $(u(t,x),v(t,x))$ be a classical solution of
\eqref{basic:linear}, possibly complex-valued.
Then
\begin{align}\label{basic:energyidentity}
E'(t)=&-p\int_0^l s(x)|\partial_x Z(t,x)|^2\,\mathrm dx
-q\int_0^l h(x)|\partial_x W(t,x)|^2\,\mathrm dx\nonumber\\
&-\int_0^l
\left(\tau a(x)s(x)+\frac q2h''(x)\right)
|Z(t,x)-W(t,x)|^2\,\mathrm dx.
\end{align}
\end{lemma}
\begin{proof}
Since $u=sZ$, $v=hW$, and $s,h$ do not depend on time,
we have $\partial_t u=s\,\partial_t Z$ and
$\partial_t v=h\,\partial_t W$.
We use the following equalities:
\begin{align*}
&ps''(x)=-qh''(x),
\quad \tau a(x)s(x)+c_\tau(x)h(x)=-qh''(x),\quad
\operatorname{Re}(\overline{f}\,f')=\frac12(|f|^2)',\\
&|Z(t,x)-W(t,x)|^2
=|Z(t,x)|^2+|W(t,x)|^2
-2\operatorname{Re}
\bigl(W(t,x)\overline{Z(t,x)}\bigr).
\end{align*}
Here $f$ is any differentiable complex-valued function of $x$.
The Neumann conditions give
\[
\partial_x u(t,x)=\partial_x v(t,x)=s'(x)=h'(x)=0
\qquad\text{at }x=0,l,
\]
so all boundary terms below vanish.

Differentiating $E(t)$ and using \eqref{basic:linear},
we obtain
\begin{align*}
E'(t)
&=\operatorname{Re}\int_0^l
\left(s(x)\overline{Z(t,x)}\,\partial_t Z(t,x)
+h(x)\overline{W(t,x)}\,\partial_t W(t,x)\right)\,\mathrm dx\\
&=\operatorname{Re}\int_0^l
\left(\overline{Z(t,x)}\,\partial_t u(t,x)
+\overline{W(t,x)}\,\partial_t v(t,x)\right)\,\mathrm dx\\
&=p\operatorname{Re}\int_0^l
\overline{Z(t,x)}\,\partial_x^2(s(x)Z(t,x))\,\mathrm dx
+q\operatorname{Re}\int_0^l
\overline{W(t,x)}\,\partial_x^2(h(x)W(t,x))\,\mathrm dx\\
&\quad+\operatorname{Re}\int_0^l
\bigl(\overline{W(t,x)}-\overline{Z(t,x)}\bigr)
\bigl(\tau a(x)u(t,x)+c_\tau(x)v(t,x)\bigr)\,\mathrm dx.
\end{align*}
We compute the two diffusion terms and the reaction term separately.

For the $Z$ term, integrating by parts twice gives
\begin{align}
&p\operatorname{Re}\int_0^l
\overline{Z(t,x)}\,\partial_x^2(s(x)Z(t,x))\,\mathrm dx\nonumber\\
=&-p\operatorname{Re}\int_0^l
\overline{\partial_x Z(t,x)}
\bigl(s'(x)Z(t,x)+s(x)\partial_x Z(t,x)\bigr)\,\mathrm dx
\nonumber\\
=&-p\int_0^l s(x)|\partial_x Z(t,x)|^2\,\mathrm dx
-\frac p2\int_0^l
s'(x)\partial_x|Z(t,x)|^2\,\mathrm dx
\nonumber\\
=&-p\int_0^l s(x)|\partial_x Z(t,x)|^2\,\mathrm dx
+\frac p2\int_0^l s''(x)|Z(t,x)|^2\,\mathrm dx.
\label{basic:firstdiffusion}
\end{align}

For the $W$ term, the same calculation gives
\begin{align}
&q\operatorname{Re}\int_0^l
\overline{W(t,x)}\,\partial_x^2(h(x)W(t,x))\,\mathrm dx
\nonumber\\
=&-q\operatorname{Re}\int_0^l
\overline{\partial_x W(t,x)}
\bigl(h'(x)W(t,x)+h(x)\partial_x W(t,x)\bigr)\,\mathrm dx
\nonumber\\
=&-q\int_0^l h(x)|\partial_x W(t,x)|^2\,\mathrm dx
-\frac q2\int_0^l
h'(x)\partial_x|W(t,x)|^2\,\mathrm dx
\nonumber\\
=&-q\int_0^l h(x)|\partial_x W(t,x)|^2\,\mathrm dx
+\frac q2\int_0^l h''(x)|W(t,x)|^2\,\mathrm dx.
\label{basic:seconddiffusion}
\end{align}

For the reaction term, using $\tau as+c_\tau h=-qh''$,
we obtain
\begin{align}
&\operatorname{Re}\int_0^l
\bigl(\overline{W(t,x)}-\overline{Z(t,x)}\bigr)
\bigl(\tau a(x)u(t,x)+c_\tau(x)v(t,x)\bigr)\,\mathrm dx
\notag\\
&\quad=\operatorname{Re}\int_0^l
\bigl(\overline{W(t,x)}-\overline{Z(t,x)}\bigr)
\bigl(\tau a(x)s(x)(Z(t,x)-W(t,x))
-qh''(x)W(t,x)\bigr)\,\mathrm dx
\notag\\
&\quad=-\int_0^l
\tau a(x)s(x)|Z(t,x)-W(t,x)|^2\,\mathrm dx
-q\int_0^l h''(x)|W(t,x)|^2\,\mathrm dx
\notag\\
&\qquad+q\int_0^l h''(x)
\operatorname{Re}
\bigl(W(t,x)\overline{Z(t,x)}\bigr)\,\mathrm dx.
\label{basic:reactionenergy}
\end{align}

Adding \eqref{basic:firstdiffusion},
\eqref{basic:seconddiffusion}, and \eqref{basic:reactionenergy},
and using $ps''=-qh''$, we find
\begin{align*}
E'(t)
&=-p\int_0^l s(x)|\partial_x Z(t,x)|^2\,\mathrm dx
-q\int_0^l h(x)|\partial_x W(t,x)|^2\,\mathrm dx\\
&\quad-\int_0^l
\tau a(x)s(x)|Z(t,x)-W(t,x)|^2\,\mathrm dx\\
&\quad+\frac p2\int_0^l s''(x)|Z(t,x)|^2\,\mathrm dx
-\frac q2\int_0^l h''(x)|W(t,x)|^2\,\mathrm dx\\
&\quad+q\int_0^l h''(x)
\operatorname{Re}
\bigl(W(t,x)\overline{Z(t,x)}\bigr)\,\mathrm dx\\
&=-p\int_0^l s(x)|\partial_x Z(t,x)|^2\,\mathrm dx
-q\int_0^l h(x)|\partial_x W(t,x)|^2\,\mathrm dx\\
&\quad-\int_0^l
\tau a(x)s(x)|Z(t,x)-W(t,x)|^2\,\mathrm dx\\
&\quad-\frac q2\int_0^l h''(x)
\left(|Z(t,x)|^2+|W(t,x)|^2
-2\operatorname{Re}
\bigl(W(t,x)\overline{Z(t,x)}\bigr)\right)\,\mathrm dx\\
&=-p\int_0^l s(x)|\partial_x Z(t,x)|^2\,\mathrm dx
-q\int_0^l h(x)|\partial_x W(t,x)|^2\,\mathrm dx\\
&\quad-\int_0^l
\left(\tau a(x)s(x)+\frac q2h''(x)\right)
|Z(t,x)-W(t,x)|^2\,\mathrm dx.
\end{align*}
This proves \eqref{basic:energyidentity}.
\end{proof}

\begin{lemma}\label{lem:stable-endpoint}
Consider $L_\tau$ on $X_{\mathbb C}$ with domain $Z_{\mathbb C}$.
Let
\[
A_*=\min_{x\in[0,l]}a(x)s(x)>0,
\]
and define
\begin{equation}\label{basic:tauone}
\tau_1=2+\frac{q\|h''\|_\infty}{2A_*}>1.
\end{equation}
There is a constant $\kappa_*>0$, independent of $\tau$, such that
every eigenvalue $\lambda$ of $L_\tau$ on $X_{\mathbb C}$ satisfies
\begin{equation}\label{basic:stableend}
\operatorname{Re}\lambda\le-\kappa_*<0
\qquad\text{for all }\tau\ge\tau_1.
\end{equation}
\end{lemma}

\begin{proof}
For $\tau\ge\tau_1$ and $x\in[0,l]$, we have
\begin{align}
\tau a(x)s(x)+\frac q2h''(x)
&\ge\tau_1A_*-\frac q2\|h''\|_\infty
=2A_*=:d_*>0.
\label{basic:dpositive}
\end{align}
Set
\[
s_{\min}=\min_{[0,l]}s>0,\qquad
s_{\max}=\max_{[0,l]}s,\qquad
h_{\max}=\max_{[0,l]}h.
\]

Let $(u(t,x),v(t,x))$ be a classical solution of
\eqref{basic:linear} with initial data in $X_{\mathbb C}$, and write
$u(t,x)=s(x)Z(t,x)$ and $v(t,x)=h(x)W(t,x)$.
Adding the two equations and using the Neumann conditions shows
that the total mass is constant. Hence
\begin{equation}\label{basic:zero-mass-energy}
\int_0^l
\bigl(s(x)Z(t,x)+h(x)W(t,x)\bigr)\,\mathrm dx=0.
\end{equation}
Define
\begin{align*}
\zeta(t)&=\frac1l\int_0^l Z(t,x)\,\mathrm dx,
\qquad M=\int_0^l(s(x)+h(x))\,\mathrm dx,\\
C_*&=\max\left\{
\left(\frac{s_{\max}}2+h_{\max}\right)l^2,\ h_{\max}
\right\},
\qquad b_*=\min\{ps_{\min},d_*\}.
\end{align*}

We first estimate $E(t)$:
\begin{align}
E(t)
&=\frac12\int_0^l
\left(s(x)|Z(t,x)|^2+h(x)|W(t,x)|^2\right)\,\mathrm dx
\notag\\
&=\frac12\int_0^l
\left(s(x)|Z(t,x)-\zeta(t)+\zeta(t)|^2
+h(x)|W(t,x)-\zeta(t)+\zeta(t)|^2\right)\,\mathrm dx
\notag\\
&=\frac12\int_0^l
\left(s(x)|Z(t,x)-\zeta(t)|^2
+h(x)|W(t,x)-\zeta(t)|^2\right)\,\mathrm dx
+\frac M2|\zeta(t)|^2
\notag\\
&\quad+\operatorname{Re}\left[
\overline{\zeta(t)}
\int_0^l
\left(s(x)(Z(t,x)-\zeta(t))
+h(x)(W(t,x)-\zeta(t))\right)\,\mathrm dx
\right]
\notag\\
&=\frac12\int_0^l
\left(s(x)|Z(t,x)-\zeta(t)|^2
+h(x)|W(t,x)-\zeta(t)|^2\right)\,\mathrm dx
-\frac M2|\zeta(t)|^2
\notag\\
&\quad+\operatorname{Re}\left[
\overline{\zeta(t)}
\int_0^l
\bigl(s(x)Z(t,x)+h(x)W(t,x)\bigr)\,\mathrm dx
\right]
\notag\\
&=\frac12\int_0^l
\left(s(x)|Z(t,x)-\zeta(t)|^2
+h(x)|W(t,x)-\zeta(t)|^2\right)\,\mathrm dx
-\frac M2|\zeta(t)|^2
\label{basic:meanidentity}\\
&\le\frac{s_{\max}}2
\int_0^l|Z(t,x)-\zeta(t)|^2\,\mathrm dx
+\frac{h_{\max}}2
\int_0^l|W(t,x)-\zeta(t)|^2\,\mathrm dx
\notag\\
&\le\left(\frac{s_{\max}}2+h_{\max}\right)
\int_0^l|Z(t,x)-\zeta(t)|^2\,\mathrm dx
+h_{\max}\int_0^l|Z(t,x)-W(t,x)|^2\,\mathrm dx
\label{basic:Qestimateone}\\
&\le\left(\frac{s_{\max}}2+h_{\max}\right)l^2
\int_0^l|\partial_x Z(t,x)|^2\,\mathrm dx
+h_{\max}\int_0^l|Z(t,x)-W(t,x)|^2\,\mathrm dx
\label{basic:poincare}\\
&\le C_*\left(
\int_0^l|\partial_x Z(t,x)|^2\,\mathrm dx
+\int_0^l|Z(t,x)-W(t,x)|^2\,\mathrm dx
\right).
\label{basic:Qestimate}
\end{align}
The cross term involving $\int_0^l(sZ+hW)\,\mathrm dx$
vanishes by \eqref{basic:zero-mass-energy}.
For the first inequality, we drop the nonpositive term
$-M|\zeta(t)|^2/2$ and bound $s,h$ by their maxima.
The next inequality uses
$|W-\zeta|^2\le2|W-Z|^2+2|Z-\zeta|^2$.
Since $\zeta(t)$ is the spatial average of $Z(t,\cdot)$,
the Poincare inequality gives
$\int_0^l|Z-\zeta|^2\,\mathrm dx
\le l^2\int_0^l|\partial_x Z|^2\,\mathrm dx$,
which proves \eqref{basic:poincare}.
The last inequality follows from the definition of $C_*$.

Next, Lemma~\ref{lem:energy-identity} gives
\begin{align}
-E'(t)
&=p\int_0^l s(x)|\partial_x Z(t,x)|^2\,\mathrm dx
+q\int_0^l h(x)|\partial_x W(t,x)|^2\,\mathrm dx
\notag\\
&\quad+\int_0^l
\left(\tau a(x)s(x)+\frac q2h''(x)\right)
|Z(t,x)-W(t,x)|^2\,\mathrm dx
\notag\\
&\ge ps_{\min}\int_0^l|\partial_x Z(t,x)|^2\,\mathrm dx
+d_*\int_0^l|Z(t,x)-W(t,x)|^2\,\mathrm dx
\label{basic:dissipation}\\
&\ge b_*\left(
\int_0^l|\partial_x Z(t,x)|^2\,\mathrm dx
+\int_0^l|Z(t,x)-W(t,x)|^2\,\mathrm dx
\right).
\label{basic:dissipation-comparison}
\end{align}
The first inequality uses \eqref{basic:dpositive},
$s(x)\ge s_{\min}$, and the nonnegativity of the term containing
$|\partial_x W|^2$.
The second follows from $b_*=\min\{ps_{\min},d_*\}$.

Set $\kappa_*=b_*/(2C_*)>0$.
Combining \eqref{basic:Qestimate} and
\eqref{basic:dissipation-comparison}, we obtain
\begin{align}
E'(t)\le-\frac{b_*}{C_*}E(t)=-2\kappa_*E(t).
\label{basic:energybound}
\end{align}

Finally, let $\lambda$ be an eigenvalue of $L_\tau$ on $X_{\mathbb C}$,
with a corresponding eigenfunction $(\phi(x),\psi(x))\ne(0,0)$.
Then $u(t,x)=e^{\lambda t}\phi(x)$ and
$v(t,x)=e^{\lambda t}\psi(x)$ form a solution in $X_{\mathbb C}$.
Since $E(t)=e^{2\operatorname{Re}\lambda\,t}E(0)$ and $E(0)>0$,
\eqref{basic:energybound} yields
\begin{align*}
2\operatorname{Re}\lambda\,E(t)
&=E'(t)\le-2\kappa_*E(t),
\end{align*}
and then $\operatorname{Re}\lambda\le-\kappa_*<0.$ The constants $C_*,b_*,\kappa_*$ are independent of
$\tau\ge\tau_1$.
\end{proof}

\subsection{Spectrum properties of $L_{\tau}$}
\label{sec:finite}

Throughout this subsection, we work on $X_{\mathbb C}$ with the
common domain $Z_{\mathbb C}$ defined in \eqref{basic:spaces}.
Recall that $L_\tau=D+J_\tau$ and that $J_\tau$ is a bounded
multiplication operator on $X_{\mathbb C}$, depending affinely on $\tau$.
Put
\[
K=\sup_{\tau\in[1,\tau_1]}\|J_\tau\|_{\mathcal L(X_{\mathbb C})}
<\infty.
\]

We
 study the spectrum properties of the linear operator $L_{\tau}$.

\begin{lemma}\label{lem:nozero}
For every $\tau\ge1$, every eigenvalue $\lambda$ of
\eqref{basic:mode} satisfies
\begin{equation}\label{basic:nozero}
\lambda\ne0.
\end{equation}
\end{lemma}

\begin{proof}
Suppose, to the contrary, that \eqref{basic:mode} has a nonzero
solution $(\phi,\psi)$ with $\lambda=0$.
Adding the two differential equations gives
\[
(p\phi+q\psi)''=0.
\]
Hence $(p\phi+q\psi)'$ is constant. The Neumann boundary conditions give
$(p\phi+q\psi)'=0$, so
$p\phi+q\psi=F,$
where $F$ may be a  complex constant. Therefore,
\begin{align}\label{essu}
    \phi=\frac{F-q\psi}{p}.
\end{align}
Substituting \eqref{essu} and \eqref{basic:derivatives} into
the second equation of \eqref{basic:mode}, we obtain
\begin{equation}\label{basic:zeroequation}
q\psi''-q\frac{h''}{h}\psi-\frac{\tau ka}{ph}\psi
+\frac{\tau a}{p}F=0.
\end{equation}

Since $h$ is real-valued and strictly positive, we first set
\[
z=\frac{\psi}{h}.
\]
Substituting this identity and $\psi=hz$ into
\eqref{basic:zeroequation}, and multiplying by $h$, gives
\[
q(h^2z')'-\frac{\tau kah}{p}z
+\frac{\tau ah}{p}F=0.
\]
We can write this equation as
\[
q(h^2z')'
-\frac{\tau kah}{p}\left(z-\frac{F}{k}\right)=0.
\]

We then define
\[
r=z-\frac{F}{k}
=\frac{\psi}{h}-\frac{F}{k}.
\]
We claim that $r\equiv 0$ and then $\psi\equiv\frac{F}{k}h$. In fact, $r$ satisfies the following equation
\begin{equation}\label{basic:ratiozero}
q(h^2r')'-\frac{\tau kah}{p}r=0,
\end{equation}
with the Neumann boundary condition
\[
r'(0)=r'(l)=0.
\]

We multiply
\eqref{basic:ratiozero} by 
$\overline r$ and integrate:
\begin{align*}  0=q\int_0^l(h^2r')'\overline r\,dx
-\frac{\tau k}{p}\int_0^l ah\,r\overline r\,dx=-q\int_0^l h^2|r'|^2\,dx
-\frac{\tau k}{p}\int_0^l ah|r|^2\,dx,
\end{align*}
where we use integration by parts and the boundary conditions.
Therefore, $r\equiv0$ on $[0,l]$.
Consequently,
\[
\psi=\frac{F}{k}h.
\]
Using $ps+qh=k$, we obtain
\[
\phi=\frac{F-q\psi}{p}
=\frac{F}{pk}(k-qh)
=\frac{F}{k}s.
\]

However, we have
\[
0=\int_0^l(\phi+\psi)\,dx
=\frac{F}{k}\int_0^l(s+h)\,dx
=\frac{F}{k}M.
\]
Since $M/k>0$, then we obtain $F=0$, which implies $\phi=\psi=0$.
This contradicts the choice of a nonzero eigenfunction.
\end{proof}
 
\begin{lemma}
    Each $L_\tau$, $\tau\in[1,\tau_1]$, is closed and has compact
resolvent. Its spectrum consists of isolated eigenvalues of finite
algebraic multiplicity and it have finitely many eigenvalues in bounded domain.
\end{lemma}
\begin{proof}
We first consider the diffusion part of $L_{\tau}$. The eigenvalues of the diffusion operator $D$ can be computed
explicitly, and its spectrum is
\[
\sigma(D)=
\left\{
0,\,
-p\left(\frac{n\pi}{l}\right)^2,\,
-q\left(\frac{n\pi}{l}\right)^2:
n\ge1
\right\}\subset  (-\infty,0]\subset \mathbb R.
\]
Moreover, its resolvent satisfies
\begin{equation}\label{reso}
    \|(\lambda I-D)^{-1}\|_{\mathcal L(X_{\mathbb C})}
=\frac{1}{\operatorname{dist}(\lambda,\sigma(D))},
\quad \lambda\in\rho(D).
\end{equation}

Since $J_\tau$ is bounded on $X_{\mathbb C}$,
$L_\tau=D+J_\tau$ is closed with domain $Z_{\mathbb C}$.
For real constant $\xi>K$, since $\sigma(D)\le 0$ and then we have $\operatorname{dist}(\xi,\sigma(D))=\xi$, then
$$
\|J_\tau(\xi I-D)^{-1}\|_{\mathcal L(X_{\mathbb C})}
\le\frac{K}{\operatorname{dist}(\xi,\sigma(D))}=\frac{K}{\xi}<1.
$$
Note that
$$
\xi I-L_\tau=\bigl[I-J_\tau(\xi I-D)^{-1}\bigr](\xi I-D)
$$
and the Neumann series therefore give
\begin{equation}\label{factor}
    (\xi I-L_\tau)^{-1}
=(\xi I-D)^{-1}\bigl[I-J_\tau(\xi I-D)^{-1}\bigr]^{-1}.
\end{equation}

The operator $(\xi I-D)^{-1}$  is compact on $X_{\mathbb C}$ and $[I-J_\tau(\xi I-D)^{-1}]^{-1}$ is a bounded operator.
Thus $(\xi I-L_\tau)^{-1}$ is compact. The spectral assertions
 follow directly from the compact  operator.
\end{proof}

\begin{lemma}
    For every $\tau\in[1,\tau_1]$ and every
$\lambda\in\sigma(L_\tau)$, we have
\begin{equation}\label{spec:box}
\operatorname{Re}\lambda\le K,
\qquad
|\operatorname{Im}\lambda|\le K.
\end{equation}
Moreover, for every fixed $\tau\in[1,\tau_1]$, the set
$\sigma(L_\tau)\cap i\mathbb R$ is finite, and there exists
$\delta_\tau>0$ such that
\[
\sigma(L_\tau)\cap
\{\lambda\in\mathbb C:0<|\operatorname{Re}\lambda|<\delta_\tau\}
=\varnothing.
\]

\end{lemma}
\begin{proof}
Let $L_\tau\Phi=\lambda\Phi$, where
$\Phi=(\phi,\psi)\in Z_{\mathbb C}$, and normalize
$\|\Phi\|_{X_{\mathbb C}}=1$.
Taking the $H^1$ inner product with $\Phi$ and integrating by parts,
using $\phi'=\psi'=0$ at the endpoints, gives
\begin{equation}\label{EST}
\lambda
=-p\int_0^l\bigl(|\phi'|^2+|\phi''|^2\bigr)\,dx
-q\int_0^l\bigl(|\psi'|^2+|\psi''|^2\bigr)\,dx
+\langle J_\tau\Phi,\Phi\rangle_{X_{\mathbb C}}.
\end{equation}
The diffusion terms are real and nonpositive, while
\[
|\langle J_\tau\Phi,\Phi\rangle_{X_{\mathbb C}}|
\le\|J_\tau\|_{\mathcal L(X_{\mathbb C})}
\|\Phi\|_{X_{\mathbb C}}^2\le K.
\]
Taking real and imaginary parts  in \eqref{EST} proves \eqref{spec:box}.

Fix $\tau\in[1,\tau_1]$. By (1) and \eqref{spec:box}, the set
$\{\lambda\in\sigma(L_\tau):|\operatorname{Re}\lambda|\le1\}$
is finite, since it lies in the bounded rectangle
$\{|\operatorname{Re}\lambda|\le1,\ |\operatorname{Im}\lambda|\le K\}$.
In particular, $\sigma(L_\tau)\cap i\mathbb R$ is finite. Set
\[
\delta_\tau
=\min\left(
\{1\}\cup
\left\{|\operatorname{Re}\lambda|:
\lambda\in\sigma(L_\tau),
0<|\operatorname{Re}\lambda|\le1\right\}
\right)>0.
\]
\end{proof}

\begin{lemma}
    For every $\tau\in[1,\tau_1]$ and every $\omega\in\mathbb R$
with $|\omega|\ge2K+1$,
\begin{equation}\label{spec:high-frequency}
i\omega\in\rho(L_\tau),\qquad
\|(i\omega I-L_\tau)^{-1}\|_{\mathcal L(X_{\mathbb C})}
\le\frac{2}{|\omega|}.
\end{equation}
\end{lemma}

\begin{proof}
   For $\omega\in\mathbb R\setminus\{0\}$, we recall that $\sigma(D)\subset(-\infty,0]$ and \eqref{reso}, we have
\[
\|(i\omega I-D)^{-1}\|_{\mathcal L(X_{\mathbb C})}
\le\frac1{|\omega|}.
\]
If $|\omega|\ge2K+1$, then
$\|J_\tau(i\omega I-D)^{-1}\|\le K/|\omega|<1$.
The fact \eqref{factor} yields
$i\omega\in\rho(L_\tau)$ and
\[
\|(i\omega I-L_\tau)^{-1}\|_{\mathcal L(X_{\mathbb C})}
\le\frac{1}{(1-K/|\omega|)|\omega|}
=\frac{1}{|\omega|-K}
\le\frac{2}{|\omega|}.
\]
\end{proof}

\begin{lemma}
\label{lem:common-circle}
There exists $R>\sqrt{2}K+1$, independent of $\tau$, such that
\[
\Gamma:=\{\lambda\in\mathbb C:|\lambda|=R\}
\subset\rho(L_\tau)
\qquad\text{for every }\tau\in[1,\tau_1].
\]
\end{lemma}

\begin{proof}
We can consequently choose $R$ such that
\[
R>\sqrt{2}K+1,
\qquad
\operatorname{dist}(-R,\sigma(D))>K+1.
\]
If $|\lambda|=R$ and $b\in\sigma(D)\subset(-\infty,0]$, then
\[
|\lambda-b|
\ge\bigl||\lambda|-|b|\bigr|
=|-R-b|.
\]
Hence $\operatorname{dist}(\lambda,\sigma(D))>K+1$, and we recall that $J_{\tau}$ is uniformly bounded with respect to $\tau$, we have
\[
\|J_\tau(\lambda I-D)^{-1}\|_{\mathcal L(X_{\mathbb C})}
\le\frac{K}{K+1}<1.
\]
The same argument used in (1)
gives $\lambda\in\rho(L_\tau)$ for every $\tau\in[1,\tau_1]$.
\end{proof}

Let \(\Gamma_R=\{\lambda\in\mathbb C:|\lambda|=R\}\) be the common
resolvent circle obtained above, oriented counterclockwise.
For \(\tau\in[1,\tau_1]\), define
\[
\Pi_\tau
:=\frac{1}{2\pi i}\int_{\Gamma_R}
(\lambda I-L_\tau)^{-1}\,d\lambda,
\qquad
E_\tau:=\Pi_\tau(X_{\mathbb C}),
\qquad
E_\tau^{\mathbb R}:=\Pi_\tau(X).
\]
By standard properties of Riesz projections,
\(\Pi_\tau\in\mathcal L(X_{\mathbb C},Z_{\mathbb C})\) is a finite-rank
projection whose range is the direct sum of the generalized eigenspaces
associated with the eigenvalues inside \(\Gamma_R\). Moreover,
\[
L_\tau\Pi_\tau U=\Pi_\tau L_\tau U,
\qquad U\in Z_{\mathbb C}.
\]
Since \(L_\tau\) has real coefficients and \(\Gamma_R\) is invariant
under complex conjugation, \(\Pi_\tau\) commutes with complex
conjugation. Consequently,
\[
E_\tau^{\mathbb R}=E_\tau\cap X\subset Z,
\qquad
E_\tau=E_\tau^{\mathbb R}+iE_\tau^{\mathbb R}.
\]

Next, we give some properties for the spectrum projection and the finite dimension function space producted by the spectrum projection in the cycle $\Gamma_R$ with respect to $\tau$.
\begin{lemma}
For the spectral projection \(\Pi_\tau\) with \(\Gamma_R=\{\lambda\in\mathbb C:|\lambda|=R\}\) and the space \(E_\tau\) defined above,
the following statements hold.
\begin{enumerate}
\item[(1)]
The map \(\tau\mapsto\Pi_\tau\) is real-analytic  with  respect to real constant $\tau$ in
\(\mathcal L(X_{\mathbb C},Z_{\mathbb C})\).

\item[(2)]
There exists a nonnegative integer \(m\), independent of \(\tau\),
such that
\[
\dim_{\mathbb C}E_\tau
=\dim_{\mathbb R}E_\tau^{\mathbb R}=m,
\qquad \tau\in[1,\tau_1].
\]

\item[(3)]
There exist real-analytic maps
\[
\Phi_1,\ldots,\Phi_m:[1,\tau_1]\to Z
\]
such that, for every \(\tau\), these functions form a real basis
of \(E_\tau^{\mathbb R}\) and a complex basis of \(E_\tau\).
\end{enumerate}
\end{lemma}

\begin{proof}
For (1), we first show that $\Pi_\tau$ is real-analytic with
values in $\mathcal L(X_{\mathbb C},Z_{\mathbb C})$,
where $Z_{\mathbb C}$ carries the $H^3$ norm.

Write
\[
\mathcal R_\tau(\lambda)=(\lambda-L_\tau)^{-1}.
\]
By Lemma~\ref{lem:common-circle} and the graph estimate
\eqref{basic:graph-estimate}, there is a constant $C_0>0$ such that
\[
\|\mathcal R_\tau(\lambda)\|_{\mathcal L(X_{\mathbb C})}
\leq1,
\qquad
\|\mathcal R_\tau(\lambda)\|_{\mathcal L(X_{\mathbb C},Z_{\mathbb C})}
\leq C_0
\]
for all $\tau\in[1,\tau_1]$ and $\lambda\in\Gamma$.

Fix $\tau_0\in[1,\tau_1]$.
Due to the definition of  $J_\tau$, there is
a fixed $B\in\mathcal L(X_{\mathbb C})$ such that
\[
L_{\tau_0+h}=L_{\tau_0}+hB, \quad h\in \mathbb R.
\]
For $|h|\|B\|_{\mathcal L(X_{\mathbb C})}<1$,
the Neumann series gives
\[
\mathcal R_{\tau_0+h}(\lambda)
=
\sum_{n=0}^{\infty}
h^n\mathcal R_{\tau_0}(\lambda)
\bigl[B\mathcal R_{\tau_0}(\lambda)\bigr]^n.
\]
Moreover,
\begin{align}
\left\|
h^n\mathcal R_{\tau_0}(\lambda)
\bigl[B\mathcal R_{\tau_0}(\lambda)\bigr]^n
\right\|_{\mathcal L(X_{\mathbb C},Z_{\mathbb C})} \leq
C_0\bigl(|h|\|B\|_{\mathcal L(X_{\mathbb C})}\bigr)^n.
\end{align}
Thus this series converges in
$\mathcal L(X_{\mathbb C},Z_{\mathbb C})$, uniformly for
$\lambda\in\Gamma$.
Integrating term by term gives a convergent power
series for $\Pi_{\tau_0+h}$ in the same operator norm.
Hence $\Pi_\tau$ is real-analytic with values in
$\mathcal L(X_{\mathbb C},Z_{\mathbb C})$.

Since $Z_{\mathbb C}$ is continuously included in
$X_{\mathbb C}$, $\Pi_\tau$ is also real-analytic,
and in particular continuous, with values in
$\mathcal L(X_{\mathbb C})$.

To prove (2), by Lemma~\ref{lem:common-circle}, each $L_\tau$ has
compact resolvent. Thus there are only  finitely many
eigenvalues inside $\Gamma$, and each has finite
algebraic multiplicity. The properties of the spectral
projection $\Pi_\tau$ give
$\Pi_\tau^2=\Pi_\tau$.
Moreover, $E_\tau$ is the span of all generalized
eigenvectors belonging to the eigenvalues inside
$\Gamma$. Hence
\[
E_\tau\subset Z_{\mathbb C},
\quad
\dim_{\mathbb C}E_\tau<\infty.
\]

We  claim that there exists $m\ge0$, for any $\tau\in[1,\tau_1]$, it holds that $\dim_{\mathbb C}E_\tau=m$.
 Fix $\tau_0$. By continuity, there exists $\delta>0$ such that  for $|\tau-\tau_0|<\delta$ we have
\[
c:=\|\Pi_\tau-\Pi_{\tau_0}\|<1.
\]
Let $\Phi\in E_{\tau_0}$ and suppose that $\Pi_\tau \Phi=0$.
Since $\Pi_{\tau_0}\Phi=\Phi$, we obtain
\[
\Phi=(\Pi_{\tau_0}-\Pi_\tau)\Phi,
\]
then
$
\|\Phi\|_{X_{\mathbb C}}\leq c\|\Phi\|_{X_{\mathbb C}}.$
Since $c<1$, this implies $\Phi=0$. It means that if $\Phi\in E_{\tau_0}$ and $\Pi_\tau \Phi=0$, then $\Phi=0$.
Thus the linear map
$\Pi_\tau:E_{\tau_0}\to E_\tau$
is one-to-one.   Hence
$
\dim_{\mathbb C}E_{\tau_0}
\leq \dim_{\mathbb C}E_\tau.$
Similarly, we have
$
\dim_{\mathbb C}E_\tau
\leq \dim_{\mathbb C}E_{\tau_0}.$
Hence we have  
$
\dim_{\mathbb C}E_\tau=\dim_{\mathbb C}E_{\tau_0}.$
Thus $\dim_{\mathbb C}E_\tau$ is locally constant.
Since $[1,\tau_1]$ is connected, it is constant
on the whole interval. We conclude that
\[
m=\dim_{\mathbb C}E_1=\dim_{\mathbb C}E_{\tau}=\dim_{\mathbb R}E_\tau^{\mathbb R}\quad \tau\in[1,\tau_1].
\]

Next, we complete the proof of this lemma. Since $\Pi_1$ maps real-valued functions to real-valued
functions, and   $\dim_{\mathbb C}E_1=\dim_{\mathbb R}E_1^{\mathbb R}=m$, we can choose
real-valued functions $f_1,\ldots,f_m$ which form
a complex basis of $E_1$.

To prove (3), we use the method in \cite{kato}.
We move this basis from $\tau=1$ to any $\tau\in [1,\tau_1]$.
All primes below mean Fr\'echet  derivatives with respect to
$\tau$. 
Set
\[
A_\tau
=
\Pi_\tau'\Pi_\tau-\Pi_\tau\Pi_\tau'.
\]
Note that $\Pi_{\tau}$ is real-analytic on $\tau$ and then $\Pi_{\tau}'$ is real-analytic on $\tau$. Thus $A_\tau$ is a real-analytic family of bounded operators.
For each $j=1,\ldots,m$, define $\Phi_j$ by the equation
\[
\Phi_j'(\tau)=A_\tau\Phi_j(\tau),
\qquad
\Phi_j(1)=f_j.
\]
 In fact, $\Phi_j({\tau})$  is the moving real basis we need.
Since $A_\tau$ is continuous and bounded on the
parameter interval, there exists
a unique solution on $[1,\tau_1]$.
Since $A_\tau$ is real-analytic, the solution $\Phi_j(\tau)$
is real-analytic with values in $X_{\mathbb C}$. Moreover, since $A_{\tau}$ maps real-valued function to  real-valued function and $f_j$ is real-valued, the solution $\Phi_j(\tau)$ is real-valued function.

Next, we show that $\Phi_j(\tau)$ stays in $E_\tau\subset Z_{\mathbb C}$.
Differentiating $\Pi_\tau^2=\Pi_\tau$ gives
$\Pi_\tau'\Pi_\tau+\Pi_\tau\Pi_\tau'
=\Pi_\tau'.$
Multiplying this identity on the left by $\Pi_\tau$
gives
$
\Pi_\tau\Pi_\tau'\Pi_\tau=0.$
Consequently,
$
(I-\Pi_\tau)A_\tau=\Pi_\tau'\Pi_\tau.$

Let
\[
r_j(\tau)=(I-\Pi_\tau)\Phi_j(\tau).
\]
We claim that $r_j(\tau)\equiv 0$.
In fact, we have
\begin{align*}
r_j'
=-\Pi_\tau'\Phi_j+(I-\Pi_\tau)\Phi_j'=-\Pi_\tau'\Phi_j+\Pi_\tau'\Pi_\tau\Phi_j=-\Pi_\tau'r_j.
\end{align*}
Since $f_j\in E_1$, we have $r_j(1)=0$.
Uniqueness for this linear equation gives $r_j=0$.
Thus
\[
\Pi_\tau\Phi_j(\tau)=\Phi_j(\tau),
\qquad
\Phi_j(\tau)\in E_\tau\subset Z_{\mathbb C}.
\]
Since $\Pi_\tau$ is real-analytic in
$\mathcal L(X_{\mathbb C},Z_{\mathbb C})$, the identity
$\Phi_j(\tau)=\Pi_\tau\Phi_j(\tau)$ also proves real-analyticity
with values in $Z_{\mathbb C}$.

Finally, we claim that for any $\tau\in[1,\tau_1]$, $\Phi_j(\tau)$ is linearly independent for $j=1,...m.$ Fix any $\hat\tau\in[1,\tau_1]$ and suppose
that
\[
\sum_{j=1}^m c_j\Phi_j(\hat\tau)=0,
\qquad c_j\in\mathbb C.
\]
The function
\[
w(\tau)=\sum_{j=1}^m c_j\Phi_j(\tau)
\]
satisfies
\[
w'=A_\tau w,
\qquad
w(\hat\tau)=0.
\]
By uniqueness, $w=0$ on the whole interval.
In particular,
\[
0=w(1)=\sum_{j=1}^m c_jf_j.
\]
Since $f_1,\ldots,f_m$ form a basis of $E_1$,
all $c_j$ are zero. Hence
$\Phi_1(\tau),\ldots,\Phi_m(\tau)$ are linearly
independent for every $\tau$.
They belong to the $m$-dimensional space $E_\tau$,
so they form a basis of $E_\tau$.
\end{proof}

\subsection{Finite-dimensional matrix representation of $L_\tau$}

We first represent the restriction of $L_\tau$ to $E_\tau$ by a finite matrix and show that this representation preserves
the eigenvalues and their algebraic multiplicities.

\textbf{Matrix representation for $L_{\tau}$ on $E_{\tau}$.}
Fix $\tau\in[1,\tau_1]$.
The space $E_\tau$ is the sum of the generalized
eigenspaces corresponding to the eigenvalues inside
$\Gamma$. Hence
\[
L_\tau E_\tau\subset E_\tau.
\]
Since $\Phi_1(\tau),\ldots,\Phi_m(\tau)$ form a basis
of $E_\tau$, there are unique numbers $M_{ij}(\tau)$
such that
\[
L_\tau\Phi_j(\tau)
=
\sum_{i=1}^m M_{ij}(\tau)\Phi_i(\tau),
\qquad j=1,\ldots,m.
\]
Define the constant matrix
\[
M(\tau)=\bigl(M_{ij}(\tau)\bigr)_{i,j=1}^m.
\]
Thus the $j$th column of $M(\tau)$ contains the
coordinates of $L_\tau\Phi_j(\tau)$ in this basis.

Every $\Phi\in E_\tau$ has a unique expression
\[
\Phi=\sum_{j=1}^m c_j\Phi_j(\tau),
\quad
c=(c_1,\ldots,c_m)^{T}\in\mathbb C^m.
\]
By the linearity of $L_\tau$,
\begin{align*}
L_\tau \Phi
=\sum_{j=1}^m c_j
  \sum_{i=1}^m M_{ij}(\tau)\Phi_i(\tau)
=\sum_{i=1}^m
  \left(\sum_{j=1}^m M_{ij}(\tau)c_j\right)
  \Phi_i(\tau)
=\sum_{i=1}^m (M(\tau)c)_i\Phi_i(\tau).
\end{align*}
Therefore,
\[
L_\tau \Phi-\lambda \Phi
=
\sum_{i=1}^m
\bigl((M(\tau)c)_i-\lambda c_i\bigr)\Phi_i(\tau).
\]
The basis functions are linearly independent, so
the right-hand side is zero if and only if
\[
(M(\tau)c)_i=\lambda c_i
\qquad\text{for every }i=1,\ldots,m.
\]
Equivalently,
\begin{align}\label{ev}
    L_\tau \Phi=\lambda \Phi
\quad\text{iff}\quad
M(\tau)c=\lambda c.
\end{align}
Also, $\Phi\ne0$ if and only if $c\ne0$.
Thus the eigenvalues of $L_\tau$ on $E_\tau$ are
exactly the eigenvalues of $M(\tau)$.

The same relation holds for generalized eigenvectors.
Indeed, for every integer $k\geq1$,
\[
(L_\tau-\lambda I)^kU
=
\sum_{i=1}^m
\bigl((M(\tau)-\lambda I)^kc\bigr)_i\Phi_i(\tau).
\]
Hence algebraic multiplicities are also preserved.

We conclude that the eigenvalues of $M(\tau)$
 represents the restriction of $L_\tau$ to
$E_\tau$. Thus, for every $\tau\in[1,\tau_1]$,
\begin{equation}\label{eq:matrix-operator-spectrum}
\sigma(M(\tau))
=
\sigma(L_\tau|_{E_\tau})
=
\sigma(L_\tau)\cap\{z\in\mathbb C:|z|<R\},
\end{equation}
with the same algebraic multiplicities.
Since the circle $\Gamma$ encloses every eigenvalue of $L_\tau$
with nonnegative real part, we also have
\begin{equation}\label{eq:nonnegative-spectrum-correspondence}
\sigma(M(\tau))\cap\{z\in\mathbb C:\operatorname{Re}z\ge0\}
=
\sigma(L_\tau)\cap\{z\in\mathbb C:\operatorname{Re}z\ge0\},
\end{equation}
again with the same algebraic multiplicities.
In particular,
\begin{equation}\label{eq:imaginary-spectrum-correspondence}
\sigma(M(\tau))\cap i\mathbb R
=
\sigma(L_\tau)\cap i\mathbb R.
\end{equation}

\textbf{Explicit formula of $M(\tau)$.}
We next derive an explicit formula for $M(\tau)$
using the $L^2$ Gram matrix of the chosen basis.
Define the real matrices $H(\tau)$ and $C(\tau)$ by
\[
H_{ij}(\tau)
=
\int_0^l
\Phi_i(\tau,x)\cdot\Phi_j(\tau,x)\,dx,\quad
C_{ij}(\tau)
=
\int_0^l
\Phi_i(\tau,x)\cdot
\bigl(L_\tau\Phi_j(\tau)\bigr)(x)\,dx.
\]
Taking the  product of
\[
L_\tau\Phi_j(\tau)
=
\sum_{k=1}^m M_{kj}(\tau)\Phi_k(\tau)
\]
with $\Phi_i(\tau)$ and integrating over $(0,l)$
gives
\begin{align*}
C_{ij}(\tau)=
\sum_{k=1}^m M_{kj}(\tau)
\int_0^l
\Phi_i(\tau,x)\cdot\Phi_k(\tau,x)\,dx=
\sum_{k=1}^m H_{ik}(\tau)M_{kj}(\tau).
\end{align*}
This is the matrix identity
\[
C(\tau)=H(\tau)M(\tau).
\]

The matrix $H(\tau)$ is the Gram matrix of the basis
$\{\Phi_j(\tau)\}_{j=1}^m$.
Hence $H(\tau)$ is positive definite and invertible.
Consequently,
\[
M(\tau)=H(\tau)^{-1}C(\tau).
\]
Since $H(\tau)$ and $C(\tau)$ are real,
$M(\tau)$ is also real.

\textbf{$M(\tau)$ depends real-analytically
on $\tau$.}
We finally prove that $M(\tau)$ depends real-analytically
on $\tau$. In view of $M(\tau)=H(\tau)^{-1}C(\tau)$,
it suffices to prove that $H(\tau)$ and $C(\tau)$
are real-analytic.

The basis functions $\Phi_j(\tau)$ are real-analytic
with values in $X_{\mathbb C}$, which is continuously
embedded in $L^2(0,l;\mathbb C^2)$.
Since the integral defining $H_{ij}(\tau)$ is a
continuous bilinear pairing, every entry of $H(\tau)$
is real-analytic.
Moreover, $\det H(\tau)\ne0$, so $H(\tau)^{-1}$
is real-analytic.

It remains to prove that $C(\tau)$ is real-analytic.
 Note that
\[
\mathcal R_\tau(\lambda)=(\lambda-L_\tau)^{-1}.
\]
By the resolvent series proved above, near each fixed $\tau_0$,
the expansion of $\mathcal R_\tau(\lambda)$ in powers of $\tau-\tau_0$
converges in $\mathcal L(X_{\mathbb C})$, uniformly for
$\lambda\in\Gamma$. Hence we may multiply this series by
$\lambda$ and integrate term by term. It follows that
\[
S_\tau:=
\frac{1}{2\pi i}\int_\Gamma
\lambda \mathcal R_\tau(\lambda)\,d\lambda
\]
is real-analytic with values in $\mathcal L(X_{\mathbb C})$.

Since
\[
L_\tau \mathcal R_\tau(\lambda)=\lambda \mathcal R_\tau(\lambda)-I
\]
and $L_\tau$ is closed, integration along $\Gamma$ gives
\[
L_\tau\Pi_\tau
=
\frac{1}{2\pi i}\int_\Gamma
\bigl(\lambda \mathcal R_\tau(\lambda)-I\bigr)\,d\lambda
=
S_\tau,
\]
where we used $\int_\Gamma d\lambda=0$.

For each $j$, we have $\Pi_\tau\Phi_j(\tau)=\Phi_j(\tau)$.
Therefore,
\[
L_\tau\Phi_j(\tau)
=
L_\tau\Pi_\tau\Phi_j(\tau)
=
S_\tau\Phi_j(\tau).
\]
Both $S_\tau$ and $\Phi_j(\tau)$ are real-analytic in their
respective spaces. Thus $\tau\mapsto L_\tau\Phi_j(\tau)$
is real-analytic with values in $X_{\mathbb C}$.
Hence each $C_{ij}(\tau)$ is real-analytic.

Therefore, $M(\tau)=H(\tau)^{-1}C(\tau)$ is real-analytic. We can therefore study the eigenvalues inside
$\Gamma$ using the characteristic polynomial
\[
P(z,\tau)=\det\bigl(zI-M(\tau)\bigr),
\]
whose coefficients are real-analytic functions of
$\tau$.

\subsection{An isolated parameter where the number of unstable eigenvalues changes}
We
 recall that 
 the entries of $M(\tau)$ are real-analytic on an
open interval $J$ containing $[1,\tau_1]$, where $\tau_1>1$.
And that every eigenvalue of $L_{\tau_1}$ has negative real part. Then we have
\begin{lemma}\label{finite}
Let $L_\tau$, $E_\tau$, and $M(\tau)$ be as above.
The set
\[
\mathcal Z
=
\left\{
\tau\in[1,\tau_1]:
\sigma(L_\tau)\cap i\mathbb R\ne\emptyset
\right\}
\]
is finite. This includes the case $\mathcal Z=\emptyset$.
\end{lemma}

\begin{proof}
Recall that $m=\dim_{\mathbb C}E_\tau$ is independent of
$\tau\in[1,\tau_1]$.
If $m=0$, then $\mathcal Z=\emptyset$ by
\eqref{eq:imaginary-spectrum-correspondence}, so the conclusion holds.
We therefore assume that $m\ge1$.
Let
\[
P(z,\tau)=\det(zI-M(\tau))
=\sum_{k=0}^{m}p_k(\tau)z^k.
\]
For each $\tau$, $P(z,\tau)$ is a polynomial of degree $m$ in $z$.
The coefficients $p_k(\tau)$ are real and real-analytic on $J$.

Fix $\tau$. We consider whether there exist complex
polynomials $A$ and $B$, not both zero, of degree at most
$m-1$, such that
\begin{equation}\label{eq:auxiliary-polynomial}
A(z)P(z,\tau)+B(z)P(-z,\tau)\equiv0.
\end{equation}
Write
\[
A(z)=\sum_{r=0}^{m-1}a_rz^r,
\qquad
B(z)=\sum_{r=0}^{m-1}b_rz^r.
\]
Here the coefficients
\[
a_0,\ldots,a_{m-1},b_0,\ldots,b_{m-1}
\]
are unknown, while the coefficients of $P(z,\tau)$ are fixed.

The polynomial on the left-hand side of
\eqref{eq:auxiliary-polynomial} has degree at most $2m-1$.
Comparing the coefficients of $z^0,\ldots,z^{2m-1}$
gives a homogeneous linear system of $2m$ equations
in the $2m$ unknowns
$a_0,\ldots,a_{m-1},b_0,\ldots,b_{m-1}$.
Let $Q(\tau)$ be its $2m\times2m$ coefficient matrix.
Then this system can be written as
\[
Q(\tau)
(a_0,\ldots,a_{m-1},b_0,\ldots,b_{m-1})^T=0.
\]
Set
\[
\Delta(\tau)=\det Q(\tau).
\]
Every entry of $Q(\tau)$ is either zero or a coefficient
$p_k(\tau)$ with a fixed sign.
Thus $\Delta$ is real-analytic on $J$.
Moreover, \eqref{eq:auxiliary-polynomial} has a solution
with $A$ and $B$ not both zero if and only if $\Delta(\tau)=0$.

We first claim that
\begin{equation}\label{common}
\Delta(\tau)=0
\quad\Longleftrightarrow\quad
P(z,\tau)\ \text{and}\ P(-z,\tau)
\ \text{have a common complex root}.
\end{equation}

Suppose first that $z_0$ is a common root of
$P(z,\tau)$ and $P(-z,\tau)$. We define
\[
A(z)=\frac{P(-z,\tau)}{z-z_0},
\qquad
B(z)=-\frac{P(z,\tau)}{z-z_0}.
\]
They are nonzero polynomials of degree $m-1$. They satisfy
\[
A(z)P(z,\tau)+B(z)P(-z,\tau)\equiv0.
\]
Thus the coefficient system has a nonzero solution,
and $\Delta(\tau)=0$.

For the converse, suppose that $P(z,\tau)$ and $P(-z,\tau)$
have no common root, and let $A,B$ satisfy
\eqref{eq:auxiliary-polynomial}, with
$\deg A,\deg B\le m-1$.
Then $B$ must vanish at every root of $P(z,\tau)$
with at least the same multiplicity.
Indeed,
\[
B(z)P(-z,\tau)=-A(z)P(z,\tau),
\]
and $P(-z,\tau)$ does not vanish at any root of $P(z,\tau)$.
Since $\deg P=m$ and $\deg B<m$, this gives $B=0$.
The identity then gives $A=0$.
Thus the corresponding homogeneous linear system
has only the zero solution, so $Q(\tau)$ is
invertible and $\Delta(\tau)\ne0$.

By \eqref{eq:matrix-operator-spectrum} and the assumption on
$L_{\tau_1}$, all roots of $P(z,\tau_1)$ have negative real part.
If $P(z,\tau_1)$ and $P(-z,\tau_1)$ had a common root $z_0$,
then both $z_0$ and $-z_0$ would be roots of $P(z,\tau_1)$.
This would require
\[
\operatorname{Re}z_0<0
\qquad\text{and}\qquad
-\operatorname{Re}z_0<0,
\]
which is impossible. Hence, by \eqref{common},
\[
\Delta(\tau_1)\ne0.
\]

Now suppose that $L_\tau$ has an eigenvalue $i\omega$,
where $\omega\neq 0$.
By \eqref{eq:imaginary-spectrum-correspondence},
$i\omega$ is also an eigenvalue of $M(\tau)$.
Since $P$ has real coefficients,
\[
P(i\omega,\tau)=0
\quad\Longrightarrow\quad
P(-i\omega,\tau)
=
\overline{P(i\omega,\tau)}
=0.
\]
Thus $i\omega$ is a common root of $P(z,\tau)$ and
$P(-z,\tau)$. By \eqref{common}, $\Delta(\tau)=0$.
Therefore,
\[
\mathcal Z
\subset
\{\tau\in[1,\tau_1]:\Delta(\tau)=0\}.
\]

Since $\Delta$ is real-analytic on $J$ and $\Delta(\tau_1)\ne0$,
it is not identically zero on $J$.
Hence its zeros cannot accumulate at any point of $J$.
If $\Delta$ had infinitely many distinct zeros in $[1,\tau_1]$,
compactness would give an accumulation point in
$[1,\tau_1]\subset J$, a contradiction.
Thus $\mathcal Z$ is finite.
\end{proof}

For $\tau\in[1,\tau_1]$, let $n_+(\tau)$ be the number
of eigenvalues of $L_\tau$ with positive real part,
counted with algebraic multiplicity.

\begin{lemma}\label{lem:unstable-count-jump}
Under the assumptions of Lemma~\ref{finite}, if
$n_+(1)>0$ and for every $\tau\in[1,\tau_1]$,
$0\notin\sigma(L_\tau)$.
Then there exist $\tau_*\in(1,\tau_1)$ and $\rho_0>0$ small enough,
such that
$
\sigma(L_{\tau_*})\cap i\mathbb R\ne\varnothing,$
whereas
\[
\sigma(L_\tau)\cap i\mathbb R=\varnothing
\qquad
\text{if }0<|\tau-\tau_*|\le\rho_0.
\]
Moreover, for every $0<\rho\le\rho_0$, we have
$n_+(\tau_*-\rho)\ne n_+(\tau_*+\rho)$.
\end{lemma}

\begin{proof}
By \eqref{eq:nonnegative-spectrum-correspondence},
the eigenvalues of $L_\tau$ and $M(\tau)$ with nonnegative
real part agree, with the same algebraic multiplicities.
Thus $n_+(\tau)$ is finite and equals the number of roots of
\[
P(z,\tau)=\det(zI-M(\tau))
\]
in the open right half-plane, counted with multiplicity.
The coefficients of $P(z,\tau)$ depend continuously on $\tau$.
By continuity of the roots, $n_+$ is locally constant outside
$\mathcal Z$ and is constant on every interval disjoint from
$\mathcal Z$.

Next, choose a root $\mu$ of $P(z,1)$ with
$\operatorname{Re}\mu>0$; such a root exists because $n_+(1)>0$.
Take a small disk around $\mu$ whose closure lies in the
open right half-plane and whose boundary contains no root
of $P(z,1)$.
 Rouch\'e Theorem shows that this disk contains
a root of $P(z,\tau)$ for all $\tau$ sufficiently close
to $1$. Therefore,
\[
n_+(\tau)>0
\qquad\text{for all }\tau>1\text{ sufficiently close to }1.
\]
Since $\mathcal Z$ is finite, we can choose
$\tau_a\in(1,\tau_1)\setminus\mathcal Z$ such that
\[
n_+(\tau_a)>0.
\]
However, at 
$\tau=\tau_1\notin\mathcal Z$,
$n_+(\tau_1)=0.$
Hence there must be at least one point of $\mathcal Z$ between
$\tau_a$ and $\tau_1$.
By Lemma~\ref{finite}, the set
\[
\mathcal Z
=
\{\tau\in[1,\tau_1]:
\sigma(L_\tau)\cap i\mathbb R\ne\varnothing\}
\]
is finite and even non-empty.
Set
\[
\mathcal Z\cap(\tau_a,\tau_1)
=
\{\theta_1,\ldots,\theta_N\},
\qquad
\theta_1<\cdots<\theta_N.
\]
Set $\theta_0=\tau_a$ and $\theta_{N+1}=\tau_1$.
For $j=0,\ldots,N$, let $k_j$ be the constant value of
$n_+$ on $(\theta_j,\theta_{j+1})$.
Local constancy at $\tau_a$ and $\tau_1$ gives
\[
k_0=n_+(\tau_a)>0,
\qquad
k_N=n_+(\tau_1)=0.
\]
Thus there is an index $j\in\{1,\ldots,N\}$ such that $k_{j-1}\ne k_j.$

Set $\tau_*=\theta_j$ and choose
\[
0<\rho_0<
\frac12\min\{\theta_j-\theta_{j-1},
             \theta_{j+1}-\theta_j\}.
\]
Then $\tau_*$ is the only point of $\mathcal Z$ in
$[\tau_*-\rho_0,\tau_*+\rho_0]$.
For every $0<\rho\le\rho_0$, we have
\[
n_+(\tau_*-\rho)=k_{j-1}\neq k_j=
n_+(\tau_*+\rho)
\]

Finally, we conclude  $\tau_*\in\mathcal Z$, and $L_{\tau_*}$ has
an isolated purely imaginary eigenvalue.
Since $0\notin\sigma(L_{\tau_*})$, every such eigenvalue
has the form $i\omega$ with $\omega\ne0$.
\end{proof}

\subsection{Construction of the center manifold}
Let $\tau_*$ and $\rho_0$ be given by
Lemma~\ref{lem:unstable-count-jump}. We use the perturbations
in \eqref{si} for system \eqref{aug:original} and the real spaces
$X$ and $Z$ defined in \eqref{basic:spaces}.

We set
$\tau=\tau_*+\eta,$
where $\eta$ is a constant and we consider the following augmented system:
\begin{equation}\label{aug:system}
\begin{cases}
\dot{\eta}=0,\\[2mm]
\begin{aligned}
\partial_t u(t,x)=p\partial_{xx}u(t,x)
-f_{\tau_*+\eta}(x,s(x)+u(t,x),h(x)+v(t,x))
+f_{\tau_*+\eta}(x,s(x),h(x)),
\end{aligned}\\[2mm]
\begin{aligned}
\partial_t v(t,x)=q\partial_{xx}v(t,x)+f_{\tau_*+\eta}(x,s(x)+u(t,x),h(x)+v(t,x))-f_{\tau_*+\eta}(x,s(x),h(x)).
\end{aligned}
\end{cases}
\end{equation}
Here $t>0$ and $0<x<l$. The boundary conditions are
\[
\partial_xu(t,0)=\partial_xu(t,l)
=\partial_xv(t,0)=\partial_xv(t,l)=0.
\]
The mass condition is
\[
\int_0^l (u(t,x)+v(t,x))\,dx=0
\quad\text{for all }t\geq 0.
\]

Let
\[
U(t)=\binom{u(t,\cdot)}{v(t,\cdot)},
\quad
J_*=
\begin{pmatrix}
-\tau_*a & -c_{\tau_*}\\
\tau_*a & c_{\tau_*}
\end{pmatrix}.
\]
For a spatial function $U=(u,v)\in Z$, define the reaction map by
\[
(R_\tau(U))(x)
=\binom{
-f_\tau(x,s(x)+u(x),h(x)+v(x))+f_\tau(x,s(x),h(x))
}{
f_\tau(x,s(x)+u(x),h(x)+v(x))-f_\tau(x,s(x),h(x))
}.
\]
Then the linear operator can be read as 
\[
L_*:=L_{\tau_*},\qquad L_*U=DU+J_*U.
\]

Then system \eqref{aug:system} can be written as
\begin{equation}\label{sys}
\frac{d}{dt}
\binom{\eta}{U}
=
A\binom{\eta}{U}
+
F(\eta,U),
\end{equation}
with 
$$
A\binom{\eta}{U}=\binom{0}{L_*U},\quad F(\eta,U)=\binom{
0
}{
R_{\tau_*+\eta}(U)-J_*U
}.
$$

Thus $\mathcal D(L_*)=Z$ on $X$, as already defined in
Section~\ref{basic:path}.

We consider the spectrum and resolvent of \(L_*\) on
\(X_{\mathbb C}\), with domain \(Z_{\mathbb C}\), and write
\[
\sigma(L_*)=\sigma_u\cup\sigma_c\cup\sigma_s,
\]
where
\[
\begin{aligned}
\sigma_u&:=\{\lambda\in\sigma(L_*):
\operatorname{Re}\lambda>0\},\\
\sigma_c&:=\sigma(L_*)\cap i\mathbb R,\\
\sigma_s&:=\{\lambda\in\sigma(L_*):
\operatorname{Re}\lambda<0\}.
\end{aligned}
\]
By Lemmas~\ref{lem:common-circle} and \ref{finite},
\(\sigma_c\) and \(\sigma_u\) are finite.
Choose a positively oriented, conjugation-invariant simple
closed contour \(\Gamma_c\subset\rho(L_*)\) enclosing
\(\sigma_c\) and no other spectral points, and define
\[
\Pi_c
:=\frac{1}{2\pi i}
\int_{\Gamma_c}(\lambda I-L_*)^{-1}\,d\lambda
\in\mathcal L(X_{\mathbb C},Z_{\mathbb C}).
\]
Since \(L_*\) has real coefficients, \(\Pi_c\) commutes with
complex conjugation and hence maps \(X\) into \(Z\).
We use the same notation for its restriction to \(X\).

Define the real subspaces
\[
X_c:=\Pi_c(X)\subset Z,
\qquad
X_h:=\{U\in X:\Pi_cU=0\},
\qquad
Z_h:=Z\cap X_h.
\]
Then \(X_c\) is finite-dimensional, and
\[
X=X_c\oplus X_h,
\qquad
Z=X_c\oplus Z_h.
\]
The corresponding complex center subspace is
\[
\Pi_c(X_{\mathbb C})=X_c+iX_c.
\]
We equip \(X_c\) and \(X_h\) with the norms inherited from
\(X\), and \(Z_h\) with the norm inherited from \(Z\).

The augmented phase space and operator domain for \eqref{sys} are
\[
\mathbb X:=\mathbb R\times X,
\qquad
\mathbb Z:=\mathbb R\times Z,
\]
with
\[
\|(\eta,U)\|_{\mathbb X}^2
=|\eta|^2+\|U\|_X^2,\quad \|(\eta,U)\|_{\mathbb Z}^2
:=|\eta|^2+\|U\|_{Z}^2,
\quad
A(\eta,U)=(0,L_*U).
\]

We stress that we consider the perturbations \(u(t,x)\) and \(v(t,x)\) in spaces of real-valued functions. When studying the spectrum of the linear operators, we use the corresponding spaces of complex-valued functions. All derivatives of maps between function spaces below are
Fr\'echet derivatives.

\begin{theorem}\label{center}
Under the preceding assumptions, there exist a neighborhood
$\mathcal V$ of $(0,0)$ in $\mathbb R\times X_c$
and a $C^3$ map
\[
\Psi:\mathcal V\longrightarrow Z_h
\]
such that
\[
\Psi(0,0)=0,\qquad D\Psi(0,0)=0,
\]
and
\[
\mathcal M_c
=
\left\{
(\eta,z+\Psi(\eta,z)):
(\eta,z)\in\mathcal V
\right\}
\]
is a finite-dimensional local center manifold of
\eqref{sys} near $(0,0)$. Moreover, for all sufficiently small $\eta$, we have $\Psi(\eta,0)=0.$
\end{theorem}

\begin{proof}
    We verify the hypotheses of Theorem 2.20 in \cite{HaragusIooss} to give the proof.

\textbf{Properties of the operators $A$ and $F(\eta,U)$:} We will prove that $A\in\mathcal L(\mathbb Z,\mathbb X)$ and there exists neighborhood $\mathcal O$ of $(0,0)$ such that $F\in C^3(\mathcal O,\mathbb X)$ with $F(0,0)=0$ and $DF(0,0)=0$.

By the definition in Section~\ref{basic:path}, $L_*:Z\to X$ is
bounded. Hence
\[
\|A(\eta,U)\|_{\mathbb X}=\|L_*U\|_X
\le C\|U\|_Z\le C\|(\eta,U)\|_{\mathbb Z},
\]
so $A\in\mathcal L(\mathbb Z,\mathbb X)$.

For the nonlinear part, due to the definition of $F$, we have $F(0,0)=0$. Next, we choose a small neighborhood $\mathcal O$ of $(0,0)$ and prove that $F\in C^3(\mathcal O,\mathbb X)$ and $DF(0,0)=0$.

Since $s$ and $h$ are smooth and strictly positive on
$[0,l]$, we have
\[
m_0
:=\min_{x\in[0,l]}
\{s(x),h(x)\}>0.
\]
The Sobolev compact embeddings
$
Z\hookrightarrow H^1(0,l;\mathbb R^2)
\hookrightarrow C([0,l];\mathbb R^2)
$
imply that there exists a constant $C>0$ such that
\[
\max\{\|u\|_{L^\infty},\|v\|_{L^\infty}\}
\le C\|U\|_Z,
\quad\text{for every }U=(u,v)\in Z.
\]
Choose $0<\varepsilon<\rho_0$ and $\delta>0$ such that
$C\delta<\frac{m_0}{2}$, and define
\[
\mathcal O
:=(-\varepsilon,\varepsilon)\times B_Z(0,\delta)
\subset\mathbb Z.
\]
 For every $(\eta,U)\in\mathcal O$ and every
$x\in[0,l]$, we have
\[
\begin{aligned}
s(x)+u(x)
&\ge m_0-\|u\|_{L^\infty}
\ge m_0-C\|U\|_Z
>\frac{m_0}{2},\\
h(x)+v(x)
&\ge m_0-\|v\|_{L^\infty}
\ge m_0-C\|U\|_Z
>\frac{m_0}{2}.
\end{aligned}
\]
Consequently,
\[
s(x)+h(x)+u(x)+v(x)>m_0>0.
\]
Thus the denominator appearing in the reaction term is uniformly
bounded away from zero on $\mathcal O$.
Since $H^1(0,l)$ is a Banach algebra, addition, multiplication,
and reciprocal operations define smooth maps on this neighborhood. Therefore,
\[
(\eta,U)\longmapsto R_{\tau_*+\eta}(U)
\]
is a $C^3$ map from $\mathcal O$ into
$H^1(0,l;\mathbb R)^2$.  Since $U\mapsto J_*U$ is bounded
and linear, it follows that
$F\in C^3(\mathcal O,\mathbb X)$: the two components of
both $R_\tau(U)$ and $J_*U$ sum to zero pointwise.

Finally, we claim that $DF(0,0)=0$. The definition of $R_\tau$ gives
$F(\eta,0)=0
$ for small $\eta.$
Thus $F(0,0)=0$ and
$\mathrm D_\eta F(0,0)[\zeta]=0$ for every $\zeta\in\mathbb R$.
For $V=(w_1,w_2)\in Z$, the chain rule
gives
\[
\mathrm D_U R_\tau(0)[V]
=
\binom{
-\partial_S f_\tau(x,s,h)w_1
-\partial_I f_\tau(x,s,h)w_2
}{
\partial_S f_\tau(x,s,h)w_1
+\partial_I f_\tau(x,s,h)w_2
}
=J_\tau V.
\]
Taking $\tau=\tau_*$, we find
\[
\mathrm D_UF(0,0)[V]
=\binom{0}{J_*V-J_*V}=0.
\]
Consequently,  $DF(0,0)=0.$

\textbf{Spectrum of $A$ and Resolvent estimate on the hyperbolic part of $A$.}
By the preceding spectral results, the center spectrum of
\(A=0\oplus L_*\) is \(\{0\}\cup\sigma_c\), which consists of
finitely many eigenvalues of finite algebraic multiplicity,
and the remaining spectrum is separated from the imaginary axis.

Since \(\Pi_c\) commutes with \(L_*\), the restriction
\[
L_h:Z_h\subset X_h\longrightarrow X_h,
\qquad L_hU:=L_*U,
\]
is well defined. The hyperbolic part of \(A\) is identified
with \(L_h\) via \(U\mapsto(0,U)\).

We now work on the complexified spaces. By
\eqref{spec:high-frequency} at \(\tau=\tau_*\),
\[
i\omega\in\rho(L_*),
\qquad
\|(i\omega I-L_*)^{-1}\|_{\mathcal L(X_{\mathbb C})}
\le\frac{2}{|\omega|},
\qquad |\omega|\ge2K+1.
\]
Since \(\Pi_c\) commutes with this resolvent,
\((i\omega I-L_*)^{-1}\) maps
\((X_h)_{\mathbb C}\) into \((Z_h)_{\mathbb C}\).
Consequently,
\[
(i\omega I-L_h)^{-1}
=
\left.(i\omega I-L_*)^{-1}\right|_{(X_h)_{\mathbb C}},
\]
and hence
\[
\|(i\omega I-L_h)^{-1}\|_{\mathcal L((X_h)_{\mathbb C})}
\le\frac{2}{|\omega|},
\qquad |\omega|\ge2K+1.
\]
Thus the spectral and resolvent hypotheses of
Theorem~2.20 in \cite{HaragusIooss} are satisfied.

Therefore, Theorem 2.20 in \cite{HaragusIooss}
yields a $C^3$ reduction map from a neighborhood
of the origin in $\mathbb R\times X_c$
into $\{0\}\times Z_h$.
Writing this map as $(0,\Psi(\eta,z))$,
we obtain
\[
\Psi(0,0)=0,\qquad D\Psi(0,0)=0,
\]
and the local center manifold
\[
\mathcal M_c
=
\left\{
(\eta,z+\Psi(\eta,z)):
(\eta,z)\in\mathcal V
\right\}.
\]

Finally, since $F(\eta,0)=0$ and $A(\eta,0)=0$,
the constant trajectory $(\eta,0)$ is an equilibrium
of the augmented system for every sufficiently
small $\eta$.
It remains in the neighborhood supplied by the
center manifold theorem for all $t\in\mathbb R$.
Hence this equilibrium belongs to $\mathcal M_c$.
Its center projection is $(\eta,0)$, so its graph
representation gives
\[
(\eta,0)=(\eta,\Psi(\eta,0)).
\]
Hence $\Psi(\eta,0)=0$ for all sufficiently small $\eta$.
\end{proof}

In $\mathcal V$, fix $\eta\in(-\delta,\delta)$  and choose $z_0\in X_c$
such that $(\eta,z_0)\in\mathcal V$.
Consider a local classical solution $U(t)$ of
\begin{equation}\label{eqU}
U_t=L_*U+R_{\tau_*+\eta}(U)-J_*U
\end{equation}
with initial value
\[
U(0)=z_0+\Psi(\eta,z_0).
\]
By the local invariance of the center manifold
(Theorem~\ref{center}), as long as the solution remains
in the neighborhood where local invariance holds,
we have
\[
U(t)=z(t)+\Psi(\eta,z(t)),
\qquad
z(t)=\Pi_cU(t)\in X_c,
\qquad
\Psi(\eta,z(t))\in Z_h.
\]

Choose a real basis $e_1(x),...,e_d(x)$ of $X_c$ and write
\[
z(t,x)=\sum_{k=1}^d z_k(t)e_k(x).
\]
Since $\Pi_c$ is a bounded linear operator independent of
time and $\Pi_cL_*=L_*\Pi_c$ on $Z$, applying $\Pi_c$
to \eqref{eqU} gives
\begin{align*}
 \partial_t z(t,x)
=\Pi_cU_t
=L_*z(t,x)+\Pi_c\Bigl[
R_{\tau_*+\eta}\bigl(z(t,x)+\Psi(\eta,z(t,x))\bigr)
-J_*\bigl(z+\Psi(\eta,z)\bigr)
\Bigr].
\end{align*}
These two sides belong to $X_c$.
Expressing it in the basis $e_1(x),...,e_d(x)$, we obtain
the finite-dimensional ODE system for $\eta\in(-\delta,\delta)$:
\[
\dot z_k=g_k(\eta,z_1,\ldots,z_d),
\qquad k=1,\ldots,d.
\]
Here $g=(g_1,\ldots,g_d)$ is $C^3$ and we define $ z=(z_1(t),...,z_d(t))$.
Moreover, since $\Psi(\eta,0)=0$ and
$R_{\tau_*+\eta}(0)=0$, we have
$g(\eta,0)=0.$
We then have
$$
\dot {z}=g(\eta,z),\quad g(\eta,0)=0.
$$

The linearization of the reduced system at the equilibrium is
\[
\dot{ z}=D_{z}g(\eta,0) z,
\quad
B_\eta:=D_{ z}g(\eta,0).
\]
We now compare its eigenvalues with those of the full linearized operator \(L_{\tau_*+\eta}\).

\[
B_\eta:=D_zg(\eta,0),\qquad
G_\eta z:=z+D_z\Psi(\eta,0)z,\qquad
E_\eta:=G_\eta X_c.
\]

\begin{proposition}\label{prop:reduced-spectrum}
There exist \(\varepsilon,\rho>0\) such that, for every \(\eta\in[-\rho,\rho]\),
\[
\sigma(B_\eta)
=
\sigma(L_{\tau_*+\eta})
\cap
\bigl\{\lambda\in\mathbb C:
|\operatorname{Re}\lambda|<\varepsilon\bigr\},
\]
with equality of the corresponding algebraic multiplicities. In particular, the remaining spectrum of \(L_{\tau_*+\eta}\) stays uniformly separated from the imaginary axis for \(\eta\in[-\rho,\rho]\).
\end{proposition}

\begin{proof}
Fix \(\eta\), and write the center-manifold parametrization
as a function of the real coordinates \(z=(z_1,\ldots,z_d)\):
\[
U(z)
=\sum_{k=1}^d z_ke_k
+\Psi\!\left(\eta,\sum_{k=1}^d z_ke_k\right).
\]
Along a solution of the reduced system, the chain rule and
the equation for \(U\) give
\[
\sum_{k=1}^d
\frac{\partial U}{\partial z_k}(z)\,g_k(\eta,z)
=L_*U(z)+R_{\tau_*+\eta}(U(z))-J_*U(z).
\]
This identity holds for every sufficiently small real \(z\),
since such a \(z\) can be chosen as the initial center coordinate.

Set
\[
\phi_j:=\frac{\partial U}{\partial z_j}(0)
=e_j+D_z\Psi(\eta,0)e_j,
\qquad j=1,\ldots,d.
\]
These are real-valued functions in \(Z\).
Since \(\Psi\) is \(C^3\) with values in \(Z_h\) and
\(L_*:Z\to X\) is bounded, we may differentiate the preceding
identity with respect to \(z_j\) at \(z=0\).
Using \(U(0)=0\),
\(D_UR_{\tau_*+\eta}(0)=J_{\tau_*+\eta}\),
and \(g(\eta,0)=0\), we obtain
\[
\begin{aligned}
L_{\tau_*+\eta}\phi_j
=\sum_{k=1}^d
\frac{\partial^2U}{\partial z_j\partial z_k}(0)\,
g_k(\eta,0)
+\sum_{k=1}^d
\phi_k\frac{\partial g_k}{\partial z_j}(\eta,0)=\sum_{k=1}^d(B_\eta)_{kj}\phi_k.
\end{aligned}
\]
Moreover, since \(D_z\Psi(\eta,0)e_j\in Z_h\),
\[
\Pi_c\phi_j=e_j.
\]
Thus \(\phi_1,\ldots,\phi_d\) are linearly independent.
Their real span is invariant under \(L_{\tau_*+\eta}\),
and the matrix of its restriction in this basis is exactly
\(B_\eta\). Consequently,
\[
\sigma(B_\eta)\subset\sigma(L_{\tau_*+\eta}),
\]
and each eigenvalue has algebraic multiplicity for \(B_\eta\)
no larger than its algebraic multiplicity for the full operator.

At \(\eta=0\), we have \(D_z\Psi(0,0)=0\).
The preceding matrix representation therefore shows that
\(B_0\) is the matrix of \(L_*|_{X_c}\), so
\[
\sigma(B_0)=\sigma(L_*)\cap i\mathbb R.
\]
Choose small disjoint disks enclosing these eigenvalues
and no other spectral points of \(L_*\).
Their total algebraic multiplicity is
\(d=\dim_{\mathbb R}X_c\).
By resolvent continuity, their boundaries remain in the
resolvent set of \(L_{\tau_*+\eta}\), and the total algebraic
multiplicity inside them remains \(d\), for sufficiently
small \(|\eta|\).

Continuity of \(B_\eta\) ensures that all its \(d\) eigenvalues,
counted with algebraic multiplicity, remain in these disks.
The multiplicity inequality proved above, together with
equality of the total multiplicities, shows that the two
groups coincide, including algebraic multiplicities.

Finally, the remaining eigenvalues inside the fixed common
circle \(\Gamma_R\) have nonzero real parts at \(\eta=0\).
By resolvent continuity, they remain uniformly separated
from the imaginary axis for sufficiently small \(|\eta|\).
The preceding spectral bounds also keep all spectrum outside
\(\Gamma_R\) in a fixed left half-plane.
Thus there exist \(\delta,\rho>0\) such that every eigenvalue
outside the chosen disks satisfies
\[
|\operatorname{Re}\lambda|\ge\delta,
\qquad |\eta|\le\rho.
\]
Choose \(0<\varepsilon<\delta\).
Since \(\sigma(B_0)\subset i\mathbb R\), we may reduce
\(\rho\) further so that
\[
\sigma(B_\eta)\subset
\{\lambda:|\operatorname{Re}\lambda|<\varepsilon\},
\qquad |\eta|\le\rho.
\]
This proves the asserted spectral equality, including
algebraic multiplicities, and the uniform separation of
the remaining spectrum.

Furthermore, the number of unstable eigenvalues outside
these disks is constant. Hence
\[
n_+(L_{\tau_*+\eta})=n_+(B_\eta)+N,
\qquad |\eta|\le\rho,
\]
where \(N\) is independent of \(\eta\).
\end{proof}

\subsection{Periodic solutions of the reduced system.}

\begin{proposition}
\label{prop:small-periodic-solutions}
Under the preceding assumptions, the reduced system
\[
\dot z=g(\eta,z)
\]
has nonconstant periodic solutions
\(z_n:\mathbb R\to\mathbb R^d\), with parameters
\(\eta_n\to0\) and periods \(T_n\to T_*\in(0,\infty)\), such that
\[
\sup_{t\in\mathbb R}|z_n(t)|\longrightarrow0.
\]
\end{proposition}

\begin{proof}
By Lemma~\ref{lem:unstable-count-jump} and
Proposition~\ref{prop:reduced-spectrum}, we can choose
$\rho>0$ sufficiently small so that $B_\eta$ is
invertible for $|\eta|\le\rho$, and only $B_0$ has
purely imaginary eigenvalues, and
\[
n_+(B_{-\rho})\ne n_+(B_\rho).
\]

Using a smooth cutoff in $z$, extend $g$ to a
$C^2$ map on $[-\rho,\rho]\times\mathbb R^d$,
keeping it unchanged in a neighborhood of
$[-\rho,\rho]\times\{0\}$.
We retain the notation $g$ for this extension.
The regularity of $g$, the identity $g(\eta,0)=0$,
and the invertibility of $B_\eta$ verify
conditions (P0)--(P3) of \cite{Hooton}.

To verify the spectral trapping condition (P4),
choose
$R>\sup_{|\eta|\le\rho}\|B_\eta\|$
and set
\[
\mathcal P
=[-\rho,\rho]\times[-R,R]\times[0,R],
\]
with coordinates $(\eta,\omega,r)$.
All eigenvalues of $B_\eta$ lie in $\{|\lambda|<R\}$.
Consequently,
\[
\det\bigl((r+i\omega)I-B_\eta\bigr)\ne0
\]
on the faces $|\omega|=R$ and $r=R$.
The two end faces $\eta=\pm\rho$ contain exactly
$n_+(B_{\pm\rho})$ characteristic roots,
counted with algebraic multiplicity.
We have $n_+(B_{\rho})\neq n_+(B_{-\rho})$, so (P4) holds.
The isolated-center condition also holds because
$\eta=0$ is the only parameter with purely
imaginary eigenvalues.

Extend \(g\) as a \(C^3\) map in \(z\), keeping it unchanged
near \(z=0\), and apply that theorem.
It gives small-amplitude nonconstant periodic solutions
with periods in a compact subinterval of \((0,\infty)\).
Since \(\eta=0\) is the only parameter with purely imaginary
eigenvalues, we obtain the stated sequences after passing
to a subsequence.
Their amplitudes tend to zero, so they also solve the
original reduced system for all sufficiently large \(n\).
\end{proof}

\subsection{Periodic solutions of the original system.}

\textbf{Proof of Theorem \ref{thm:pathp}:}
Choose \(p_0\in(0,1)\),
\(q=1\), and smooth coefficients satisfying the hypotheses
of Theorem~\ref{thm:pathp}.
Let \((\beta_\tau,\gamma_\tau)\), \(\tau\in[1,\tau_1]\),
be the resulting path, and let \((s,h)\) be its common
positive equilibrium. Set
\[
M_0:=\int_0^2(s+h)\,dx>0.
\]

By Proposition~\ref{prop:small-periodic-solutions},
there are \(\eta_n\to0\) and nonconstant periodic solutions
\(z_n=(z_{n,1},\ldots,z_{n,d})\) of the reduced system,
with periods \(T_n\to T_*\in(0,\infty)\), such that
\(\sup_t|z_n(t)|\to0\).
Define
\[
\tau_n=\tau_*+\eta_n,\qquad
U_n(t)=\sum_{k=1}^d z_{n,k}(t)e_k
+\Psi\!\left(\eta_n,\sum_{k=1}^d z_{n,k}(t)e_k\right).
\]
Since \(\Psi\) is \(C^3\) with values in \(Z_h\) and
\(\Psi(\eta,0)=0\), we have
\[
\sup_{t\in\mathbb R}\|U_n(t)\|_Z
\le C\sup_{t\in\mathbb R}|z_n(t)|
\longrightarrow0.
\]
Thus, for all sufficiently large \(n\), the whole orbit
lies in the neighborhood where the center-manifold
equation agrees with \eqref{eqU}.
The center-manifold invariance equation then shows that
\(U_n\) solves \eqref{eqU} at \(\eta=\eta_n\).
Since \((s,h)\) is an equilibrium at \(\tau=\tau_n\),
\[
(S_n,I_n):=(s,h)+U_n
\]
solves the original system \eqref{eq:model}.

The graph representation shows that \(U_n\) is
\(T_n\)-periodic. Moreover,
\[
\Pi_cU_n(t)=\sum_{k=1}^d z_{n,k}(t)e_k,
\]
so \(U_n\) is nonstationary.
This identity and the graph representation also show
that \(U_n\) and \(z_n\) have exactly the same periods.

Since \(z_n\) is \(C^1\), \(e_k\in Z\), and
\(\Psi\) is \(C^3\) with values in \(Z_h\),
we have \(U_n\in C^1(\mathbb R;Z)\).
Together with
\[
Z\subset H^3(0,2;\mathbb R^2)
\hookrightarrow C^2([0,2];\mathbb R^2),
\]
this shows that \((S_n,I_n)\) is a classical solution.
Moreover we have
\[
\int_0^2\bigl(S_n(t,x)+I_n(t,x)\bigr)\,dx=M_0
\qquad(t\in\mathbb R).
\]
Also, \(U_n\to0\) uniformly in \(t\) and \(x\).
Since \(s,h>0\) on \([0,2]\), we have \(S_n,I_n>0\)
for all sufficiently large \(n\).

Choose one such index \(n_0\) with
\(\tau_{n_0}\in(1,\tau_1)\), and fix
\[
\widehat\beta:=\beta_{\tau_{n_0}},\qquad
\widehat\gamma:=\gamma_{\tau_{n_0}},\qquad
(S^0,I^0):=(S_{n_0},I_{n_0}).
\]
The path construction gives
\(\widehat\beta>\widehat\gamma>0\).
These coefficients are now fixed.

For every \(c>0\),
\[
\widehat\beta\frac{(cS)(cI)}{cS+cI}
-\widehat\gamma(cI)
=
c\left(\widehat\beta\frac{SI}{S+I}
-\widehat\gamma I\right).
\]
Hence, by the linearity of diffusion, for every \(M>0\),
\[
(S^{(M)},I^{(M)})
:=\frac{M}{M_0}(S^0,I^0)
\]
is a positive nonstationary classical periodic solution
of the same system, with total mass \(M\).
Likewise, \(\frac{M}{M_0}(s,h)\) is a positive equilibrium
of mass \(M\).

It remains to prove that this positive equilibrium is unique.
Let \((\bar S,\bar I)\) be any positive equilibrium of
mass \(M\). Adding its equations and using the Neumann
conditions gives
\[
(p_0\bar S+\bar I)''=0,\qquad
p_0\bar S+\bar I=\kappa
\]
for some constant \(\kappa>0\).
Set \(w=\bar I/\kappa\). Then
\[
0<w<1,\qquad
\bar S=\frac{\kappa}{p_0}(1-w),\qquad
\bar I=\kappa w.
\]
The infected equation becomes
\[
w''+w\,a(x,w)=0,\qquad w'(0)=w'(2)=0,
\]
where
\[
a(x,w)=
\widehat\beta(x)\frac{1-w}{1-(1-p_0)w}
-\widehat\gamma(x).
\]
We have
\[
\partial_w a(x,w)
=
-\frac{p_0\widehat\beta(x)}
       {[1-(1-p_0)w]^2}<0.
\]

Let \(w_1,w_2\) be two such solutions, and set
\[
\theta=\max_{[0,2]}\frac{w_1}{w_2}.
\]
If \(\theta>1\), then \(v=\theta w_2-w_1\ge0\)
and \(v(x_0)=0\) for some \(x_0\in[0,2]\).
Thus \(v''(x_0)\ge0\).
If \(x_0\) is an endpoint, the same conclusion follows
from \(v'(x_0)=0\) and the one-sided Taylor expansion.
On the other hand, the equations for \(w_1,w_2\) give
\[
v''(x_0)
=
\theta w_2(x_0)
\bigl[a(x_0,w_1(x_0))-a(x_0,w_2(x_0))\bigr]<0,
\]
because \(w_1(x_0)=\theta w_2(x_0)>w_2(x_0)\)
and \(a\) is strictly decreasing in \(w\).
This contradiction gives \(w_1\le w_2\).
Interchanging \(w_1,w_2\) gives \(w_1=w_2\).

Finally, the mass condition determines \(\kappa\) uniquely:
\[
\kappa=
\frac{M}{\displaystyle
\int_0^2\left(\frac{1-w(x)}{p_0}+w(x)\right)\,dx}.
\]
Therefore the positive equilibrium is unique for each
prescribed mass \(M>0\).

\section{Construction of an unstable positive equilibrium}
\label{seed:construction}

We construct the unstable starting equilibrium needed for the parameter
path. The construction is inverse: we first prescribe an equilibrium
profile and then choose the transmission and recovery rates so that this
profile is an exact equilibrium. A special choice of the profile produces
an unstable eigenvalue.

\begin{proposition}
\label{prop:unstable-seed}
There exist $p\in(0,1)$ and $\beta,\gamma\in C^\infty([0,2])$ satisfying
$\beta(x)>\gamma(x)>0$ on $[0,2]$ such that the Neumann problem
\begin{equation}\label{seed:model}
\begin{cases}
S_t=pS_{xx}-f(x,S,I),\\
I_t=I_{xx}+f(x,S,I),
\end{cases}
\qquad
f(x,S,I)=\beta(x)\dfrac{SI}{S+I}-\gamma(x)I
\end{equation}
has a strictly positive equilibrium whose linearization, restricted to
perturbations of zero total mass, has an eigenvalue with strictly positive
real part. The same coefficients admit such an equilibrium for every
prescribed positive total mass. In each mass class the positive equilibrium
is unique.
\end{proposition}

There are three successive choices of parameters. We first choose
$\varepsilon>0$ small and fix it. We then choose $p>0$ small, with
$\delta=p^4$. Finally, with these parameters fixed, we smooth the
coefficients using a sufficiently small parameter $\kappa>0$.
The intermediate profile may vanish, but the equilibrium used in
\eqref{seed:model} is strictly positive everywhere.

All unspecified integrals below are over $(0,2)$. We use the complex
$L^2$ inner product $\langle v,w\rangle=\int v\overline w\,dx$, linear in
the first variable, and write
\[
H_N^2=\{v\in H^2(0,2;\mathbb C):v'(0)=v'(2)=0\},
\qquad Av=-v'',\quad \mathcal D(A)=H_N^2.
\]

\subsection{The inverse construction of the coefficients}

At a Neumann equilibrium $(s,h)$, adding the two equations gives
$(ps+h)''=0$, hence $ps+h$ is constant. We normalize this constant to one.
Thus, for any smooth profile $h$ with
\[
0<h<1,\qquad h'(0)=h'(2)=0,
\]
we set
\begin{equation}\label{seed:inverse-variables}
X=1-h,\qquad s=\frac Xp,\qquad
r=\frac{h}{s+h}=\frac{ph}{X+ph},\qquad K=\frac{h''}{h}.
\end{equation}
The infected equilibrium equation is equivalent to
$\beta(1-r)-\gamma=-K$. It is useful to parametrize the remaining freedom
by the function $t=\beta r$. For a chosen $t$, define
\begin{equation}\label{seed:inverse-coefficients}
\beta=\frac tr,\qquad \gamma=\frac{t(1-r)}r+K.
\end{equation}
These formulas make $(s,h)$ an exact equilibrium. Moreover,
\begin{equation}\label{seed:two-inequalities}
\beta-\gamma=t-K,\qquad \gamma=\frac{tX}{ph}+K.
\end{equation}
They show precisely what must be imposed on $t$ to obtain
$\beta>\gamma>0$.

For a real-valued function $b$, write $b_+=\max\{b,0\}$ and
$b_-=\max\{-b,0\}$. We choose
\begin{equation}\label{seed:choice-t}
t=K_++pK_-\frac hX+p^2.
\end{equation}
Since $K=K_+-K_-$, direct substitution gives
\begin{equation}\label{seed:coefficient-positivity}
\begin{aligned}
\beta-\gamma&=K_-\left(1+\frac{ph}{X}\right)+p^2>0,\\
\gamma&=K_+\left(1+\frac X{ph}\right)+\frac{pX}{h}>0.
\end{aligned}
\end{equation}
The positive and negative parts are used separately to enforce the two
inequalities in \eqref{seed:two-inequalities}. The small term $p^2$ makes
both inequalities strict, including where $K=0$.

This also explains the equilibrium calculation without any approximation:
\[
f(x,s,h)=h\,[\beta(1-r)-\gamma]=-Kh=-h'',
\qquad ps''=-h''.
\]
Consequently both stationary equations hold. The functions in
\eqref{seed:inverse-coefficients} are initially Lipschitz; the final
smoothing will preserve this exact equilibrium relation.

\subsection{A quartic profile and its positive approximation}

Fix $\varepsilon\in(0,1/2)$, to be made small later, and define
\begin{equation}\label{seed:hzero}
a=(1-\varepsilon)^{1/4},\qquad
h_0(x)=(1-\varepsilon-x^4)_+,\qquad
Z_\varepsilon=\int_0^a h_0^2\,dx.
\end{equation}
The profile is nearly one near $x=0$ and vanishes on $[a,2]$.
In particular,
\begin{equation}\label{seed:quartic-denominator}
1-h_0=\varepsilon+x^4\quad\hbox{on }(0,a).
\end{equation}
This small denominator will concentrate the coupling near $x=0$.
The quartic power has a second role: the resolvent calculation below
produces a term proportional to $\mu^{-2}$, leading to a cubic
characteristic equation with roots in the right half-plane.

We next replace $h_0$ by a smooth, strictly positive profile. Choose
$\Theta\in C^\infty(\mathbb R;[0,1])$ such that
\[
\Theta=0\ \hbox{on }(-\infty,0],\qquad
\Theta=1\ \hbox{on }[2,\infty),\qquad
\Theta(2-z)=1-\Theta(z),\qquad \Theta'>0\ \hbox{on }(0,2).
\]
Define the smooth convex function
\[
\Psi(z)=
\begin{cases}
1,&z\le0,\\
1+\displaystyle\int_0^z\Theta(t)\,dt,&0<z<2,\\
z,&z\ge2.
\end{cases}
\]
The symmetry of $\Theta$ gives $\int_0^2\Theta=1$, so the pieces join
smoothly. For sufficiently small $\delta>0$, set
\begin{equation}\label{seed:hdelta}
h_\delta(x)=\delta\Psi\!\left(\frac{1-\varepsilon-x^4}{\delta}\right).
\end{equation}
Then
\begin{equation}\label{seed:profile-basic}
\delta\le h_\delta\le1-\varepsilon,\qquad
1-h_\delta\ge\varepsilon,\qquad
h_\delta'(0)=h_\delta'(2)=0.
\end{equation}
It equals $1-\varepsilon-x^4$ where this function is at least $2\delta$,
and equals $\delta$ where it is nonpositive. The transition interval has
length $O(\delta)$.

Here and below, estimates involving $\delta$ may depend on the fixed
$\varepsilon$. Writing $g(x)=1-\varepsilon-x^4$, we have
\[
h_\delta'=\Theta(g/\delta)g',\qquad
h_\delta''=\delta^{-1}\Theta'(g/\delta)(g')^2+\Theta(g/\delta)g''.
\]
Since $g''=-12x^2$, these formulas imply
\begin{equation}\label{seed:profile-estimates}
\begin{gathered}
\|h_\delta-h_0\|_\infty=O(\delta),\qquad
\|h_\delta'-h_0'\|_2=O(\delta^{1/2}),\\
\|(h_\delta'')_-\|_\infty\le C,\qquad
\operatorname{supp}(h_\delta'')_+\subset\{0\le g\le2\delta\},\qquad
\int h_\delta(h_\delta'')_+\,dx\le C\delta.
\end{gathered}
\end{equation}
For the last estimate, use $h_\delta\le2\delta$ in the transition
interval and $\int(h_\delta'')_+\,dx\le C$.

From now on, let $\delta=p^4$, use $h=h_\delta$ in
\eqref{seed:inverse-variables}, and choose $t,\beta,\gamma$ by
\eqref{seed:choice-t} and \eqref{seed:inverse-coefficients}.

\subsection{The exact eigenvalue equations and the limiting scalar equation}

The derivatives of the reaction term at the prescribed equilibrium are
\begin{equation}\label{seed:reaction-derivatives}
a_1=f_S(x,s,h)=tr,\qquad
b_1=f_I(x,s,h)=-K-t(1-r).
\end{equation}
Thus a spectral eigenfunction $(\varphi,\psi)$ satisfies
\begin{equation}\label{seed:original-eigenproblem}
\lambda\varphi=p\varphi''-a_1\varphi-b_1\psi,\qquad
\lambda\psi=\psi''+a_1\varphi+b_1\psi,
\end{equation}
with Neumann boundary conditions. Introduce
\begin{equation}\label{seed:eigenvalue-scaling}
\lambda=p\mu,\qquad \xi=p(\varphi+\psi).
\end{equation}
Adding the equations in \eqref{seed:original-eigenproblem} and substituting
$\varphi=\xi/p-\psi$ gives the exact system
\begin{equation}\label{seed:exacteigen}
\mu\xi=-A\xi-(1-p)A\psi,\qquad
(L_p+p\mu)\psi=\frac{a_1}{p}\xi,
\end{equation}
where
\begin{equation}\label{seed:operators}
\begin{gathered}
M_\delta=A+2K_+-K_-,\qquad
D_\delta=\frac{(h'')_-}{X},\qquad
L_p=M_\delta+pD_\delta+p^2,\\
\frac{a_1}{p}=B_{\delta,p}+p\alpha_{\delta,p},\qquad
B_{\delta,p}=\frac{(h'')_+}{X+ph},\qquad
\alpha_{\delta,p}=
\frac{(h'')_-h}{X(X+ph)}+\frac{ph}{X+ph}.
\end{gathered}
\end{equation}
In particular, no small terms have been discarded in
\eqref{seed:exacteigen}.

For $\operatorname{Re}\mu>0$, define the bounded operator
\begin{equation}\label{seed:Tdefinition}
T_\mu=I-\mu(\mu+A)^{-1}.
\end{equation}
Equivalently, if $y\in H_N^2$ solves $\mu y-y''=f$, then
$T_\mu f=f-\mu y$. For $f\in H^1$, its value $w=T_\mu f$ is
characterized by
\begin{equation}\label{seed:Tweak}
\mu\int w\overline\chi\,dx+\int w'\overline{\chi'}\,dx
=\int f'\overline{\chi'}\,dx
\qquad(\chi\in H^1).
\end{equation}
The Neumann cosine expansion shows that $T_\mu$ multiplies the $n$th
coefficient by $\rho_n/(\mu+\rho_n)$, where $\rho_n=(n\pi/2)^2$.
It is analytic in $\mu$ and uniformly bounded on $H^1$ on each compact
subset of the right half-plane. The first equation in
\eqref{seed:exacteigen} becomes
\begin{equation}\label{seed:first-elimination}
\xi=-(1-p)T_\mu\psi.
\end{equation}

The choice of $t$ gives $M_\delta$ just one eigenvalue near zero; its
corresponding eigenfunction approaches a multiple of $h_0$.
We shall prove that, uniformly for $\mu$ in compact subsets of
$\{\operatorname{Re}\mu>0\}$,
\begin{equation}\label{seed:target-response}
(L_p+p\mu)^{-1}\frac{a_1}{p}\xi
\longrightarrow
\frac{h_0}{\mu+c_\varepsilon}\int_0^a k_\varepsilon\xi\,dx
\quad\hbox{in operator norm on }H^1,
\end{equation}
where
\begin{equation}\label{seed:kc}
c_\varepsilon=\frac1{Z_\varepsilon}
\int_0^a\frac{12x^2h_0^2}{1-h_0}\,dx,\qquad
k_\varepsilon(x)=
\frac{12x^2h_0(x)^2}{Z_\varepsilon(1-h_0(x))^2}\,
\mathbf1_{(0,a)}(x).
\end{equation}
Both are nonnegative. Consequently, the limiting equations are
\[
\xi=-T_\mu h_0\,z,\qquad
(\mu+c_\varepsilon)z=\int_0^a k_\varepsilon\xi\,dx.
\]
A nonzero solution therefore requires
\begin{equation}\label{seed:F}
F_\varepsilon(\mu):=
\mu+c_\varepsilon+\int_0^a k_\varepsilon T_\mu h_0\,dx=0.
\end{equation}
We first show why this scalar equation is unstable. We then justify
\eqref{seed:target-response} and transfer the root to the original system.

\subsection{A right half-plane root of the scalar equation}

\begin{lemma}\label{lem:seed-scalar-root}
For all sufficiently small $\varepsilon>0$, $F_\varepsilon$ has a simple
zero $\mu_\varepsilon$ with $\operatorname{Re}\mu_\varepsilon>0$.
\end{lemma}

\begin{proof}
The derivative of $h_0$ jumps from $-4a^3$ to zero at $a$. In particular,
for every $\chi\in H^1$,
\begin{equation}\label{seed:jumpweak}
\int h_0'\overline{\chi'}\,dx
=\int_0^a12x^2\overline\chi\,dx-4a^3\overline{\chi(a)}.
\end{equation}
The point contribution in this identity must be retained.
Let $w=T_\mu h_0$ and choose $\sigma=\sqrt\mu$ with
$\operatorname{Re}\sigma>0$. Equations \eqref{seed:Tweak} and
\eqref{seed:jumpweak} give
\[
\mu w-w''=
\begin{cases}12x^2,&0<x<a,\\0,&a<x<2,\end{cases}
\qquad w'(0)=w'(2)=0,
\]
together with continuity of $w$ and the jump condition
$w'(a+)-w'(a-)=4a^3$. Solving these two constant-coefficient equations
gives, for $0\le x\le a$,
\begin{equation}\label{seed:wformula}
w(x)=\frac{12x^2}{\mu}+\frac{24}{\mu^2}
+C_\mu\cosh(\sigma x),
\end{equation}
where
\begin{equation}\label{seed:Cformula}
C_\mu=-\frac{
(4a^3+24a/\mu)\cosh(\sigma(2-a))
+\sigma(12a^2/\mu+24/\mu^2)\sinh(\sigma(2-a))
}{\sigma\sinh(2\sigma)}.
\end{equation}
The formula on $(a,2)$ is a multiple of $\cosh(\sigma(2-x))$.
Since $\operatorname{Re}\sigma\ge |\mu|^{1/2}/\sqrt2$ in the right
half-plane, the exponential formulas for the hyperbolic functions imply
\begin{equation}\label{seed:werror}
|C_\mu\cosh(\sigma x)|
\le C|\mu|^{-1/2}e^{-c\sqrt{|\mu|}(a-x)}
\qquad(0\le x\le a)
\end{equation}
for sufficiently large $|\mu|$. The constants can be chosen independently
of small $\varepsilon$.

Set $K_j(\varepsilon)=\int_0^a x^jk_\varepsilon(x)\,dx$ for $j=0,2$.
Polynomial integration and the substitution $x=\varepsilon^{1/4}y$ give
\begin{equation}\label{seed:moments}
\begin{gathered}
Z_\varepsilon=\frac{32}{45}(1-\varepsilon)^{9/4}
\longrightarrow Z_0=\frac{32}{45},\\
K_0(\varepsilon)\sim
\frac{3\pi\sqrt2}{4Z_0}\varepsilon^{-5/4},\qquad
K_2(\varepsilon)=O(\varepsilon^{-3/4}),\qquad
c_\varepsilon=O(\varepsilon^{-1/4}).
\end{gathered}
\end{equation}
For example, the leading constant in $K_0$ follows from
\[
\int_0^\infty\frac{y^2}{(1+y^4)^2}\,dy=\frac{\pi\sqrt2}{16}.
\]
The other estimates use the integrability of
$y^4/(1+y^4)^2$ and $y^2/(1+y^4)$ on $(0,\infty)$.
Dominated convergence applies after extending the transformed integrands
by zero outside their integration intervals.

On $[a/2,a]$, $k_\varepsilon$ is bounded independently of small
$\varepsilon$. Splitting the error integral at $a/2$ and applying
\eqref{seed:werror}, we obtain
\begin{equation}\label{seed:integratederror}
\left|\int_0^a k_\varepsilon C_\mu\cosh(\sigma x)\,dx\right|
\le CK_0(\varepsilon)|\mu|^{-1/2}e^{-c\sqrt{|\mu|}}
+C|\mu|^{-1}.
\end{equation}
Hence
\begin{equation}\label{seed:Fexpanded}
F_\varepsilon(\mu)=\mu+c_\varepsilon
+\frac{12K_2(\varepsilon)}{\mu}
+\frac{24K_0(\varepsilon)}{\mu^2}+E_\varepsilon(\mu),
\end{equation}
with $E_\varepsilon$ bounded by \eqref{seed:integratederror}.
Balancing $\mu$ against $K_0/\mu^2$ suggests
$\mu=\varepsilon^{-5/12}z$. Indeed, uniformly on every compact subset
of $\{\operatorname{Re}z>0\}$,
\begin{equation}\label{seed:scaledF}
\varepsilon^{5/12}F_\varepsilon(\varepsilon^{-5/12}z)
\longrightarrow G(z):=z+\frac{C_*}{z^2},
\qquad C_*=\frac{405\pi\sqrt2}{16}.
\end{equation}
The normalized $K_2$ and $c_\varepsilon$ terms are respectively
$O(\varepsilon^{1/12})$ and $O(\varepsilon^{1/6})$.
The normalized error is bounded on such a compact set by
\[
C\varepsilon^{-5/8}e^{-c\varepsilon^{-5/24}}
+C\varepsilon^{5/6},
\]
which tends to zero.

The equation $G(z)=0$ is $z^3=-C_*$. In particular,
\[
z_*=C_*^{1/3}e^{i\pi/3},\qquad
\operatorname{Re}z_*>0,\qquad G'(z_*)=3.
\]
Choose a small closed disk about $z_*$ contained in the right half-plane
and containing no other zero of $G$. By the uniform convergence in
\eqref{seed:scaledF}, Rouch\'e's theorem gives exactly one zero, counted
with multiplicity, of the normalized $F_\varepsilon$ in this disk.
It is therefore simple. Equivalently,
\begin{equation}\label{seed:muepsilon}
\mu_\varepsilon=\varepsilon^{-5/12}z_\varepsilon,\qquad
z_\varepsilon\longrightarrow z_*,\qquad
\operatorname{Re}\mu_\varepsilon>0.
\end{equation}
\end{proof}

We now fix such an $\varepsilon$. All constants in the remaining
small-$p$ estimates may depend on this choice.

\subsection{Why the infected component has only one surviving direction}

The following energy argument justifies the single function direction in
\eqref{seed:target-response}.

\begin{lemma}\label{lem:seed-one-mode}
For all sufficiently small $\delta>0$, the first two Neumann eigenvalues
of $M_\delta$ satisfy
\begin{equation}\label{seed:gapestimate}
0\le\nu_\delta:=\lambda_1(M_\delta)\le C\delta,
\qquad \lambda_2(M_\delta)\ge g_*>0,
\end{equation}
where $g_*$ is independent of $\delta$. Its first eigenfunction can be
chosen real, strictly positive and normalized by
$\|\phi_\delta\|_2=1$. With this choice,
\begin{equation}\label{seed:phiconvergence}
\phi_\delta\longrightarrow
\phi_0:=\frac{h_0}{\sqrt{Z_\varepsilon}}
\quad\hbox{strongly in }H^1(0,2).
\end{equation}
\end{lemma}

\begin{proof}
The quadratic form of $M_\delta$ is
\begin{equation}\label{seed:energy}
\begin{aligned}
E_\delta[v]
&=\int|v'|^2\,dx+\int(2K_+-K_-)|v|^2\,dx\\
&=\int h_\delta^2\left|\left(v/h_\delta\right)'\right|^2dx
+\int K_+|v|^2\,dx\ge0.
\end{aligned}
\end{equation}
The identity follows by expanding the square and integrating by parts;
the endpoint terms vanish because $h_\delta'=0$ there.
Testing with $h_\delta$ and using \eqref{seed:profile-estimates} gives
$\nu_\delta\le C\delta$, since $\|h_\delta\|_2$ is bounded away from zero.

We record the compactness estimates needed for the second eigenvalue.
The formula for $h_\delta''$ provides an interval $J_\delta$ of length
at least $c\delta$, at distance at most $C\delta$ from $a$, on which
$K_+\ge c\delta^{-2}$. Also,
\begin{equation}\label{seed:negative-potential}
K_-(x)\le\frac{C}{a-x+\delta}\quad(0<x<a),\qquad
K_-=0\quad(a<x<2).
\end{equation}
To obtain the first assertion, restrict $g/\delta$ to a fixed closed
subinterval of $(0,2)$ on which $\Theta'$ is bounded below. The second
uses $(h_\delta'')_-\le C$ and
$h_\delta\ge c(a-x+\delta)$ on $(0,a)$.

Choose $x_\delta\in J_\delta$ with $|v(x_\delta)|^2$ no larger than
its average on $J_\delta$. The lower bound for $K_+$ and the fundamental
theorem of calculus give
\begin{equation}\label{seed:trace}
|v(a)|^2\le C\delta
\left(\|v'\|_2^2+\int K_+|v|^2\,dx\right).
\end{equation}
For a fixed small $r>0$, use
\[
|v(a-t)|^2\le2|v(a)|^2+2t\int_{a-t}^a|v'|^2\,dx
\]
in $(a-r,a)$, and bound the potential on $(0,a-r)$. This yields
\[
\int K_-|v|^2\,dx
\le C\log(1/\delta)|v(a)|^2
+Cr\|v'\|_2^2+C_r\|v\|_2^2.
\]
Inserting \eqref{seed:trace}, first choosing $r$ small and then $\delta$
small, we can absorb the derivative and positive-potential terms.
Consequently
\begin{equation}\label{seed:uniformenergy}
E_\delta[v]+C\|v\|_2^2
\ge c\left(\|v\|_{H^1}^2+\int K_+|v|^2\,dx\right).
\end{equation}

The limiting form is
\begin{equation}\label{seed:limitenergy}
E_0[v]=\int_0^2|v'|^2\,dx
+\int_0^a\frac{h_0''}{h_0}|v|^2\,dx,
\qquad v\in H^1(0,2),\quad v(a)=0.
\end{equation}
Its potential term is integrable because
\begin{equation}\label{seed:zerotrace}
|v(x)|^2\le(a-x)\int_x^a|v'|^2\,dy\quad(x<a).
\end{equation}
If $\delta_n\downarrow0$ and
$E_{\delta_n}[v_n]+\|v_n\|_2^2$ is bounded, then
\eqref{seed:uniformenergy} gives, after passing to a subsequence, weak
$H^1$ convergence and strong $L^2$ convergence to some $v$.
One-dimensional compactness also gives uniform convergence, and
\eqref{seed:trace} implies $v(a)=0$.
For the negative potential near $a$, the preceding estimates give
\begin{equation}\label{seed:tail}
\int_{a-r}^a K_-|v_n|^2\,dx
\le Cr\|v_n'\|_2^2
+C\delta_n\log(1/\delta_n)
\bigl(E_{\delta_n}[v_n]+C\|v_n\|_2^2\bigr).
\end{equation}
First let $n\to\infty$ and then $r\downarrow0$. Away from $a$, the
negative potentials converge uniformly. Thus their integrals converge
to $\int_0^a(-h_0''/h_0)|v|^2\,dx$. Dropping the nonnegative positive
potential and using weak lower semicontinuity gives
\begin{equation}\label{seed:lowerlimit}
E_0[v]\le\liminf_{n\to\infty}E_{\delta_n}[v_n].
\end{equation}

The kernel of $E_0$ is one-dimensional. Indeed, integration by parts on
$(0,a-t)$ gives
\[
\int_0^{a-t}\left(|v'|^2+\frac{h_0''}{h_0}|v|^2\right)dx
=\int_0^{a-t}h_0^2\left|\left(v/h_0\right)'\right|^2dx
+\frac{h_0'(a-t)}{h_0(a-t)}|v(a-t)|^2.
\]
The final term tends to zero: its absolute value is bounded by
$C|v(a-t)|^2/t\le C\int_{a-t}^a|v'|^2\,dx$.
Letting $t\downarrow0$ therefore proves
\begin{equation}\label{seed:limitkernel}
E_0[v]\ge0,\qquad
E_0[v]=0\ \Longleftrightarrow\ v\in\operatorname{span}\{h_0\}.
\end{equation}
For the equality case, $v/h_0$ is constant on $(0,a)$, whereas $v$ is
constant on $(a,2)$ and its value there is zero because $v(a)=0$.

If the second bound in \eqref{seed:gapestimate} failed, two orthogonal
normalized eigenfunctions with eigenvalues tending to zero would have
strong $L^2$ limits. By \eqref{seed:lowerlimit} both limits would belong
to $\operatorname{span}\{h_0\}$. Strong convergence preserves their
norms and orthogonality, which is impossible. This proves the gap.
The first eigenfunction is positive and simple by the elementary
one-dimensional ground-state argument: minimize the Rayleigh quotient,
replace a minimizer by its absolute value, and use uniqueness of the
second-order initial-value problem to exclude zeros; the Wronskian
then proves simplicity.

The same compactness argument identifies its only possible limit as
$\phi_0$. To strengthen the convergence to $H^1$, write
\[
P_\delta=\int K_+|\phi_\delta|^2\,dx,\qquad
N_\delta=\int K_-|\phi_\delta|^2\,dx.
\]
The tail estimate \eqref{seed:tail} and local convergence give
\[
N_\delta\longrightarrow
\int_0^a\frac{-h_0''}{h_0}|\phi_0|^2\,dx
=\|\phi_0'\|_2^2.
\]
Since
$\|\phi_\delta'\|_2^2+2P_\delta=\nu_\delta+N_\delta$
and $P_\delta\ge0$, weak lower semicontinuity implies convergence of
the derivative norms. Weak convergence and convergence of these norms
give strong convergence of the derivatives, proving
\eqref{seed:phiconvergence}.
\end{proof}

\begin{lemma}\label{lem:seed-response}
With $\varepsilon$ fixed as above and $\delta=p^4$, the convergence in
\eqref{seed:target-response} holds uniformly for $\mu$ in each compact
subset of the right half-plane.
\end{lemma}

\begin{proof}
Fix such a compact set $\mathcal C$ and let
$\sigma_0=\min_{\mu\in\mathcal C}\operatorname{Re}\mu>0$.
Write $R_p(\mu)=(L_p+p\mu)^{-1}$.
This inverse exists: the real part of its homogeneous energy identity is
\[
E_\delta[v]+p\int D_\delta|v|^2\,dx
+(p^2+p\operatorname{Re}\mu)\|v\|_2^2=0,
\]
which forces $v=0$. Existence follows from the self-adjoint Neumann
operator $L_p$ and its eigenfunction expansion. For fixed $p$, the inverse
maps $L^2$ into $H_N^2$ and depends analytically on $\mu$.

\paragraph{The transition contribution vanishes.}
Although $B_{\delta,p}$ is large in the transition interval, for
$\xi,v\in H^1$ we have
\begin{equation}\label{seed:Bestimate}
\begin{aligned}
\left|\int B_{\delta,p}\xi\overline v\,dx\right|
&\le C\|\xi\|_\infty
\left(\int h(h'')_+\,dx\right)^{1/2}
\left(\int\frac{(h'')_+}{h}|v|^2\,dx\right)^{1/2}\\
&\le C\sqrt\delta\,\|\xi\|_{H^1}E_\delta[v]^{1/2}.
\end{aligned}
\end{equation}
Here we used \eqref{seed:profile-estimates}, \eqref{seed:energy},
and $H^1(0,2)\hookrightarrow C([0,2])$.
Apply the energy identity to $v=R_p(\mu)B_{\delta,p}\xi$ and denote its
nonnegative left-hand side by $\mathcal E$. Then
\[
\mathcal E\le C\sqrt\delta\,\|\xi\|_{H^1}\mathcal E^{1/2},
\qquad \mathcal E\le C\delta\|\xi\|_{H^1}^2.
\]
Using $p\sigma_0\|v\|_2^2\le\mathcal E$ and
\eqref{seed:uniformenergy}, we conclude that
\begin{equation}\label{seed:layernorm}
\|R_p(\mu)B_{\delta,p}\xi\|_{H^1}
\le C\sqrt{\delta/p}\,\|\xi\|_{H^1}
=Cp^{3/2}\|\xi\|_{H^1}.
\end{equation}

\paragraph{The remaining response approaches a multiple of $h_0$.}
Because $X\ge\varepsilon$ and $(h'')_-$ is uniformly bounded,
$D_\delta$ and $\alpha_{\delta,p}$ are bounded in $L^\infty$ uniformly
as $p\to0$. Pointwise convergence away from $a$ and dominated convergence
give
\begin{equation}\label{seed:multipliers}
\begin{aligned}
D_\delta&\longrightarrow D_0=
\frac{12x^2}{1-h_0}\mathbf1_{(0,a)},\\
\alpha_{\delta,p}&\longrightarrow\alpha_0=
\frac{12x^2h_0}{(1-h_0)^2}\mathbf1_{(0,a)}
\end{aligned}
\qquad\hbox{in }L^2.
\end{equation}
Now let $v=R_p(\mu)(p\alpha_{\delta,p}\xi)$ and decompose
\[
v=z\phi_\delta+w,\qquad
\langle w,\phi_\delta\rangle=0.
\]
Let $Q_\delta$ be the $L^2$-orthogonal projection onto the complement
of $\phi_\delta$. Its projected equation is
\[
M_\delta w+pQ_\delta(D_\delta w)+(p\mu+p^2)w
=pQ_\delta(\alpha_{\delta,p}\xi)
-pzQ_\delta(D_\delta\phi_\delta).
\]
The spectral gap implies $E_\delta[w]\ge g_*\|w\|_2^2$.
Taking the real part of the inner product with $w$, and then using
\eqref{seed:uniformenergy}, yields
\begin{equation}\label{seed:wbound}
\|w\|_{H^1}\le Cp\bigl(\|\xi\|_{H^1}+|z|\bigr).
\end{equation}
Taking the inner product of the unprojected equation with $\phi_\delta$
and dividing by $p$ gives
\begin{equation}\label{seed:zequation}
\left(\mu+\frac{\nu_\delta}{p}+p
+\int D_\delta\phi_\delta^2\,dx\right)z
+\int D_\delta w\phi_\delta\,dx
=\int\alpha_{\delta,p}\xi\phi_\delta\,dx.
\end{equation}
The real part of the coefficient of $z$ is at least $\sigma_0$.
Substitute \eqref{seed:wbound} and absorb the resulting $Cp|z|$ term
for small $p$. This proves
\begin{equation}\label{seed:zwbounds}
|z|\le C\|\xi\|_{H^1},\qquad
\|w\|_{H^1}\le Cp\|\xi\|_{H^1}.
\end{equation}
Furthermore,
\[
\frac{\nu_\delta}{p}=O(p^3),\qquad
\int D_\delta\phi_\delta^2\,dx\longrightarrow
\int D_0\phi_0^2\,dx=c_\varepsilon.
\]
Strong $H^1$ convergence of $\phi_\delta$ and \eqref{seed:multipliers}
also imply
\[
\|\alpha_{\delta,p}\phi_\delta-\alpha_0\phi_0\|_2\longrightarrow0.
\]
Thus \eqref{seed:zequation} gives, uniformly for $\mu\in\mathcal C$,
\[
z=\frac{\int\alpha_0\xi\phi_0\,dx}{\mu+c_\varepsilon}
+o(1)\|\xi\|_{H^1}.
\]
The remainder is uniform over the $H^1$ unit ball: the error in the
numerator is at most
$\|\alpha_{\delta,p}\phi_\delta-\alpha_0\phi_0\|_2\|\xi\|_2$,
and all denominators have modulus at least $\sigma_0$.
Combining this with \eqref{seed:zwbounds} and
$\alpha_0h_0/Z_\varepsilon=k_\varepsilon$ gives
\[
R_p(\mu)(p\alpha_{\delta,p}\xi)
=\frac{h_0}{\mu+c_\varepsilon}\int k_\varepsilon\xi\,dx
+o(1)\|\xi\|_{H^1}
\quad\hbox{in }H^1.
\]
Adding \eqref{seed:layernorm} proves the asserted operator-norm
convergence.
\end{proof}

\subsection{An actual unstable eigenvalue for a fixed positive diffusion}

We now transfer the simple root of $F_\varepsilon$ to
\eqref{seed:original-eigenproblem}. By
\eqref{seed:first-elimination}, the latter is equivalent to
\begin{equation}\label{seed:operatorroot}
\psi+\mathcal K_p(\mu)\psi=0,\qquad
\mathcal K_p(\mu)=(1-p)R_p(\mu)\frac{a_1}{p}T_\mu.
\end{equation}
Lemma~\ref{lem:seed-response} implies
\begin{equation}\label{seed:Klimit}
\mathcal K_p(\mu)\longrightarrow\mathcal K_0(\mu),\qquad
\mathcal K_0(\mu)v=
\frac{h_0}{\mu+c_\varepsilon}\int k_\varepsilon T_\mu v\,dx,
\end{equation}
in operator norm on $H^1$, uniformly on compact subsets of the right
half-plane. These operators depend analytically on $\mu$.

Choose a small closed disk $\mathcal D$ around $\mu_\varepsilon$, contained
in the right half-plane, whose only zero of $F_\varepsilon$ is
$\mu_\varepsilon$. Define bounded projections on $H^1$ by
\[
\ell(v)=\frac1{Z_\varepsilon}\int vh_0\,dx,
\qquad Pv=h_0\ell(v),\qquad Q=I-P.
\]
Then $\ell(h_0)=1$ and $Q\mathcal K_0=0$.
Seek a solution of \eqref{seed:operatorroot} in the form
$\psi=h_0+y$, with $\ell(y)=0$. Its $Q$ component is
\[
(I+Q\mathcal K_p(\mu))y=-Q\mathcal K_p(\mu)h_0.
\]
The norm of $Q\mathcal K_p(\mu)$ tends to zero uniformly on
$\mathcal D$. Hence $I+Q\mathcal K_p(\mu)$ is invertible on
$\ker\ell$ for small $p$ by a Neumann series. This yields an analytic function
$y_p(\mu)\in\ker\ell$ satisfying
$\sup_{\mu\in\mathcal D}\|y_p(\mu)\|_{H^1}\to0$.
The remaining scalar equation is
\begin{equation}\label{seed:gp}
g_p(\mu):=1+\ell\!\left(
\mathcal K_p(\mu)(h_0+y_p(\mu))\right)=0.
\end{equation}
Uniformly on $\mathcal D$,
\[
g_p(\mu)\longrightarrow
1+\ell(\mathcal K_0(\mu)h_0)
=\frac{F_\varepsilon(\mu)}{\mu+c_\varepsilon}.
\]
The denominator does not vanish in the right half-plane. Rouch\'e's
theorem therefore provides a zero $\mu_p\in\mathcal D$ for all
sufficiently small $p>0$.

Set
\[
\psi=h_0+y_p(\mu_p),\qquad
\xi=-(1-p)T_{\mu_p}\psi,\qquad
\varphi=\xi/p-\psi.
\]
The function $\psi$ is nonzero because $\ell(\psi)=1$.
Equation \eqref{seed:operatorroot} implies $\psi\in H_N^2$;
the resolvent definition of $T_{\mu_p}$ then gives $\xi\in H_N^2$ and
$\varphi\in H_N^2$. Reversing the exact substitutions proves that
$(\varphi,\psi)$ is an eigenfunction of the original linearization, with
\begin{equation}\label{seed:actualpositive}
\lambda_p=p\mu_p,\qquad \operatorname{Re}\lambda_p>0.
\end{equation}
This is an eigenvalue for a fixed, strictly positive $p$, rather than
only a root of a limiting equation.

Integrating the first equation in \eqref{seed:exacteigen} gives
$\mu_p\int\xi\,dx=0$. Since $\mu_p\ne0$,
\begin{equation}\label{seed:zero-mass}
\int(\varphi+\psi)\,dx=p^{-1}\int\xi\,dx=0.
\end{equation}
Thus the unstable eigenfunction belongs to the required fixed-mass
space. The coefficients are Lipschitz for the fixed $p$, so
\eqref{seed:original-eigenproblem} also implies
$\varphi,\psi\in H^3$. Consequently the same eigenvalue is present in
the $H^1$ realization, whose domain has $H^3$ regularity and the same
Neumann and zero-mass conditions.

\subsection{Smoothing the coefficients while preserving instability}

Fix a sufficiently small $p$ for which \eqref{seed:actualpositive}
holds, and fix $\delta=p^4$. For $\kappa>0$, define
\begin{equation}\label{seed:smoothcoeff}
\begin{gathered}
K_\kappa^+=\frac{\sqrt{K^2+\kappa^2}+K}{2},\qquad
K_\kappa^-=\frac{\sqrt{K^2+\kappa^2}-K}{2},\\
t_\kappa=K_\kappa^++pK_\kappa^-\frac hX+p^2,\qquad
\beta_\kappa=\frac{t_\kappa}{r},\qquad
\gamma_\kappa=\frac{t_\kappa(1-r)}r+K.
\end{gathered}
\end{equation}
These coefficients are smooth. Since $K_\kappa^+-K_\kappa^-=K$,
the positivity identities remain exact:
\[
\beta_\kappa-\gamma_\kappa
=K_\kappa^-\left(1+\frac{ph}{X}\right)+p^2>0,
\qquad
\gamma_\kappa=K_\kappa^+\left(1+\frac X{ph}\right)+\frac{pX}{h}>0.
\]
Also $\beta_\kappa(1-r)-\gamma_\kappa=-K$, so the same $(s,h)$
remains an exact equilibrium.

It remains to preserve the eigenvalue. The estimate
$0\le\sqrt{K^2+\kappa^2}-|K|\le\kappa$ gives
\[
\|t_\kappa-t\|_\infty
\le\frac\kappa2\left(1+p\|h/X\|_\infty\right)\longrightarrow0.
\]
The corresponding change in the reaction matrix is exactly
\begin{equation}\label{seed:matrixperturbation}
\Delta J_\kappa=(t_\kappa-t)
\begin{pmatrix}-r&1-r\\r&-(1-r)\end{pmatrix},
\qquad
\|\Delta J_\kappa\|_{\mathcal L(L^2)}
\le\sqrt2\|t_\kappa-t\|_\infty.
\end{equation}
Its column sums are zero, so it preserves the closed subspace
\[
\mathcal H_0=\left\{(\varphi,\psi)\in L^2(0,2;\mathbb C)^2:
\int(\varphi+\psi)\,dx=0\right\}.
\]
Let $\mathcal L$ denote the unsmoothed linearization on this space, with
domain $(H_N^2)^2\cap\mathcal H_0$. It has compact resolvent. Choose a
small positively oriented contour $\Gamma$ surrounding $\lambda_p$,
with its interior contained in the right half-plane and with
$\Gamma\subset\rho(\mathcal L)$.
For sufficiently small $\kappa$,
\[
\sup_{\lambda\in\Gamma}
\|\Delta J_\kappa(\lambda-\mathcal L)^{-1}\|<1.
\]
The factorization
\[
\lambda-(\mathcal L+\Delta J_\kappa)
=\bigl[I-\Delta J_\kappa(\lambda-\mathcal L)^{-1}\bigr]
(\lambda-\mathcal L)
\]
and a Neumann series give uniform convergence of the resolvents on
$\Gamma$. Their Riesz projections therefore converge in operator norm.
The original projection is nonzero and finite rank. Once the distance
between two projections is less than one, their restrictions to each
other's ranges are injective; hence these finite ranks agree.
The smoothed linearization consequently still has an eigenvalue inside
$\Gamma$, with strictly positive real part. Smoothness of the
corresponding eigenfunction follows from the differential equations.

\subsection{Arbitrary total mass and uniqueness}

For the fixed final coefficients, $f(x,cS,cI)=c f(x,S,I)$ for every
$c>0$. Hence $c(s,h)$ is an equilibrium with total mass
$c\int(s+h)\,dx$. The derivatives
\[
f_S=\beta\frac{I^2}{(S+I)^2},\qquad
f_I=\beta\frac{S^2}{(S+I)^2}-\gamma
\]
are unchanged by this common scaling. Thus the same unstable
linearization is obtained for every prescribed positive mass.

For completeness, uniqueness can be checked directly. Any positive
equilibrium $(\widetilde s,\widetilde h)$ satisfies
$p\widetilde s+\widetilde h=C$ for some $C>0$. Setting
$w=\widetilde h/C\in(0,1)$ gives the scalar Neumann equation
\[
w''+w b(x,w)=0,\qquad
b(x,w)=\beta(x)\frac{1-w}{1-(1-p)w}-\gamma(x).
\]
For $0<w<1$,
\[
b_w(x,w)=-\frac{p\beta(x)}{[1-(1-p)w]^2}<0.
\]
If $w_1,w_2$ were two positive solutions and
$\theta=\max_{[0,2]}w_1/w_2>1$, then at a point attaining this maximum,
\[
w_1''-\theta w_2''\le0,
\qquad
w_1''-\theta w_2''
=\theta w_2\,[b(x,w_2)-b(x,w_1)]>0,
\]
a contradiction. At an endpoint the same second-derivative inequality
holds because the first derivatives vanish. Interchanging $w_1,w_2$
proves $w_1=w_2$. Every positive equilibrium is therefore a common
multiple of $(s,h)$, and its total mass determines that multiple uniquely.

\noindent\textbf{Completion of the proof of Proposition~\ref{prop:unstable-seed}:}
Choose $\varepsilon$ by Lemma~\ref{lem:seed-scalar-root}. With this
$\varepsilon$ fixed, choose $p>0$ sufficiently small and set
$\delta=p^4$ so that \eqref{seed:actualpositive} holds. Finally choose
$\kappa>0$ sufficiently small in \eqref{seed:smoothcoeff}. The final
coefficients satisfy $\beta_\kappa>\gamma_\kappa>0$, the prescribed
equilibrium is strictly positive, and its fixed-mass linearization has
a right half-plane eigenvalue. The preceding scaling and uniqueness
arguments give the statement for every positive mass.

\noindent\textbf{Proof of Theorem \ref{thm:main}:} Combining with Theorem \ref{thm:pathp} and Proposition \ref{prop:unstable-seed}, we complete the proof of Theorem \ref{thm:main}.

\section{Two, three, and four patches: a contrast in endemic stability}\label{sec:patch}
The reaction--diffusion model \eqref{eq:model} has a natural discrete-space counterpart. We restrict attention to $n$ equal-volume patches arranged along a path, with identical bidirectional weights on each edge. Let
\[
L_n=\begin{pmatrix}
-1&1&&&\\
1&-2&1&&\\
&\ddots&\ddots&\ddots&\\
&&1&-2&1\\
&&&1&-1
\end{pmatrix},
\qquad n\geq 2,
\]
where $L_2=\bigl(\begin{smallmatrix}-1&1\\1&-1\end{smallmatrix}\bigr)$. We consider the standard-incidence system
\begin{equation}\label{eq:patchmodel}
\begin{split}
\dot S_i&=p(L_nS)_i-\beta_i\frac{S_iI_i}{S_i+I_i}+\gamma_iI_i,\\
\dot I_i&=q(L_nI)_i+\beta_i\frac{S_iI_i}{S_i+I_i}-\gamma_iI_i,\qquad i=1,\ldots,n,
\end{split}
\end{equation}
with $p,q>0$ and $\beta_i>\gamma_i>0$. Thus the migration matrix is \emph{symmetric} in the ordinary Euclidean inner product, and $\sum_i(S_i+I_i)=M$ is conserved. The associated basic reproduction number is
\[
\mathcal R_0^{(n)}=\rho\bigl[\operatorname{diag}(\beta_i)
 (\operatorname{diag}(\gamma_i)-qL_n)^{-1}\bigr]>1.
\]
For each $M>0$ there is a unique positive endemic equilibrium $(S_*,I_*)$ \cite{AllenPatch,LiPengPatch}. As in the PDE, stability is understood on the fixed-mass space
$\mathcal T_M=\{(u,v)\in\R^{2n}:\sum_i(u_i+v_i)=0\}$.

\begin{theorem}[A four-patch threshold for local instability]\label{thm:patchcomparison}
For the equal-weight path model \eqref{eq:patchmodel} with $\beta_i>\gamma_i>0$:
\begin{enumerate}
\item for $n=2$ or $n=3$, the unique positive endemic equilibrium is locally exponentially asymptotically stable for every $p,q>0$ and every admissible choice of coefficients;
\item for $n=4$, there are strictly positive coefficients with $\beta_i>\gamma_i$ for which the positive equilibrium has a conjugate pair of eigenvalues in the open right half-plane on $\mathcal T_M$.
\end{enumerate}
Consequently, four is the smallest number of patches for which \emph{local endemic instability} can occur in this equal-weight path class.
\end{theorem}

The two-patch assertion follows from the cubic Routh--Hurwitz criterion. For three patches the fixed-mass characteristic polynomial has degree five; the all-parameter result uses an exact polynomial-positivity certificate over four exhaustive parameter regions. The four-patch example below is checked by an exact rational Routh table. The algebraic proofs and their verification scripts are supplied separately, rather than reproduced here. In particular, local exponential stability for $n=2,3$ should not be confused with global convergence from all positive initial data; the latter is not asserted in Theorem~\ref{thm:patchcomparison}.

For completeness we give the exact four-patch parameters. Take $p=1$, $q=20000$, and write $D_0=20562972$. The positive equilibrium of total mass $M=1$ is
\begin{equation}\label{eq:patchEE}
 S_* =\frac1{D_0}(19800000,700000,20000,40000)^\top,
 \qquad
 I_* =\frac1{D_0}(10,965,999,998)^\top.
\end{equation}
Choose
\begin{equation}\label{eq:patchrates}
\begin{aligned}
\beta&=\left(3781803890001,\ \frac{3785211}{193},\ \frac{83996}{111},\ \frac{410410479}{499}\right)^\top,\\
\gamma&=\left(3781803890000,\ \frac{96000}{193},\ \frac{20000}{999},\ \frac{400430000}{499}\right)^\top.
\end{aligned}
\end{equation}
Direct substitution gives $pL_4S_*=F(S_*,I_*)=-qL_4I_*$, where $F_i=\beta_iS_iI_i/(S_i+I_i)-\gamma_iI_i$; moreover all four rate differences are strictly positive. After removal of the conserved-mass direction, the $7\times7$ Jacobian has Routh first-column signs
\[
(+,+,+,+,+,+,-,+),
\]
and hence precisely two unstable eigenvalues. They form a nonreal conjugate pair, numerically
\[
\lambda_{\pm}\approx 0.668305872335\ \pm\ 10.483140993846\,\mathrm i.
\]
The instability is therefore present despite the symmetry of $L_4$ and the pointwise high-risk condition.

Figure~\ref{fig:patchorbit} illustrates this example using the original nonlinear eight-dimensional system. A mass-preserving perturbation of the equilibrium produces growing oscillations; the panels show total infection and the susceptible population in the second patch. The time integration is independent of the exact Routh proof, and the finite observation interval does not establish the existence or stability of a periodic orbit for the stated parameters.
\begin{figure}[H]
\centering
\includegraphics[width=0.96\textwidth]{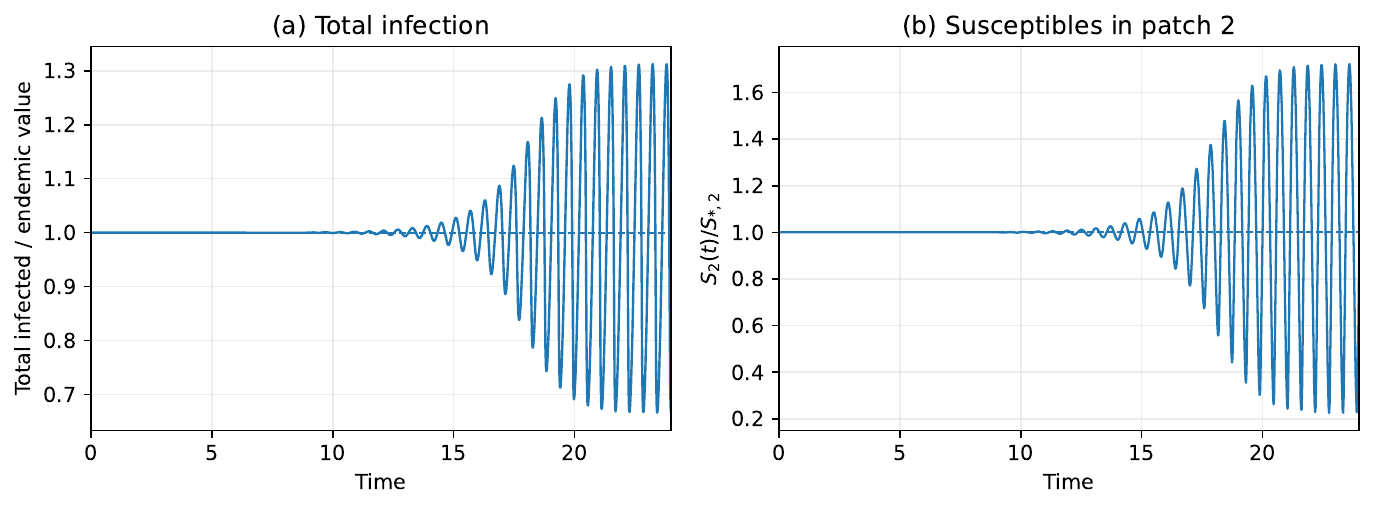}
\caption{Evolution of the four-patch model \eqref{eq:patchmodel}--\eqref{eq:patchrates} from a positive, mass-preserving perturbation. Left: total infected population relative to its endemic value; right: $S_2(t)/S_{*,2}$. Dashed lines denote equilibrium levels. The trajectories solve the full nonlinear ODE with $p=1$, $q=20000$; the numerical observation is not used to infer a periodic-orbit theorem.}\label{fig:patchorbit}
\end{figure}

The two- and three-patch systems thus exclude local instability in this model class, whereas the four-patch chain permits it, as does the continuum equation. This is a comparison of stability properties, not a claim of spectral convergence for discretizations of one fixed heterogeneous medium. The minimum-patch assertion does not extend to arbitrary weighted graphs or nonsymmetric movement.

\appendix
\section{Unstable construction for   \eqref{int-positive} by shadow system}\label{A}

Here, we recall that $p=d_S$ and $q=d_I$. We construct the coefficients $\beta,\gamma$ such that the corresponding positive equilibrium whose fiexed mass linearization has an eigenvalue with positive real part.
\begin{theorem}\label{thm:mains}
For every $p>0$, on the interval $\Om=(-1,2\pi)$ there exist fixed functions $\beta,\gamma\in C^\infty(\overline\Om)$ such that
\begin{equation}\label{eq:positivecoeff}
\min\beta>0,\qquad \min\gamma>0,\qquad
\int_\Om(\beta-\gamma)\dd x>0.
\end{equation}
There is a finite $q_0>0$ such that, for every $q>q_0$ and every $M>0$, system \eqref{eq:model} has a unique positive equilibrium of mass $M$, and its linearization on $\Xzero$ has a nonreal eigenvalue with positive real part. Such eigenvalues occur arbitrarily close to
\begin{equation}\label{eq:lambdastar}
\lambda_*=p(3+4\mathrm i)
\end{equation}
for all sufficiently large $q$. The equilibrium is also nonlinearly Lyapunov unstable in the positive fixed-mass phase space.
\end{theorem}
The coefficients and $q_0$ do not depend on $M$. This follows from the degree-one homogeneity of the incidence term.

\noindent\textbf{Preliminary:}
Define the average operator and its complementary projection by
\[
Pf=\frac1{|\Om|}\int_\Om f\dd x,\qquad Q=\operatorname{Id}-P,
\]
identifying an average with a constant function. Spectral statements are made on the fixed-mass tangent space
\begin{equation}\label{eq:massspace}
\Xzero=\{(\phi,\psi)\in L^2(\Om;\C)^2:P(\phi+\psi)=0\}.
\end{equation}

Let $a\in C^\infty(\overline\Om)$ be real and strictly positive, and set
\begin{equation}\label{eq:A}
A=-p\partial_{xx}+a(x),\qquad D(A)=\HN(\Om;\C).
\end{equation}
This is a strictly positive self-adjoint operator with compact resolvent. For $\Rea\lambda>0$ define
\begin{equation}\label{eq:z}
z_\lambda=A(A+\lambda)^{-1}1=(A+\lambda)^{-1}a.
\end{equation}
Equivalently, $z_\lambda$ is the unique Neumann solution of
\begin{equation}\label{eq:zode}
-pz''+(a+\lambda)z=a.
\end{equation}

\begin{lemma}\label{lem:response}
There exists $a\in C^\infty(\overline\Om)$, with $\min a>0$, for which the origin belongs to the interior, in $\C\simeq\R^2$, of the convex hull of four interior values of $z_{\lambda_*}$.
\end{lemma}
\begin{proof}
Write $L=2\pi$ and $\kappa=2+\mathrm i$, so $\lambda_*=p\kappa^2$. First use the auxiliary, nonsmooth potential
\[
a_0(x)=\begin{cases}p,&-1<x<0,\\0,&0<x<L.\end{cases}
\]
The associated weak Neumann problem has a unique solution $z_0$ by coercivity, since $\Rea\lambda_*=3p>0$. Put
\[
\mu=\sqrt{4+4\mathrm i}\quad(\Rea\mu>0),\qquad c=(4+4\mathrm i)^{-1}.
\]
The solution has the form
\begin{equation}\label{eq:piecewise}
z_0(x)=\begin{cases}
c+D\cosh(\mu(x+1)),&-1<x<0,\\
C\cosh(\kappa(L-x)),&0<x<L.
\end{cases}
\end{equation}
Matching the value and first derivative at $0$ gives
\begin{equation}\label{eq:C}
C=\frac{c\mu\tanh\mu}
{\mu\tanh\mu\cosh(\kappa L)+\kappa\sinh(\kappa L)}\ne0.
\end{equation}
The numerator is nonzero because $\Rea\mu>0$. A vanishing denominator would produce a nonzero solution of the homogeneous Neumann problem, contradicting coercivity.

At $y=L-x=0,\pi/2,\pi,3\pi/2$, respectively, the right-hand formula gives
\begin{equation}\label{eq:fourvalues}
C,\qquad \mathrm iC\sinh\pi,\qquad
-C\cosh(2\pi),\qquad -\mathrm iC\sinh(3\pi).
\end{equation}
Before multiplication by $C$, these lie on the four coordinate half-axes with strictly positive lengths. Their convex hull therefore contains the origin strictly in its interior; multiplication by $C\ne0$ preserves this property. Move the first point from $y=0$ to a sufficiently small positive $y$. All four points are then interior points, and the strict inclusion remains true.

We now smooth the potential and make it positive. Choose $\chi_\delta\in C^\infty(\overline\Om)$ such that $0\le\chi_\delta\le1$, with $\chi_\delta=1$ for $x\le-\delta$ and $\chi_\delta=0$ for $x\ge\delta$. For $0<\eta<1$ let
\begin{equation}\label{eq:asmooth}
a_{\eta,\delta}=p[\eta+(1-\eta)\chi_\delta].
\end{equation}
Then $p\eta\le a_{\eta,\delta}\le p$, and $a_{\eta,\delta}\to a_0$ in $L^1$ as $\eta,\delta\to0$. If $w=z_{\eta,\delta}-z_0$, subtraction yields
\[
-pw''+(a_{\eta,\delta}+\lambda_*)w
=(a_{\eta,\delta}-a_0)(1-z_0).
\]
Testing by $\overline w$ and taking real parts gives
\[
p\norm{w'}_2^2+3p\norm{w}_2^2
\le\norm{a_{\eta,\delta}-a_0}_1\norm{1-z_0}_\infty\norm{w}_\infty.
\]
The embedding $H^1(\Om)\hookrightarrow C(\overline\Om)$ therefore implies
\[
\norm{z_{\eta,\delta}-z_0}_\infty
\le C_p\norm{a_{\eta,\delta}-a_0}_1\longrightarrow0.
\]
Strict interior inclusion is open under small perturbations of the four vertices. Fix sufficiently small, strictly positive $\eta,\delta$ once and for all, and put $a=a_{\eta,\delta}$. This proves the lemma with a fixed smooth positive potential.
\end{proof}

\noindent\textbf{Positive coefficients and an exact unstable shadow eigenvalues:}
From now on, fix the $a$ supplied by Lemma~\ref{lem:response}, and write $z=z_{\lambda_*}$.

\begin{lemma}\label{lem:moment}
There exist $w_0,w_1\in C^\infty(\overline\Om)$, constant near the endpoints, such that
\begin{equation}\label{eq:moments}
w_0\ge1,\qquad w_1\ge0,\qquad
P(w_0z)=0,\qquad P(w_1z)=-1.
\end{equation}
\end{lemma}
\begin{proof}
Choose four nonnegative functions $f_j\in C_c^\infty(\Om)$ with disjoint sufficiently small supports near the four points of Lemma~\ref{lem:response}, normalized by $Pf_j=1$. By continuity of $z$, their moments
\[
m_j=P(f_jz),\qquad j=1,2,3,4,
\]
still have a convex hull containing $0$ in its interior. Their nonnegative cone is consequently all of $\C$: a small disk centered at $0$ lies in the convex hull, and any complex number is a positive multiple of a point of that disk. Hence there exist real $c_j,d_j\ge0$ such that
\[
\sum_jc_jm_j=-Pz,\qquad \sum_jd_jm_j=-1.
\]
Take $w_0=1+\sum_jc_jf_j$ and $w_1=\sum_jd_jf_j$. These functions have all the required properties.
\end{proof}
Define, with a positive constant $T$ to be chosen,
\begin{equation}\label{eq:coeff}
t=Tw_0+w_1,\qquad
\beta=a(1+t)^2,\qquad
\gamma=at(1+t)-pt''.
\end{equation}
Let $a_{\min}=\min a$, $C_0=\norm{w_0''}_\infty$, and $C_1=\norm{w_1''}_\infty$. Choose
\begin{equation}\label{eq:T}
T\ge\max\left\{1,\frac{2p(C_0+C_1)}{a_{\min}}\right\}.
\end{equation}
Then $t\ge T$, $t'=0$ at the endpoints, and
\begin{align}
\gamma&\ge a_{\min}T^2-p(TC_0+C_1)
\ge\tfrac12a_{\min}T^2>0,\label{eq:gammapos}\\
\int_\Om(\beta-\gamma)\dd x
&=\int_\Om a(1+t)\dd x+p\int_\Om t''\dd x
=\int_\Om a(1+t)\dd x>0.\label{eq:risk}
\end{align}
Thus all coefficient assumptions hold. Crucially, increasing $T$ has not altered the exact moment condition:
\begin{equation}\label{eq:momentexact}
P(tz)=-1.
\end{equation}
All of $a,t,\beta,\gamma$ are now fixed, before $q$ is chosen.

\subsection{The shadow operator}
If the infected component is spatially constant, an equilibrium is
\[
(S_0,I_0)=(Ht,H),\qquad H=\frac{M}{\int_\Om(1+t)\dd x}.
\]
The susceptible equilibrium equation follows from \eqref{eq:coeff}; the integrated infected equation follows by integration. Let
\[
f(x,S,I)=\beta(x)\frac{SI}{S+I}-\gamma(x)I.
\]
At this shadow equilibrium,
\begin{equation}\label{eq:shadowderiv}
f_S=a,\qquad f_I=pt''-at=-At.
\end{equation}
A fixed-mass perturbation has spatially constant infected component $\psi=-P\phi$. Its susceptible component evolves under
\begin{equation}\label{eq:B}
B=-A-(At)P,\qquad D(B)=\HN(\Om;\C).
\end{equation}
The term $(At)P$ denotes $\phi\mapsto(At)P\phi$; it is a bounded rank-one operator. Thus $B$ has compact resolvent.

\begin{proposition}\label{prop:shadow}
The number $\lambda_*=p(3+4\mathrm i)$ is an eigenvalue of $B$.
\end{proposition}
\begin{proof}
Since $\Rea\lambda_*>0$, the operator $A+\lambda_*$ is invertible. For a nonzero eigenfunction of $B$, the rank-one characteristic equation is
\begin{equation}\label{eq:characteristic}
1+P[(A+\lambda)^{-1}At]=0.
\end{equation}
For completeness, the moment identity needed here uses bilinear symmetry, without complex conjugation. Neumann integration by parts gives
\begin{equation}\label{eq:bilinear}
P[(A+\lambda)^{-1}At]
=P[tA(A+\lambda)^{-1}1]=P(tz_\lambda).
\end{equation}
It can also be checked by multiplying $(A+\lambda)v=At$ by the solution of $(A+\lambda)g=1$, and integrating twice by parts. At $\lambda=\lambda_*$, equation \eqref{eq:characteristic} follows from \eqref{eq:momentexact}.

Explicitly, let $\phi_*=-(A+\lambda_*)^{-1}At$. Then $P\phi_*=1$, so $\phi_*\ne0$, and
\[
-A\phi_*=At+\lambda_*\phi_*,\qquad
B\phi_*=\lambda_*\phi_*.
\]
The fixed-mass shadow eigenvector is $(\phi_*,-1)$.
\end{proof}
\begin{remark}
A negative direction for a proposed Lyapunov functional would not establish an unstable eigenvalue. Proposition~\ref{prop:shadow} instead constructs an actual eigenfunction for an explicitly specified spectral point.
\end{remark}

\noindent\textbf{Positive equilibrium for finite diffusion:}
The preceding coefficients remain unchanged throughout this section. Choose $0<\alpha<1$, set $\varepsilon=q^{-1}$, and for $v\in C_N^{2,\alpha}(\overline\Om)$ define
\begin{equation}\label{eq:j}
j_\varepsilon[v]=1-p\varepsilon(v-Pv).
\end{equation}
Its average is one, and $qj_\varepsilon[v]''=-pv''$. Consider
\begin{equation}\label{eq:F}
F(\varepsilon,v)=pv''
-\beta\frac{v j_\varepsilon[v]}{v+j_\varepsilon[v]}
+\gamma j_\varepsilon[v].
\end{equation}
This is a smooth map from a neighborhood of $(0,t)$ in
$\R\times C_N^{2,\alpha}$ to $C^{0,\alpha}$, restricted to positive $v,j_\varepsilon[v]$. We have $F(0,t)=0$, and
\begin{equation}\label{eq:DF}
D_vF(0,t)=p\partial_{xx}-\frac{\beta}{(1+t)^2}=-A.
\end{equation}
Because $\min a>0$, this Neumann elliptic operator is an isomorphism from $C_N^{2,\alpha}$ onto $C^{0,\alpha}$: uniqueness follows from the maximum principle or the energy identity, and existence with the stated regularity follows from the scalar elliptic problem.

The implicit function theorem provides $v_\varepsilon=t+O(\varepsilon)$ in $C^{2,\alpha}$ with $F(\varepsilon,v_\varepsilon)=0$. Write $j_\varepsilon=j_\varepsilon[v_\varepsilon]=1+O(\varepsilon)$. Both functions are strictly positive for sufficiently small $\varepsilon>0$. Set
\begin{equation}\label{eq:finiteeq}
H_\varepsilon=\frac{M}{\int_\Om(v_\varepsilon+j_\varepsilon)\dd x},\qquad
S_{0,q}=H_\varepsilon v_\varepsilon,\quad
I_{0,q}=H_\varepsilon j_\varepsilon.
\end{equation}
The first equilibrium equation is exactly $F=0$; the second follows from the first and $pv_\varepsilon''+qj_\varepsilon''=0$. The boundary conditions and mass constraint also hold. Elliptic bootstrapping gives smoothness.

The reaction derivatives at this equilibrium are
\begin{align}
a_q&=\beta\frac{j_\varepsilon^2}{(v_\varepsilon+j_\varepsilon)^2},\label{eq:aq}\\
c_q&=\beta\frac{v_\varepsilon^2}{(v_\varepsilon+j_\varepsilon)^2}-\gamma.\label{eq:cq}
\end{align}
Uniformly on $\overline\Om$,
\begin{equation}\label{eq:coeflimit}
a_q=a+O(q^{-1}),\qquad c_q=-At+O(q^{-1}).
\end{equation}
None of these quantities depends on the mass normalization $H_\varepsilon$.

\noindent\textbf{Persistence of unstable spectrum at finite \texorpdfstring{$q$:}{q}}\label{sec:spectrum}
The linearized operator on \eqref{eq:massspace} is
\begin{equation}\label{eq:linear}
L_q\binom\phi\psi=
\binom{p\phi''-a_q\phi-c_q\psi}
{q\psi''+a_q\phi+c_q\psi},
\qquad D(L_q)=\HN(\Om;\C)^2\cap\Xzero.
\end{equation}
Introduce the fixed Hilbert space
\[
\mathcal H=L^2(\Om;\C)\times L_0^2(\Om;\C),\qquad
L_0^2=\{w:Pw=0\},
\]
and the bounded isomorphism $\mathcal T(\phi,w)=(\phi,w-P\phi)$ onto $\Xzero$. In these coordinates,
\begin{equation}\label{eq:block}
\mathscr L_q=\mathcal T^{-1}L_q\mathcal T
=\begin{pmatrix}B_q&C_q\\E_q&D_q\end{pmatrix},
\end{equation}
where
\begin{align}
B_q&=p\partial_{xx}-a_q+c_qP,&C_q&=-c_q,\label{eq:blocks1}\\
E_q&=Q(a_q-c_qP),&D_q&=q\partial_{xx}|_{L_0^2}+Qc_q|_{L_0^2}.\label{eq:blocks2}
\end{align}
Multiplication by a coefficient is understood in these formulas. The common domain is $\HN\times(\HN\cap L_0^2)$. In particular, $B_q-B=O(q^{-1})$ as a bounded operator on $L^2$, and $C_q,E_q$ are uniformly bounded as $q\to\infty$.

Choose a positively oriented circle $\Gamma$ around $\lambda_*$, with its closed disk contained in $\{\Rea\lambda>0,\Ima\lambda>0\}$ and containing no other distinct spectral point of $B$. Such a circle exists because $B$ has compact resolvent. The eigenvalue $\lambda_*$ may have any finite algebraic multiplicity.

Let $\mu_1=(\pi/|\Om|)^2$ be the first positive Neumann eigenvalue of $-\partial_{xx}$. On $L_0^2$, the operator $Qc_q$ is self-adjoint, since $c_q$ is real and $Q$ is the orthogonal mean-zero projection. Thus
\[
\Rea\langle D_qw,w\rangle
\le(-q\mu_1+\norm{c_q}_\infty)\norm{w}_2^2.
\]
For large $q$, uniformly on $\Gamma$,
\begin{equation}\label{eq:fastresolvent}
R_q(\lambda):=(\lambda-D_q)^{-1},\qquad
\norm{R_q(\lambda)}\le
\frac{1}{q\mu_1+\min_\Gamma\Rea\lambda-\norm{c_q}_\infty}
=O(q^{-1}).
\end{equation}
The Schur complement of $\lambda-\mathscr L_q$ is
\begin{equation}\label{eq:schur}
\mathcal S_q(\lambda)=\lambda-B_q-C_qR_q(\lambda)E_q
=(\lambda-B)+O(q^{-1}),
\end{equation}
where the remainder is bounded on $L^2$. Uniform boundedness of $(\lambda-B)^{-1}$ on $\Gamma$ and a Neumann series show that $\mathcal S_q$ is invertible there, with
\[
\mathcal S_q(\lambda)^{-1}\longrightarrow(\lambda-B)^{-1}
\quad\text{uniformly in operator norm on }\Gamma.
\]
The exact block inverse is
\begin{equation}\label{eq:blockresolvent}
(\lambda-\mathscr L_q)^{-1}=
\begin{pmatrix}
\mathcal S_q^{-1}&\mathcal S_q^{-1}C_qR_q\\
R_qE_q\mathcal S_q^{-1}&R_q+R_qE_q\mathcal S_q^{-1}C_qR_q
\end{pmatrix}.
\end{equation}
Consequently,
\begin{equation}\label{eq:reslimit}
(\lambda-\mathscr L_q)^{-1}\longrightarrow
\begin{pmatrix}(\lambda-B)^{-1}&0\\0&0\end{pmatrix}
\quad\text{uniformly on }\Gamma.
\end{equation}
Integrating \eqref{eq:reslimit} around $\Gamma$ gives convergence in norm of Riesz projections:
\begin{equation}\label{eq:riesz}
\Pi_q=\frac{1}{2\pi\mathrm i}\int_\Gamma(\lambda-\mathscr L_q)^{-1}\dd\lambda
\longrightarrow
\Pi_\infty=\begin{pmatrix}\Pi_B&0\\0&0\end{pmatrix}.
\end{equation}
The projection $\Pi_B$ has positive finite rank. Two projections at operator-norm distance less than one have equal rank: restricting either to the range of the other is injective, in both directions. Hence, for every sufficiently large finite $q$, $\Pi_q$ also has positive rank. The operator $L_q$ therefore has an eigenvalue inside $\Gamma$, with positive real and imaginary parts. Shrinking $\Gamma$ proves the approximation assertion in Theorem~\ref{thm:mains}. Real coefficients provide its complex conjugate as well.

\section{Patch Model and Main Results}\label{B}

The standard-incidence SIS patch model describes the combined effects of
spatial heterogeneity and population movement on disease transmission.
Allen et al.~\cite{Allen} developed a threshold theory for this model
and investigated the existence, uniqueness, and small-dispersal asymptotic
profiles of endemic equilibria. Li and Peng~\cite{LiPengPatch} further
studied the dynamics and asymptotic properties of endemic equilibria and
established global stability in certain special cases. Gao and
Li~\cite{GaoLiPatch} considered equilibrium structure and disease prevalence
when the two population classes may have different connectivity matrices.
Here we focus on local stability for a common symmetric path graph and
compare explicitly the cases of two, three, and four patches.

Let $n\ge2$ and define the dispersal matrix $L_n$ for the path graph with
unit edge weights by
\begin{equation}\label{eq:L}
(L_n z)_1=z_2-z_1,\quad
(L_n z)_i=z_{i-1}-2z_i+z_{i+1}\ (2\le i\le n-1),\quad
(L_n z)_n=z_{n-1}-z_n.
\end{equation}
For $n=2$, only the endpoint formulas are used. Set $K_n=-L_n$ and
consider the system
\begin{equation}\label{eq:modelp}
\begin{aligned}
\dot S_i&=p(L_nS)_i-F_i(S_i,I_i),\\
\dot I_i&=q(L_nI)_i+F_i(S_i,I_i),\\
F_i(S_i,I_i)&=\beta_i\frac{S_iI_i}{S_i+I_i}-\gamma_iI_i,
\qquad i=1,\dots,n.
\end{aligned}
\end{equation}
Throughout the paper, we assume that
\begin{equation}\label{eq:assumptions}
p,q>0,\qquad \beta_i>\gamma_i>0\quad(i=1,\dots,n).
\end{equation}
All coefficients are independent of time, although $\beta_i$ and $\gamma_i$
may vary with the patch. A common positive edge weight can be absorbed
into $p$ and $q$; hence, the unit-weight formulation covers all path
graphs with equal edge weights.

For $M>0$, define the positive phase space with prescribed total population
and its tangent space by
\begin{equation}\label{eq:mass}
\mathcal H_M^+=\left\{(S,I)\in(0,\infty)^{2n}:
\sum_{i=1}^n(S_i+I_i)=M\right\},\qquad
\mathcal T=\{(u,v):\one^\top(u+v)=0\}.
\end{equation}
All notions of local stability in this paper are relative to
$\mathcal H_M^+$. A positive equilibrium $(s,h)\in\mathcal H_M^+$ is
said to be locally exponentially asymptotically stable if there exist
$\varepsilon,C,\eta>0$ such that every initial condition in the same
phase space satisfying
$\|(S(0),I(0))-(s,h)\|<\varepsilon$ obeys
\begin{equation}\label{eq:exp}
\|(S(t),I(t))-(s,h)\|
\le C e^{-\eta t}\|(S(0),I(0))-(s,h)\|,\qquad t\ge0.
\end{equation}
These constants may depend on the model parameters and the equilibrium.

\begin{theorem}\label{thm:mainp}
Under \eqref{eq:assumptions}, for every $n\ge2$ and $M>0$, system
\eqref{eq:model} admits a unique positive equilibrium in $\mathcal H_M^+$.
Moreover:
\begin{enumerate}[label=(\roman*),leftmargin=2.5em]
\item If $n=2$, this equilibrium is locally exponentially asymptotically
stable for all admissible parameters.
\item If $n=3$, this equilibrium is locally exponentially asymptotically
stable for all admissible parameters. The proof uses exact
integer-polynomial computations.
\item If $n=4$, there exist rational parameters satisfying
\eqref{eq:assumptions} for which the equilibrium in $\mathcal H_1^+$ is
locally unstable; its restricted Jacobian has exactly two eigenvalues,
counted with algebraic multiplicity, in the open right half-plane.
\end{enumerate}
\end{theorem}

The three-patch statement concerns only the equal-edge-weight path
\eqref{eq:L}; it does not extend, as stated, to arbitrary graphs on
three vertices or to three-patch chains with unequal edge weights. We do
not establish global attraction for general parameters or the existence
of time-periodic solutions.

\subsection{Positive Equilibria and the Structure of the Linearization}

\subsubsection{Conservation of Total Population asubnd Existence and Uniqueness of a Positive Equilibrium}

Since $\one^\top L_n=0$, summing all the equations in \eqref{eq:model}
gives
\begin{equation}\label{eq:mass-conservation}
\frac{d}{dt}\sum_i(S_i+I_i)=0.
\end{equation}
Solutions with strictly positive initial data remain strictly positive.
Indeed, in the positive region the incidence term satisfies
$0\le\beta_iS_iI_i/(S_i+I_i)\le\beta_iS_i$, and consequently each
component satisfies an inequality of the form $\dot S_i\ge-C_iS_i$ or
$\dot I_i\ge-\widetilde C_iI_i$. Combined with the upper bound furnished
by conservation of total population, these inequalities imply that
positive solutions exist for all $t\ge0$.

If $(s,h)$ is a positive equilibrium, adding its two equilibrium equations
shows that $L_n(ps+qh)=0$. Since the path graph is connected, there
exists $c_0>0$ such that
\begin{equation}\label{eq:equilibrium-constant}
ps_i+qh_i=c_0\qquad(i=1,\dots,n).
\end{equation}

\begin{proposition}\label{prop:exist-unique}
Under \eqref{eq:assumptions}, there exists a unique positive equilibrium
in $\mathcal H_M^+$ for every $M>0$.
\end{proposition}
\begin{proof}
Set $x_i=qh_i/c_0$. Then $0<x_i<1$ and
\begin{equation}\label{eq:reconstruction}
s_i=\frac{c_0(1-x_i)}p,\qquad h_i=\frac{c_0x_i}q.
\end{equation}
The equilibrium equation for infected individuals is equivalent to
\begin{equation}\label{eq:scalar-equilibrium}
q(L_nx)_i+x_i g_i(x_i)=0,\qquad
g_i(z)=\frac{\beta_iq(1-z)}{q(1-z)+pz}-\gamma_i.
\end{equation}
For $0\le z\le1$,
\[
g_i(0)=\beta_i-\gamma_i>0,\qquad
g_i(1)=-\gamma_i<0,\qquad
g_i'(z)=-\frac{\beta_ipq}{[q(1-z)+pz]^2}<0.
\]
Minimize the continuous function
\[
\mathcal V(x)=\frac q2x^\top K_nx
-\sum_i\int_0^{x_i} z g_i(z)\,dz
\]
over the compact cube $[0,1]^n$. For sufficiently small $\delta>0$,
$\mathcal V(\delta\one)<0=\mathcal V(0)$, so no minimizer is the zero
vector. If a minimizer has a zero component, the connectivity of the path
implies that some zero component $x_j=0$ is adjacent to a positive
component. Then $\partial_j\mathcal V=q(K_nx)_j<0$, and increasing
$x_j$ into the cube strictly decreases $\mathcal V$, a contradiction.
If $x_j=1$ for some component, then $(K_nx)_j\ge0$ and hence
$\partial_j\mathcal V=q(K_nx)_j+\gamma_j>0$. Decreasing this component
into the cube again lowers $\mathcal V$. Thus every minimizer lies in
$(0,1)^n$ and satisfies \eqref{eq:scalar-equilibrium}.

For uniqueness, let $u,v\in(0,1)^n$ be two solutions and set
$\theta=\max_i u_i/v_i$. Suppose that $\theta>1$. At some contact
index $j$ we have $u_j=\theta v_j$ and $u_i-\theta v_i\le0$ for all
$i$. Therefore $(L_n(u-\theta v))_j\le0$. Subtracting the two
equilibrium equations yields
\[
0=q(L_n(u-\theta v))_j+
\theta v_j[g_j(\theta v_j)-g_j(v_j)]<0,
\]
a contradiction. It follows that $u\le v$; interchanging $u$ and $v$
gives $v\le u$, hence $u=v$. Finally, the total-population constraint
uniquely determines
\[
c_0=\frac{M}{\displaystyle\sum_i
\left(\frac{1-x_i}{p}+\frac{x_i}{q}\right)}.
\]
Substituting into \eqref{eq:reconstruction} completes the proof.
\end{proof}

\subsubsection{A Quadratic Matrix Pencil and the Simple Zero Eigenvalue}

Fix a positive equilibrium $(s,h)$ and set
\begin{equation}\label{eq:notation}
N_i^*=s_i+h_i,\quad r_i=\frac{h_i}{N_i^*},\quad
t_i=\beta_i r_i,\quad c_i=\frac{q(L_nh)_i}{h_i}.
\end{equation}
The equilibrium equation for infected individuals gives
$\gamma_i=\beta_i(1-r_i)+c_i$. The two partial derivatives of the
reaction term at the equilibrium are
\begin{equation}\label{eq:reaction-derivatives}
\partial_SF_i=t_ir_i,\qquad
\partial_IF_i=-c_i-t_i(1-r_i).
\end{equation}
Define
\[
\mathsf A=\diag(t_ir_i),\qquad
\mathsf G=\diag[-c_i-t_i(1-r_i)].
\]
The full Jacobian is therefore
\begin{equation}\label{eq:full-J}
J=\begin{pmatrix}-pK_n-\mathsf A&-\mathsf G\\
\mathsf A&-qK_n+\mathsf G\end{pmatrix}.
\end{equation}

\begin{lemma}\label{lem:pencil}
Let
\[
\mathcal H=qK_n+\diag(c_i+t_i),\qquad
\mathcal C=p\mathcal H+(q-p)\mathsf A.
\]
Then
\begin{equation}\label{eq:pencil}
\det(\lambda I_{2n}-J)
=\det\!\left[\lambda^2I_n+\lambda(pK_n+\mathcal H)
+\mathcal C K_n\right].
\end{equation}
\end{lemma}
\begin{proof}
In $\lambda I_{2n}-J$, add the upper block row to the lower block row.
This leaves the determinant unchanged and gives
\[
\begin{pmatrix}
\lambda I_n+pK_n+\mathsf A&\mathsf G\\
\lambda I_n+pK_n&\lambda I_n+qK_n
\end{pmatrix}.
\]
Denote the four blocks, in order, by $P,Q,R,T$. Whenever $T$ is
invertible, $RT=TR$, and therefore
\[
\det\begin{pmatrix}P&Q\\R&T\end{pmatrix}
=\det(P-QT^{-1}R)\det T=\det(PT-QR).
\]
Expanding $PT-QR$ gives \eqref{eq:pencil}. Since both sides are
polynomials in $\lambda$, the identity extends to the finitely many
values at which $T$ is singular. This calculation only commutes two
polynomials in $K_n$; it does not commute a heterogeneous diagonal
matrix with $K_n$.
\end{proof}

\begin{lemma}\label{lem:zero}
The zero eigenvalue of the full Jacobian is algebraically simple. A
corresponding right eigenvector is $(s,h)$, and zero is not an
eigenvalue of the restriction to the fixed-mass tangent space
$\mathcal T$.
\end{lemma}
\begin{proof}
Define
\[
\mathcal B=q\left[K_n-\diag\!\left(\frac{(K_nh)_i}{h_i}\right)\right].
\]
Summing over the edges yields, for every $z\in\R^n$,
\begin{equation}\label{eq:ground}
z^\top\mathcal Bz
=q\sum_{i=1}^{n-1}h_ih_{i+1}
\left(\frac{z_i}{h_i}-\frac{z_{i+1}}{h_{i+1}}\right)^2\ge0.
\end{equation}
Since $c_i=-q(K_nh)_i/h_i$,
\[
\mathcal H=\mathcal B+\diag(t_i)>0,\qquad
\mathcal C=p\mathcal B+
\diag\!\left(\frac{t_ic_0}{N_i^*}\right)>0.
\]
Here $>0$ denotes symmetric positive definiteness. Moreover,
\[
\mathcal Ch=c_0\mathsf A\one,\qquad
\mathcal C\one=p\mathcal H\one+(q-p)\mathsf A\one.
\]
Using $pN^*=c_0\one-(q-p)h$, we obtain
\begin{equation}\label{eq:Cidentity}
\mathcal C^{-1}\mathcal H\one=N^*/c_0.
\end{equation}
The Laplacian of the unit-weight path graph satisfies
$\adj K_n=\one\one^\top$: its kernel is spanned by $\one$, and the
principal cofactor obtained by deleting an endpoint equals $1$ by the
tridiagonal determinant recurrence. Differentiating \eqref{eq:pencil}
at $\lambda=0$ and using
$\adj(\mathcal CK_n)=\adj K_n\,\adj\mathcal C$ gives
\begin{align}
\left.\frac d{d\lambda}\det(\lambda I_{2n}-J)\right|_{\lambda=0}
&=\det\mathcal C\,\one^\top
\mathcal C^{-1}(pK_n+\mathcal H)\one \notag\\
&=\det\mathcal C\,\frac{M}{c_0}>0. \label{eq:zero-derivative}
\end{align}
Thus the zero eigenvalue is algebraically simple. The original vector
field is homogeneous of degree one in $(S,I)$, so $J(s,h)=0$. Moreover,
$(s,h)\notin\mathcal T$, because its components sum to $M>0$.
Conservation of mass makes $\mathcal T$ an invariant subspace, proving
the assertion.
\end{proof}

By \eqref{eq:mass-conservation}, in a basis adapted to $\mathcal T$,
\begin{equation}\label{eq:restricted-characteristic}
\det(\lambda I_{2n}-J)=\lambda\,
\det(\lambda I_{2n-1}-J|_{\mathcal T}).
\end{equation}
Henceforth we study the restricted characteristic polynomial on the
right-hand side.

\subsection{Two Patches: A Direct Algebraic Proof}

On the fixed-mass manifold, use the coordinates
$n_1=N_1=S_1+I_1,\ x=I_1,\ y=I_2$.
Then $N_2=M-n_1$ and
\begin{equation}\label{eq:two-system}
\begin{aligned}
\dot n_1&=p(M-2n_1)+(q-p)(y-x),\\
\dot x&=q(y-x)+(\beta_1-\gamma_1)x-\beta_1x^2/n_1,\\
\dot y&=q(x-y)+(\beta_2-\gamma_2)y-\beta_2y^2/(M-n_1).
\end{aligned}
\end{equation}
At a positive equilibrium, define
\[
\theta=\frac{y^*}{x^*},\quad
f_1=\frac{x^*}{n_1^*},\quad f_2=\frac{y^*}{M-n_1^*},\quad
b_i=\beta_if_i,\quad k_i=b_if_i,
\]
\[
u=b_1+q\theta,\quad v=b_2+q/\theta,\quad
\Sigma=u+v,\quad T=uv-q^2.
\]
Note that $0<f_i<1$, all these quantities are positive, and
\begin{equation}\label{eq:two-T}
T=b_1b_2+q(b_1/\theta+b_2\theta)>0,\qquad k_1+k_2<\Sigma.
\end{equation}
Using the equilibrium equations for the infected components to
differentiate \eqref{eq:two-system}, we obtain
\[
J_2=\begin{pmatrix}
-2p&-(q-p)&q-p\\
k_1&-u&q\\
-k_2&q&-v
\end{pmatrix}.
\]
Write its characteristic polynomial as
$\lambda^3+A\lambda^2+B\lambda+C$. Direct expansion gives
\begin{align}
A&=2p+\Sigma, \notag\\
B&=T+2p\Sigma+(q-p)(k_1+k_2),\label{eq:two-abc}\\
C&=2pT+(q-p)\{k_1(v-q)+k_2(u-q)\}. \notag
\end{align}

Dividing the equilibrium relation \eqref{eq:equilibrium-constant} by
$x^*$ gives
\begin{equation}\label{eq:two-balance}
(q-p)(\theta-1)=p(1/f_1-\theta/f_2).
\end{equation}
Consequently,
\[
\begin{aligned}
2p/\theta+(q-p)f_1(1/\theta-1)&=p(1/\theta+f_1/f_2),\\
2p\theta+(q-p)f_2(\theta-1)&=p(\theta+f_2/f_1).
\end{aligned}
\]
Expanding $C$ in \eqref{eq:two-abc} and substituting these two
identities yields
\begin{align}
C={}&b_1b_2\{p(2-f_1-f_2)+q(f_1+f_2)\}\notag\\
&+pq\{b_1/\theta+b_2\theta+b_1f_1/f_2+b_2f_2/f_1\}>0.
\label{eq:two-Cpositive}
\end{align}
Set $H=k_1(2p+u+q)+k_2(2p+v+q)$.
Expanding $AB-C$ gives
\begin{equation}\label{eq:two-last}
AB-C=\Sigma T+2p\Sigma(2p+\Sigma)+(q-p)H.
\end{equation}
If $q\ge p$, then $B>0$ and \eqref{eq:two-last} is strictly positive.
If $q<p$, then \eqref{eq:two-T} gives
\[
B>T+(p+q)\Sigma>0.
\]
Since $k_1<u$ and $k_2<v$,
\[
H<u(2p+u+q)+v(2p+v+q)
<\Sigma(2p+\Sigma+q).
\]
Using $q-p<0$ in \eqref{eq:two-last}, we conclude that
\[
AB-C>\Sigma\{T+(p+q)\Sigma+2p^2+pq+q^2\}>0.
\]
Thus $A,B,C,AB-C$ are all strictly positive, and the cubic
Routh--Hurwitz criterion \cite{Hurwitz1895} shows that every eigenvalue
of $J_2$ has strictly negative real part. This proves the linear
stability assertion in Theorem~\ref{thm:main}(i).

\begin{remark}
Conditional on the existence of a positive equilibrium, the above
two-patch calculation uses only $p,q,\beta_i>0$ and $0<f_i<1$.
The patchwise assumptions $\beta_i>\gamma_i>0$ also ensure existence
and uniqueness of the positive equilibrium throughout this paper.
\end{remark}

\subsection{Three Patches: Exact Positivity Certificates for All Parameters}

\subsubsection{Normalization and the Admissible Parameter Domain}

Rescale time by $\tau=pt$ and multiply the population variables by
$p/c_0$; relabel the rates $(q/p,\beta_i/p,\gamma_i/p)$ as
$(q,\beta_i,\gamma_i)$. These are positive rescalings of time and
coordinates and do not change the sign of spectral real parts. We may
therefore take $p=1$, with equilibrium satisfying $s_i+qh_i=1$.
Define
\begin{equation}\label{eq:three-x}
x_i=qh_i,\quad s_i=1-x_i,\quad
r_i=\frac{x_i}{q(1-x_i)+x_i},\quad
t_i=\beta_ir_i,\quad c_i=\frac{q(L_3x)_i}{x_i}.
\end{equation}
Here $0<x_i<1$, and
\[
\gamma_i=\frac{t_i(1-r_i)}{r_i}+c_i,\qquad
\beta_i-\gamma_i=t_i-c_i.
\]
The two strict inequalities in \eqref{eq:assumptions} are therefore
equivalent to
\begin{equation}\label{eq:legal}
t_i>\max\left\{\frac{q(L_3x)_i}{x_i},
-\frac{(L_3x)_i}{1-x_i}\right\}.
\end{equation}
The simplification of the second term uses
$r_i/(1-r_i)=x_i/[q(1-x_i)]$.
This constraint is retained throughout the exact computations below.

Now $K=K_3$. By Lemma~\ref{lem:pencil},
\begin{equation}\label{eq:three-pencil}
\begin{aligned}
P(\lambda)&=\lambda^2I_3+
\lambda\{(1+q)K+\diag(d_i)\}
+qK^2+\diag(e_i)K,\\
d_i&=c_i+t_i,\qquad
e_i=c_i+\frac{qt_i}{q(1-x_i)+x_i}.
\end{aligned}
\end{equation}
The restricted characteristic polynomial is
\begin{equation}\label{eq:quintic}
\chi(\lambda)=\frac{\det P(\lambda)}{\lambda}
=\lambda^5+a_1\lambda^4+a_2\lambda^3+a_3\lambda^2+a_4\lambda+a_5.
\end{equation}
In this normalization, \eqref{eq:zero-derivative} also gives directly
$a_5=\det(\mathcal C)\sum_iN_i^*>0$.

The Hurwitz determinants for this quintic can be written as
\begin{align}
\Delta_1&=a_1,\qquad \Delta_2=a_1a_2-a_3,\notag\\
\Delta_3&=\Delta_2a_3-a_1^2a_4+a_1a_5,\label{eq:hurwitz5}\\
\Delta_4&=a_4\Delta_3-a_5(a_2\Delta_2-a_1a_4+a_5),\qquad
\Delta_5=a_5\Delta_4.\notag
\end{align}
It is therefore enough to establish
$a_5,\Delta_1,\Delta_2,\Delta_3,\Delta_4>0$.

\subsection{Four Exhaustive Parameterizations}

After reflecting the path to interchange its endpoints, we may assume
$x_1\le x_3$. Scale by $\min_i x_i$ and write $x_i=X_i/T$.
In each of the following cases, the parameters satisfy
\begin{equation}\label{eq:parameter-domain}
q,d,z_1,z_2,z_3>0,\qquad b,c\ge0.
\end{equation}
\begin{enumerate}[label=\textbf{P\arabic*.},leftmargin=3.2em]
\item The middle component is maximal:
\[
X=(1,1+b+c,1+b),\qquad T=1+b+c+d,
\]
\[
t=\left(q(b+c)+z_1,\ \frac{b+2c+z_2}{d},\
\frac{qc+z_3}{1+b}\right).
\]
\item The middle component is minimal:
\[
X=(1+b,1,1+b+c),\qquad T=1+b+c+d,
\]
\[
t=\left(\frac{b+z_1}{c+d},\ q(2b+c)+z_2,\
\frac{b+c+z_3}{d}\right).
\]
\item The components increase monotonically and the second gap is at
least as large as the first:
\[
X=(1,1+b,1+2b+c),\qquad T=1+2b+c+d,
\]
\[
t=\left(qb+z_1,\ \frac{qc+z_2}{1+b},\
\frac{b+c+z_3}{d}\right).
\]
\item The components increase monotonically and the first gap is at
least as large as the second:
\[
X=(1,1+b+c,1+2b+c),\qquad T=1+2b+c+d,
\]
\[
t=\left(q(b+c)+z_1,\ \frac{c+z_2}{b+d},\
\frac{b+z_3}{d}\right).
\]
\end{enumerate}

\begin{lemma}\label{lem:coverage}
The four parameterizations above, together with reflection of the
endpoints, cover all parameters satisfying $q>0$, $0<x_i<1$, and
\eqref{eq:legal}.
\end{lemma}
\begin{proof}
Let $m=\min_i x_i$, and set $X_i=x_i/m$ and $T=1/m$.
If the middle component is maximal or minimal, we obtain P1 or P2,
respectively. Otherwise, after ordering the endpoints, we have
$x_1\le x_2\le x_3$, and comparison of the two successive gaps leads
to P3 or P4. More explicitly, the nonnegative parameters in the four
cases are recovered as follows:
\[
\begin{array}{c|cc}
&b&c\\\hline
\mathrm{P1}&X_3-1&X_2-X_3\\
\mathrm{P2}&X_1-1&X_3-X_1\\
\mathrm{P3}&X_2-1&X_3-2X_2+1\\
\mathrm{P4}&X_3-X_2&2X_2-X_3-1
\end{array}
\]
Moreover, $d=T-\max_iX_i=(1-\max_i x_i)/m>0$.
Computing $L_3X$ in each case shows that the active lower bound in
\eqref{eq:legal} is $q(L_3X)_i/X_i$ when $(L_3X)_i\ge0$, and
$-(L_3X)_i/(T-X_i)$ when $(L_3X)_i\le0$.
Adding a positive slack variable to each bound yields precisely the
listed values of $t_i$. Conversely, each $t_i$ satisfying the strict
lower bound uniquely determines the corresponding positive slack
variable $z_i$. The choices $b=0$ and $c=0$ are allowed, so ties and
boundary cases are included.
\end{proof}

\subsubsection{Construction of Finite Polynomial Certificates}

To specify the objects to be recomputed, write each $t_i$ in the
parameterizations above as $U_i/V_i$, where $U_i,V_i$ are the listed
integer-coefficient polynomials. Set
\[
W_i=(L_3X)_i,\qquad D_i=q(T-X_i)+X_i.
\]
For each parameterization, first construct
\begin{align*}
\widehat R_i&=X_iV_iD_i,\\
\widehat D_i&=(qW_iV_i+U_iX_i)D_i,\\
\widehat E_i&=qW_iV_iD_i+qTU_iX_i.
\end{align*}
Divide these three polynomials by their normalized greatest common
divisor and denote the results by $R_i,D_i^{\,0},E_i^{\,0}$.
Then $d_i=D_i^{\,0}/R_i$ and $e_i=E_i^{\,0}/R_i$.
The corresponding row-wise denominator-cleared matrix is
\begin{equation}\label{eq:polynomial-matrix}
\widetilde P_{ij}(\lambda)=
qR_i(K^2)_{ij}+E_i^{\,0}K_{ij}
+\lambda\{(q+1)R_iK_{ij}+D_i^{\,0}\delta_{ij}\}
+\lambda^2R_i\delta_{ij}.
\end{equation}
Expanding the $3\times3$ determinant over its six permutations, define
\[
\det\widetilde P(\lambda)=
C_0\lambda^6+C_1\lambda^5+\cdots+C_5\lambda.
\]
The constant term vanishes identically, and $C_0=R_1R_2R_3$.
The coefficients $a_i=C_i/C_0$, after cancellation, are exactly those
of \eqref{eq:quintic}. All rational expressions are computed over the
ring of integer polynomials, and greatest common divisors are normalized
to have a positive leading coefficient in a fixed monomial order.
The positivity of the final denominators is also checked separately
by the coefficient certificates below; thus the sign conclusions do
not depend on any unverified sign convention during cancellation.

\begin{lemma}\label{lem:positive-polynomial}
Let $Q$ be a polynomial in $(q,b,c,d,z_1,z_2,z_3)$ with nonnegative
coefficients. If $Q(1,0,0,1,1,1,1)>0$, then $Q$ is strictly positive
throughout the domain \eqref{eq:parameter-domain}.
\end{lemma}
\begin{proof}
Positivity at this point shows that at least one monomial with a
positive coefficient contains neither $b$ nor $c$. The other five
variables are strictly positive in \eqref{eq:parameter-domain}, so this
monomial is strictly positive everywhere in the domain; all remaining
monomials are nonnegative.
\end{proof}

\begin{proposition}[Exact computer-assisted certificate]\label{prop:certificate}
For all parameters in P1--P4, the quantities
$a_5,\Delta_1,\Delta_2,\Delta_3,\Delta_4$ are strictly positive.
\end{proposition}
\begin{proof}
Appendix~\ref{app:code3} provides the complete exact verification
program. It computes the $C_i$ from \eqref{eq:polynomial-matrix} and
then evaluates the relevant quantities using \eqref{eq:hurwitz5}.
The numbers of nonzero terms in the reduced numerators are given in
Table~\ref{tab:counts}. With the exception of $\Delta_4$ in P4,
all listed numerators and denominators have only nonnegative coefficients
and are strictly positive at
$(q,b,c,d,z_1,z_2,z_3)=(1,0,0,1,1,1,1)$.

\begin{table}[htbp]\centering
\small
\begin{tabular}{c r r r r r}
\toprule
Region&$a_5$&$\Delta_1$&$\Delta_2$&$\Delta_3$&Positivity numerator for $\Delta_4$\\
\midrule
P1&1,230&35&6,204&143,282&1,231,162\\
P2&1,559&64&10,560&233,895&1,946,588\\
P3&1,194&34&6,300&145,406&1,239,473\\
P4&1,766&71&11,863&261,713&2,159,003\\
\bottomrule
\end{tabular}
\caption{Exact positivity checks. For P1--P3, the last column is the
reduced numerator of $\Delta_4$; for P4, it is the remainder $R$ defined
below. Every nonzero coefficient in the listed polynomials is positive.}
\label{tab:counts}
\end{table}

We treat P4 in detail. Suppose the verifier reduces $\Delta_4=N/D$.
Set $u=2cq+z_1$ and
\begin{align}
F_A&=b^8c^{11}d^4q^{12}u^3(22cq+7z_1),\notag\\
F_B&=b^6c^{11}d^3q^{11}(qd+c)u^4,\label{eq:FABC}\\
F_C&=b^5c^{11}d^4q^{12}z_3u^4.\notag
\end{align}
Define the integer polynomial
\begin{equation}\label{eq:three-squares}
R=16N-F_A(16d-b)^2-16F_B(cd-b^2)^2-8F_C(cd-2b^2)^2.
\end{equation}
The complete exact expansion verifies that all $2,159,003$
nonzero coefficients of $R$ are positive, as are all $9,909$
nonzero coefficients of $D$. At the point specified above,
\[
R=12\,386\,304\,000,\qquad D=512.
\]
Lemma~\ref{lem:positive-polynomial} therefore gives $R,D>0$.
Since $F_A,F_B,F_C\ge0$, it follows that
\begin{equation}\label{eq:D4positive}
\Delta_4=
\frac{R+F_A(16d-b)^2+16F_B(cd-b^2)^2+8F_C(cd-2b^2)^2}{16D}
>0.
\end{equation}
This identity follows directly from the definition
\eqref{eq:three-squares} and its exact expansion; no floating-point
approximation is involved.

At this point, all four parameterizations give the same characteristic
polynomial,
\[
(\lambda+1)^2(\lambda+2)(\lambda+3)(\lambda+4),
\]
and
\[
(a_5,\Delta_1,\Delta_2,\Delta_3,\Delta_4)
=(24,11,410,26160,1512000).
\]
Together with Lemma~\ref{lem:positive-polynomial}, the strict
positivity therefore includes all boundary cases with $b=0$ or
$c=0$. The certificates for all four regions are established.
\end{proof}

By Lemma~\ref{lem:coverage}, Proposition~\ref{prop:certificate}, and
the quintic Routh--Hurwitz criterion, all eigenvalues of
$J|_{\mathcal T}$ have strictly negative real parts. This proves the
linear stability assertion in Theorem~\ref{thm:main}(ii).

\subsubsection{From Linear Stability to Nonlinear Exponential Stability}

The argument applies to both the two- and three-patch cases. Choose
linear coordinates on $\mathcal H_M^+$ near the positive equilibrium
and write the equation as $\dot z=Az+\mathcal R(z)$, where
$A=J|_{\mathcal T}$ and $\mathcal R(z)=O(|z|^2)$.
Since the spectrum of $A$ lies strictly in the left half-plane,
there are $C_0\ge1$ and $\eta>0$ such that
$\|e^{At}\|\le C_0e^{-2\eta t}$. After reducing the neighborhood,
we may assume $|\mathcal R(z)|\le(\eta/C_0)|z|$.
Variation of constants yields
\[
e^{2\eta t}|z(t)|
\le C_0|z(0)|+\eta\int_0^t e^{2\eta s}|z(s)|\,ds.
\]
By Gronwall's inequality,
$|z(t)|\le C_0e^{-\eta t}|z(0)|$.
For sufficiently small initial data, this estimate ensures that the
trajectory remains in the chosen neighborhood, proving
\eqref{eq:exp}. As the positive equilibrium is in the interior of
the positive phase space, the neighborhood may be taken entirely
inside that space. This completes the proof of
Theorem~\ref{thm:main}(i)--(ii).

\subsection{Four Patches: An Exact Rational Counterexample to Stability}

\subsubsection{Parameters and Direct Verification of the Positive Equilibrium}

Take
\begin{equation}\label{eq:four-L}
L_4=\begin{pmatrix}
-1&1&0&0\\1&-2&1&0\\0&1&-2&1\\0&0&1&-1
\end{pmatrix},\qquad p=1,\quad q=20000.
\end{equation}
The transmission and recovery rates are
\begin{equation}\label{eq:four-rates}
\begin{aligned}
\beta&=\left(3781803890001,\frac{3785211}{193},
\frac{83996}{111},\frac{410410479}{499}\right),\\
\gamma&=\left(3781803890000,\frac{96000}{193},
\frac{20000}{999},\frac{400430000}{499}\right).
\end{aligned}
\end{equation}
Every $\gamma_i$ is positive, and
\begin{equation}\label{eq:four-rate-gap}
\beta-\gamma=
\left(1,\frac{3689211}{193},\frac{735964}{999},
\frac{9980479}{499}\right)>0.
\end{equation}
Let $D_0=20562972$ and set
\begin{equation}\label{eq:four-equilibrium}
s=\frac{(19800000,700000,20000,40000)}{D_0},\qquad
h=\frac{(10,965,999,998)}{D_0}.
\end{equation}
Clearly $s,h>0$, and the sum of their numerators is $D_0$, so the
total population is $1$. Moreover,
\[
ps+qh=\frac{20000000}{D_0}\one.
\]
Substituting \eqref{eq:four-rates} and \eqref{eq:four-equilibrium}
into the original reaction term and combining fractions gives,
component by component,
\begin{equation}\label{eq:four-F}
D_0F(s,h)=(-19100000,18420000,700000,-20000).
\end{equation}
Direct computation of the dispersal terms also yields
$D_0pL_4s=D_0F(s,h)$ and
$D_0qL_4h=-D_0F(s,h)$.
Thus \eqref{eq:four-equilibrium} is an exact positive equilibrium of
the original model, not merely an equilibrium of an approximate or
reduced system.

\subsubsection{The Fixed-Mass Jacobian and a Routh Certificate}

Order the original variables as
$(S_1,\ldots,S_4,I_1,\ldots,I_4)$ and construct the $8\times8$
Jacobian $J$ directly from \eqref{eq:full-J}. Set
\begin{equation}\label{eq:four-embedding}
T_0=\begin{pmatrix}I_7\\-\one_7^\top\end{pmatrix},\qquad
A=(JT_0)_{1:7,:}.
\end{equation}
The matrix $T_0$ identifies $\R^7$ with the fixed-mass tangent space
$\mathcal T$. Since $\one_8^\top J=0$, we have $JT_0=T_0A$;
hence $A$ is the genuine linearization of the original system on this
tangent space. Lemma~\ref{lem:zero} also gives
\[
\det(\lambda I_8-J)=\lambda\det(\lambda I_7-A).
\]
Write
\begin{equation}\label{eq:seven-poly}
\det(\lambda I_7-A)=\lambda^7+a_1\lambda^6+\cdots+a_6\lambda+a_7.
\end{equation}
All $a_i$ are exact rational numbers; their values are listed in
Appendix~\ref{app:coeff4}.

To avoid relying on floating-point eigenvalues to establish instability,
we construct the Routh table directly. Set $a_0=1$ and take its first
two rows to be
\[
\mathcal R_{0,\cdot}=(a_0,a_2,a_4,a_6),\qquad
\mathcal R_{1,\cdot}=(a_1,a_3,a_5,a_7).
\]
Pad the remaining rows with zeros on the right and use the recursion
\begin{equation}\label{eq:routh-recurrence}
\mathcal R_{i,j}=
\frac{\mathcal R_{i-1,0}\mathcal R_{i-2,j+1}
-\mathcal R_{i-2,0}\mathcal R_{i-1,j+1}}
{\mathcal R_{i-1,0}}
\quad (i=2,\dots,7;\ j=0,1,2).
\end{equation}
Table~\ref{tab:routh} gives strict rational bounds on every entry in
the first column. The verification program in Appendix~\ref{app:code4}
reconstructs $J$ and $A$ from the original parameters and checks each
bound exactly.

\begin{table}[htbp]\centering
\begin{tabular}{c c c}
\toprule
Power & Strict enclosure of the first-column entry & Sign\\
\midrule
$\lambda^7$&$1$&$+$\\
$\lambda^6$&$(3\cdot10^6,\ 4\cdot10^6)$&$+$\\
$\lambda^5$&$(3\cdot10^{11},\ 4\cdot10^{11})$&$+$\\
$\lambda^4$&$(9\cdot10^{15},\ 10^{16})$&$+$\\
$\lambda^3$&$(3\cdot10^{19},\ 4\cdot10^{19})$&$+$\\
$\lambda^2$&$(3\cdot10^{20},\ 4\cdot10^{20})$&$+$\\
$\lambda^1$&$(-2\cdot10^{21},\ -10^{21})$&$-$\\
$\lambda^0$&$(4\cdot10^{22},\ 5\cdot10^{22})$&$+$\\
\bottomrule
\end{tabular}
\caption{Exact Routh certificate for the four-patch counterexample.
All interval endpoints and comparisons are exact integers or rationals.}
\label{tab:routh}
\end{table}

All pivots are nonzero, and the first column exhibits exactly two
sign changes. The degree-seven polynomial \eqref{eq:seven-poly} therefore
has exactly two roots in the open right half-plane, counted with
multiplicity. The original vector field is smooth near the positive
equilibrium. By the principle of linearized instability for ordinary
differential equations, the equilibrium is nonlinearly locally unstable
within the fixed-mass manifold. For numerical illustration, these two
roots are approximately
\begin{equation}\label{eq:unstable-pair}
\lambda_\pm=0.66830587233517
\pm 10.48314099384637\,\mathrm i.
\end{equation}
The proof itself uses only the exact rational sign computations in
Table~\ref{tab:routh}, and not the approximations in
\eqref{eq:unstable-pair}. This proves Theorem~\ref{thm:main}(iii).

\subsection{Scope of the Results and Minimal Number of Patches}

\begin{corollary}\label{cor:minimum}
Within the class of standard-incidence SIS models with a common
equal-edge-weight path dispersal network, $p,q>0$, and
$\beta_i>\gamma_i>0$ in every patch, four is the smallest number of
patches for which a positive equilibrium can be locally unstable.
\end{corollary}
\begin{proof}
For a single patch with total population $M$, the infected component
satisfies
$\dot I=(\beta-\gamma)I-\beta I^2/M$.
At its positive equilibrium, the derivative is
$-(\beta-\gamma)<0$. Theorem~\ref{thm:main}(i)--(ii) establishes
stability for two and three patches, and part (iii) supplies the
four-patch counterexample.
\end{proof}

This minimality statement does not concern other dispersal graphs or
other incidence functions. Moreover, local stability for two and
three patches does not imply that every positive trajectory converges
to the equilibrium. Since the restricted linearization is strictly
stable throughout the admissible parameter domain, a Hopf bifurcation
cannot occur at these positive equilibria; however, local spectral
information alone does not rule out periodic orbits elsewhere in the
phase space. The four-patch counterexample shows that the patchwise
conditions $\beta_i>\gamma_i>0$ alone do not ensure local stability
of positive equilibria for all chain-patch models. Transversality
conditions and nonlinear bifurcation coefficients have not been
verified for this example, so we do not claim to have proved the
existence of periodic solutions.

\subsection{Complete Verification Program for the Three-Patch Certificates}\label{app:code3}

The following program uses only integer-polynomial and rational-function
arithmetic. It requires \texttt{python-flint}, version \texttt{0.9.0}
in the present verification. Run the program separately with
\texttt{peak}, \texttt{valley}, \texttt{mono\_up}, and
\texttt{mono\_down}, corresponding to P1--P4. It checks
nonnegativity of polynomial coefficients and strict positivity at the
specified boundary point; the certificate for a region is rejected
if any assertion fails. Running Python with assertions disabled
(\texttt{-O} or \texttt{-OO}) is explicitly disallowed.
For example:
\begin{verbatim}
python verify_sis.py peak
python verify_sis.py valley
python verify_sis.py mono_up
python verify_sis.py mono_down
\end{verbatim}
\begin{lstlisting}
#!/usr/bin/env python3
"""Exact certificate for local stability of the three-patch SIS chain.

Dependency: python-flint (audited with version 0.9.0).
Usage: python verify_sis.py peak|valley|mono_up|mono_down
No floating-point arithmetic or numerical eigenvalues are used.
Variables: q,d,z1,z2,z3 > 0 and b,c >= 0; p is scaled to 1.
The companion paper explains the formulas and four-region coverage.
"""

from flint import fmpz_mpoly_ctx
from time import time
import sys, json
from pathlib import Path

if not __debug__:
    raise RuntimeError("Run without -O/-OO: checks require assertions.")

case = sys.argv[1] if len(sys.argv) > 1 else 'peak'
ctx = fmpz_mpoly_ctx.get(['q', 'b', 'c', 'd', 'z1', 'z2', 'z3'])
q, b, c, d, z1, z2, z3 = ctx.gens()
Z = ctx.constant(0)
O = ctx.constant(1)

# Each region gives x_i = X_i/T and t_i = TN_i/TD_i.
if case == 'peak':
    X = [O, 1+b+c, 1+b]
    T = 1+b+c+d
    TN = [q*(b+c)+z1, b+2*c+z2, q*c+z3]
    TD = [O, d, 1+b]
elif case == 'valley':
    X = [1+b, O, 1+b+c]
    T = 1+b+c+d
    TN = [b+z1, q*(2*b+c)+z2, b+c+z3]
    TD = [c+d, O, d]
elif case == 'mono_up':
    X = [O, 1+b, 1+2*b+c]
    T = 1+2*b+c+d
    TN = [q*b+z1, q*c+z2, b+c+z3]
    TD = [O, 1+b, d]
elif case == 'mono_down':
    X = [O, 1+b+c, 1+2*b+c]
    T = 1+2*b+c+d
    TN = [q*(b+c)+z1, c+z2, b+z3]
    TD = [O, b+d, d]
else:
    raise ValueError(case)

K = [[1, -1, 0], [-1, 2, -1], [0, -1, 1]]
K2 = [[2, -3, 1], [-3, 6, -3], [1, -3, 2]]
U = [-sum(K[i][j]*X[j] for j in range(3)) for i in range(3)]
D = [q*(T-x)+x for x in X]
R = []
DN = []
EN = []

# Clear denominators row by row, retaining the audited gcd order.
for i in range(3):
    rr = X[i]*TD[i]*D[i]
    dd = (q*U[i]*TD[i]+TN[i]*X[i])*D[i]
    ee = q*U[i]*TD[i]*D[i]+q*T*TN[i]*X[i]
    g = rr.gcd(dd).gcd(ee)
    R.append(rr//g)
    DN.append(dd//g)
    EN.append(ee//g)

# Matrix entries are lists of coefficients in increasing powers of lambda.
mat = [
    [
        [
            q*R[i]*K2[i][j]+EN[i]*K[i][j],
            (q+1)*R[i]*K[i][j]+(DN[i] if i == j else Z),
            R[i] if i == j else Z,
        ]
        for j in range(3)
    ]
    for i in range(3)
]


def mul(a, b):
    """Multiply two polynomials in lambda by coefficient convolution."""
    o = [Z]*(len(a)+len(b)-1)
    for i, x in enumerate(a):
        for j, y in enumerate(b):
            o[i+j] = o[i+j]+x*y
    return o


# Expand the 3 by 3 determinant in the original permutation order.
f = [Z]*7
for js, sgn in [
    ([0, 1, 2], 1), ([1, 2, 0], 1), ([2, 0, 1], 1),
    ([2, 1, 0], -1), ([1, 0, 2], -1), ([0, 2, 1], -1),
]:
    zz = mul(mul(mat[0][js[0]], mat[1][js[1]]), mat[2][js[2]])
    f = [a+sgn*b for a, b in zip(f, zz)]

assert f[0].is_zero()
C0, C1, C2, C3, C4, C5 = f[6:0:-1]
assert C0 == R[0]*R[1]*R[2]
print(
    case, 'coeffcounts',
    [len(v.coeffs()) for v in [C0, C1, C2, C3, C4, C5]],
    flush=True,
)


# Rational expressions are stored as (numerator, denominator) pairs.
def red(n, d):
    g = n.gcd(d)
    return n//g, d//g


def mulf(a, b):
    n1, d1 = a
    n2, d2 = b
    g = n1.gcd(d2)
    h = n2.gcd(d1)
    return (n1//g)*(n2//h), (d1//h)*(d2//g)


def subf(a, b):
    n1, d1 = a
    n2, d2 = b
    g = d1.gcd(d2)
    return red(n1*(d2//g)-n2*(d1//g), (d1//g)*d2)


def addf(a, b):
    return subf(a, (-b[0], b[1]))


a1, a2, a3, a4, a5 = [red(C, C0) for C in [C1, C2, C3, C4, C5]]
D2 = subf(mulf(a1, a2), a3)
D3 = addf(
    subf(mulf(D2, a3), mulf(mulf(a1, a1), a4)),
    mulf(a1, a5),
)
summary = {}


def inspect(name, rat):
    """Check coefficient signs and strict positivity at a boundary point."""
    num, den = rat
    vv = num.coeffs()
    ww = den.coeffs()
    subs = {'q': 1, 'b': 0, 'c': 0, 'd': 1, 'z1': 1, 'z2': 1, 'z3': 1}
    nn = num.subs(subs)
    dd = den.subs(subs)
    info = {
        'num_terms': len(vv),
        'num_neg': sum(v < 0 for v in vv),
        'den_terms': len(ww),
        'den_neg': sum(v < 0 for v in ww),
        'num_corner': str(nn),
        'den_corner': str(dd),
    }
    assert (
        min(vv) > 0 and min(ww) > 0
        and int(str(nn)) > 0 and int(str(dd)) > 0
    ), (name, info)
    summary[name] = info
    print(name, info, flush=True)


for name, rat in [('a5', a5), ('D1', a1), ('D2', D2), ('D3', D3)]:
    inspect(name, rat)

H = addf(subf(mulf(a2, D2), mulf(a1, a4)), a5)
print('delta4 start', flush=True)
ti = time()
D4 = subf(mulf(a4, D3), mulf(a5, H))
print('delta4secs', time()-ti, flush=True)

if case == 'mono_down':
    original, denominator = D4
    U = 2*c*q+z1
    FA = b**8*c**11*d**4*q**12*U**3*(22*c*q+7*z1)
    FB = b**6*c**11*d**3*q**11*(q*d+c)*U**4
    FC = b**5*c**11*d**4*q**12*z3*U**4
    square_sum = (
        FA*(16*d-b)**2 + 16*FB*(c*d-b*b)**2
        + 8*FC*(c*d-2*b*b)**2
    )
    residual = 16*original-square_sum
    # Every factor multiplying a square is nonnegative on the whole domain.
    # Thus Delta4 >= residual / (16*denominator) > 0.
    for factor in [FA, FB, FC]:
        assert all(v > 0 for v in factor.coeffs())
    inspect('D4_positive_lower_bound', (residual, 16*denominator))
    summary['D4_square_count'] = 3
    summary['D4_original_terms'] = len(original.coeffs())
    summary['D4_original_negative_coefficients'] = sum(
        v < 0 for v in original.coeffs()
    )
    summary['D4_original_denominator_terms'] = len(denominator.coeffs())
else:
    inspect('D4', D4)

Path('results_'+case+'.json').write_text(json.dumps(summary, indent=2))
\end{lstlisting}

\subsection{Exact Coefficients of the Four-Patch Characteristic Polynomial of Degree Seven}\label{app:coeff4}

The following coefficients are listed in the order used in
\eqref{eq:seven-poly}. They are computed directly from the
original eight-dimensional Jacobian restricted to the fixed-mass
tangent space.
\begin{align*}
a_1&=\frac{379101134939063}{96210693}\\[5pt]
a_2&=\frac{193766864332063468067756905110304087}{500325970604190369701337}\\[5pt]
a_3&=\frac{520687638879246021078634334760985433193367}{55536182737065131036848407}\\[5pt]
a_4&=\frac{1612103833169558194724173560965884227493629293}{47308600109351778290648643}\\[5pt]
a_5&=\frac{20727403672028117211993541652904365306718074000}{55536182737065131036848407}\\[5pt]
a_6&=\frac{123981990079279790905013081846746303759300000000}{38707036453106000419621617}\\[5pt]
a_7&=\frac{5358076291508776112148992615827345600000000000}{116237346706024025284149}.
\end{align*}

\subsection{Complete Verification Program for the Four-Patch Counterexample}\label{app:code4}

The following complete program requires only \texttt{SymPy} and performs
all comparisons over the rational numbers. It directly verifies the
original model, the invariance of the fixed-mass tangent space, the
characteristic polynomial, and Table~\ref{tab:routh}, without calling
any numerical eigenvalue solver.
\begin{lstlisting}
#!/usr/bin/env python3
"""Exact verification of an unstable four-patch SIS endemic equilibrium.

Dependency: SymPy.  All quantities used in the proof are integers or rational
numbers; the script does not compute floating-point roots.

Run:
    python verification_n4.py
    python verification_n4.py --json /path/to/n4_certificate.json

The model is S' = p L S - F, I' = q L I + F, where
F_i = beta_i S_i I_i / (S_i + I_i) - gamma_i I_i.
L is the symmetric four-vertex path generator, with negative diagonal.
The seven-dimensional restriction uses the first seven coordinates of
(S_1,...,S_4,I_1,...,I_4); the last perturbation is minus their sum.
"""

from __future__ import annotations

import argparse
import json
from pathlib import Path

import sympy as sp


def require(condition: bool, message: str) -> None:
    """Keep verification active even when Python is invoked with -O."""
    if not condition:
        raise AssertionError(message)


def rational_string(value: sp.Expr) -> str:
    value = sp.cancel(value)
    require(value.is_Rational is True, "A certificate value is not rational.")
    numerator, denominator = value.as_numer_denom()
    if denominator == 1:
        return str(numerator)
    return f"{numerator}/{denominator}"


def matrix_strings(matrix: sp.Matrix) -> list[list[str]]:
    return [
        [rational_string(matrix[i, j]) for j in range(matrix.cols)]
        for i in range(matrix.rows)
    ]


def vector_strings(vector: sp.Matrix) -> list[str]:
    return [rational_string(value) for value in vector]


def routh_table(coefficients: list[sp.Expr]) -> list[list[sp.Expr]]:
    """Build the ordinary Routh table; reject exceptional zero cases."""
    degree = len(coefficients) - 1
    width = (degree + 2) // 2
    rows = [[sp.Integer(0) for _ in range(width)] for _ in range(degree + 1)]
    for j, value in enumerate(coefficients[0::2]):
        rows[0][j] = value
    for j, value in enumerate(coefficients[1::2]):
        rows[1][j] = value
    for i in range(2, degree + 1):
        require(rows[i - 1][0] != 0, "A Routh pivot vanishes.")
        for j in range(width - 1):
            rows[i][j] = sp.cancel(
                (
                    rows[i - 1][0] * rows[i - 2][j + 1]
                    - rows[i - 2][0] * rows[i - 1][j + 1]
                )
                / rows[i - 1][0]
            )
        require(any(value != 0 for value in rows[i]), "A Routh row vanishes.")
    require(all(row[0] != 0 for row in rows), "A Routh first entry vanishes.")
    return rows


def verify() -> dict:
    rational = sp.Rational
    p = sp.Integer(1)
    q = sp.Integer(20000)
    common_denominator = sp.Integer(20562972)
    S = sp.Matrix([19800000, 700000, 20000, 40000]) / common_denominator
    I = sp.Matrix([10, 965, 999, 998]) / common_denominator
    beta = sp.Matrix(
        [3781803890001, rational(3785211, 193), rational(83996, 111),
         rational(410410479, 499)]
    )
    gamma = sp.Matrix(
        [3781803890000, rational(96000, 193), rational(20000, 999),
         rational(400430000, 499)]
    )
    L = sp.Matrix(
        [[-1, 1, 0, 0], [1, -2, 1, 0],
         [0, 1, -2, 1], [0, 0, 1, -1]]
    )

    require(p > 0 and q > 0, "Dispersal rates must be positive.")
    require(all(value > 0 for value in S), "S is not strictly positive.")
    require(all(value > 0 for value in I), "I is not strictly positive.")
    require(sum(S) + sum(I) == 1, "The total mass is not one.")
    require(all(gamma[i] > 0 for i in range(4)), "Some gamma_i is not positive.")
    require(all(beta[i] > gamma[i] for i in range(4)), "Some beta_i <= gamma_i.")
    require(L == L.T and L * sp.ones(4, 1) == sp.zeros(4, 1),
            "The diffusion generator is incorrect.")

    reaction = sp.Matrix(
        [beta[i] * S[i] * I[i] / (S[i] + I[i]) - gamma[i] * I[i]
         for i in range(4)]
    )
    residual_S = p * L * S - reaction
    residual_I = q * L * I + reaction
    require(residual_S == sp.zeros(4, 1), "The susceptible equilibrium equation fails.")
    require(residual_I == sp.zeros(4, 1), "The infected equilibrium equation fails.")

    # Differentiate the standard-incidence reaction directly at (S, I).
    reaction_dS = sp.diag(
        *[beta[i] * (I[i] / (S[i] + I[i])) ** 2 for i in range(4)]
    )
    reaction_dI = sp.diag(
        *[beta[i] * (S[i] / (S[i] + I[i])) ** 2 - gamma[i] for i in range(4)]
    )
    J_full = sp.Matrix.vstack(
        sp.Matrix.hstack(p * L - reaction_dS, -reaction_dI),
        sp.Matrix.hstack(reaction_dS, q * L + reaction_dI),
    )
    embedding = sp.Matrix.vstack(sp.eye(7), -sp.ones(1, 7))
    J_mass = (J_full * embedding)[:7, :]
    require(sp.ones(1, 8) * J_full == sp.zeros(1, 8),
            "The Jacobian does not conserve total mass.")
    require(J_full * embedding == embedding * J_mass,
            "The fixed-mass coordinate restriction is incorrect.")
    require(J_full * S.col_join(I) == sp.zeros(8, 1),
            "The homogeneity null vector is incorrect.")

    lam = sp.Symbol("lambda")
    polynomial = J_mass.charpoly(lam).as_poly()
    coefficients = polynomial.all_coeffs()
    require(polynomial.degree() == 7 and coefficients[0] == 1,
            "The restricted characteristic polynomial is not monic of degree seven.")
    require(
        sp.expand(J_full.charpoly(lam).as_expr() - lam * polynomial.as_expr()) == 0,
        "The full and restricted characteristic polynomials disagree.",
    )
    require(coefficients[-1] != 0, "The full zero eigenvalue is not simple.")
    primitive_polynomial = polynomial.clear_denoms(convert=True)[1].primitive()[1]
    require(primitive_polynomial.LC() > 0, "The primitive leading coefficient is negative.")

    table = routh_table(coefficients)
    first_column = [row[0] for row in table]
    signs = [int(sp.sign(value)) for value in first_column]
    require(signs == [1, 1, 1, 1, 1, 1, -1, 1], "Unexpected Routh sign pattern.")
    sign_changes = sum(a != b for a, b in zip(signs, signs[1:]))
    require(sign_changes == 2, "There are not exactly two Routh sign changes.")
    bounds = [
        (3*10**6, 4*10**6), (3*10**11, 4*10**11),
        (9*10**15, 10**16), (3*10**19, 4*10**19),
        (3*10**20, 4*10**20), (-2*10**21, -10**21),
        (4*10**22, 5*10**22),
    ]
    require(first_column[0] == 1, "The first Routh entry is not one.")
    for value, (lower, upper) in zip(first_column[1:], bounds):
        require(lower < value < upper, "A paper Routh bound failed.")
    enclosures = []
    for power, value in zip(range(7, -1, -1), first_column):
        lower, upper = int(sp.floor(value)), int(sp.ceiling(value))
        require(lower <= value <= upper, "An exact integer enclosure failed.")
        if lower != upper:
            require(lower < value < upper, "A strict integer enclosure failed.")
        enclosures.append(
            {"row_power": power, "exact": rational_string(value),
             "floor": str(lower), "ceil": str(upper), "sign": int(sp.sign(value))}
        )

    return {
        "certificate": "Four-patch standard-incidence SIS unstable endemic equilibrium",
        "arithmetic": "Exact integers and rationals; no floating-point roots are used",
        "sympy_version": sp.__version__,
        "model": {
            "equations": "S'=p*L*S-F; I'=q*L*I+F; F_i=beta_i*S_i*I_i/(S_i+I_i)-gamma_i*I_i",
            "p": rational_string(p), "q": rational_string(q),
            "L": matrix_strings(L),
            "beta": vector_strings(beta), "gamma": vector_strings(gamma),
            "beta_minus_gamma": vector_strings(beta - gamma),
        },
        "equilibrium": {
            "S": vector_strings(S), "I": vector_strings(I), "total_mass": "1",
            "reaction": vector_strings(reaction),
            "residual_S": vector_strings(residual_S),
            "residual_I": vector_strings(residual_I),
            "pS_plus_qI": vector_strings(p * S + q * I),
        },
        "characteristic_polynomial": {
            "variable": "lambda", "degree": 7,
            "coefficient_order": "descending powers, lambda^7 through lambda^0",
            "monic_rational_coefficients": [rational_string(v) for v in coefficients],
            "primitive_integer_coefficients":
                [str(v) for v in primitive_polynomial.all_coeffs()],
            "full_jacobian_identity": "det(lambda*I_8-J_full)=lambda*det(lambda*I_7-J_mass)",
        },
        "routh": {
            "table": [[rational_string(v) for v in row] for row in table],
            "first_column": enclosures,
            "signs": signs, "sign_changes": sign_changes,
            "zero_pivots": False, "zero_rows": False,
            "right_half_plane_root_count": 2,
        },
        "conclusion": "The endemic equilibrium is linearly unstable on the fixed-total-mass manifold.",
        "verified": True,
    }


def main() -> None:
    parser = argparse.ArgumentParser(description=__doc__)
    parser.add_argument(
        "--json", type=Path,
        default=Path(__file__).resolve().with_name("n4_certificate.json"),
        help="Path for the exact JSON certificate (default: next to this script).",
    )
    args = parser.parse_args()
    certificate = verify()
    print("Model, positivity, mass, equilibrium, and Jacobian assertions: PASSED")
    print("Monic characteristic polynomial coefficients (lambda^7 to lambda^0):")
    for power, value in zip(
        range(7, -1, -1),
        certificate["characteristic_polynomial"]["monic_rational_coefficients"],
    ):
        print(f"  lambda^{power}: {value}")
    print("Routh first-column exact integer enclosures:")
    for entry in certificate["routh"]["first_column"]:
        if entry["floor"] == entry["ceil"]:
            bound = f"= {entry['floor']}"
        else:
            bound = f"in ({entry['floor']}, {entry['ceil']})"
        print(f"  row lambda^{entry['row_power']}: {bound}; sign {entry['sign']:+d}")
    print("Exact Routh sign changes: 2; right-half-plane eigenvalues: 2")
    print(certificate["conclusion"])
    args.json.parent.mkdir(parents=True, exist_ok=True)
args.json.write_text(json.dumps(certificate, indent=2) + "\n", encoding="utf-8")
    print(f"Certificate written to {args.json}")


if __name__ == "__main__":
    main()
\end{lstlisting}

\section*{Declaration of AI use}

The authors used ChatGPT  to assist in
developing the construction of an unstable endemic
equilibrium under the pointwise condition
$\beta>\gamma>0$ in
Section~\ref{seed:construction},
and in developing the analytical arguments,
proofs, and computational verification code
for the patch-model results in
Section~\ref{sec:patch} and Appendix~\ref{B}.
The authors take full responsibility for all
mathematical results and the final content
of the manuscript.

\end{document}